\documentclass[letterpaper,11pt]{article}
\usepackage[margin=1in]{geometry}
\usepackage[bookmarks, colorlinks=true, plainpages = false, colorlinks=true,
citecolor=red,
linkcolor=blue,
anchorcolor=red,
urlcolor=blue]{hyperref}
\usepackage{url}
\usepackage{amsmath,amsfonts,amsthm,amssymb,bm, verbatim,dsfont,mathtools}
\usepackage{color,graphicx,appendix}
\usepackage{etoolbox}

\usepackage{array}
\usepackage{multirow}

\usepackage{float}
\usepackage{tikz}
\usetikzlibrary{matrix,arrows,calc,shapes,backgrounds}
\usetikzlibrary{shapes.callouts,decorations.text}

\usetikzlibrary{shapes.misc}

\tikzset{cross/.style={cross out, draw=black, minimum size=2*(#1-\pgflinewidth), inner sep=0pt, outer sep=0pt},
    cross/.default={1pt}}

\tikzstyle{int}=[draw, fill=blue!20, minimum size=2em]
\tikzstyle{dot}=[circle, draw, fill=blue!20, minimum size=2em]
\tikzstyle{dotred}=[circle, draw, fill=red!20, minimum size=2em]
\tikzstyle{init} = [pin edge={to-,thin,black}]
\tikzstyle{initred} = [pin edge={to-,thin,red}]
\tikzstyle{plan}=[draw, fill=blue!20, minimum size=2em, text width=5em, rounded corners,align=center]
\tikzstyle{planwide}=[draw, fill=blue!20, minimum size=2em, text width=8em, rounded corners,align=center]
\tikzstyle{nodedot}=[circle, draw, fill=white, minimum size=0.3cm,inner sep=0pt]
\tikzstyle{vertexdot}=[circle, draw, fill=white, minimum size=3,inner sep=0pt]
\tikzstyle{Medge}=[green!60!black, thick]
\tikzstyle{Bedge}=[red, thick]
\tikzstyle{Cedge}=[blue, thick]
\tikzstyle{Sedge}=[black, thick]
\tikzstyle{Mgiantedge}=[green!60!black, line width=3.0pt]
\tikzstyle{Bgiantedge}=[red,line width=3.0pt]
\tikzstyle{Cgiantedge}=[blue,line width=3.0pt]
\tikzstyle{Sgiantedge}=[black,line width=3.0pt]
\tikzstyle{shadedgiantnode}=[circle, draw, fill=black!10, minimum size=1cm, inner sep=0pt]
\tikzstyle{unshadedgiantnode}=[circle, draw, fill=white, minimum size=1cm, inner sep=0pt]
\makeatletter
\tikzset{my loop/.style =  {to path={
            \pgfextra{}
            [looseness=5,min distance=10mm]
            \tikz@to@curve@path},font=\sffamily\small
}}  
\makeatletter 
\newcolumntype{C}[1]{>{\centering\arraybackslash}p{#1}}

\usepackage{xr,xspace}
\usepackage{todonotes}
\usepackage{paralist}
\usepackage{enumitem}

\usepackage{caption,subcaption,soul}
\usepackage{algorithm}
\usepackage{algpseudocode}
\makeatletter
\usepackage{romanbar}

\theoremstyle{plain}
\newtheorem{theorem}{Theorem}
\newtheorem{lemma}{Lemma}
\newtheorem{proposition}{Proposition}

\theoremstyle{definition}
\newtheorem{definition}{Definition}

\newtheorem{problem}{Problem}

\newtheorem{remark}{Remark}
\newtheorem{claim}{Claim}
\newtheorem*{remark*}{Remark}

\newcommand{\floor}[1]{\left\lfloor #1 \right\rfloor}

\newcommand{\argmax}{\arg\max}

\usepackage{xspace,prettyref}

\newcommand{\id}{\mathrm{id}}

\newcommand{\termI}{\text{(I)}}
\newcommand{\termII}{\text{(II)}}

\newcommand{\sigmae}{\sigma^{\sf E}} 

\newcommand{\diverge}{\to\infty}

\newcommand{\iiddistr}{{\stackrel{\text{\iid}}{\sim}}}

\newcommand{\reals}{{\mathbb{R}}}

\newcommand{\naturals}{{\mathbb{N}}}

\newcommand{\Expect}{\mathbb{E}}
\newcommand{\expect}[1]{\mathbb{E}\left[ #1 \right]}

\newcommand{\Prob}{\mathbb{P}}

\newcommand{\prob}[1]{ \mathbb{P}\left\{ #1 \right\} }

\newcommand{\var}{\mathsf{var}}

\def\Var{\mathrm{Var}}

\newcommand{\Bern}{{\rm Bern}}
\newcommand{\Binom}{{\rm Binom}}

\newcommand{\ie}{i.e.\xspace}
\newcommand{\iid}{i.i.d.\xspace}
\newrefformat{eq}{(\ref{#1})}
\newrefformat{chap}{Chapter~\ref{#1}}
\newrefformat{sec}{Section~\ref{#1}}
\newrefformat{alg}{Algorithm~\ref{#1}}
\newrefformat{fig}{Fig.~\ref{#1}}
\newrefformat{tab}{Table~\ref{#1}}
\newrefformat{rmk}{Remark~\ref{#1}}
\newrefformat{clm}{Claim~\ref{#1}}
\newrefformat{def}{Definition~\ref{#1}}
\newrefformat{cor}{Corollary~\ref{#1}}
\newrefformat{lmm}{Lemma~\ref{#1}}
\newrefformat{prop}{Proposition~\ref{#1}}
\newrefformat{prob}{Problem~\ref{#1}}
\newrefformat{app}{Appendix~\ref{#1}}
\newrefformat{hyp}{Hypothesis~\ref{#1}}
\newrefformat{thm}{Theorem~\ref{#1}}

\newcommand{\pth}[1]{\left( #1 \right)}
\newcommand{\qth}[1]{\left[ #1 \right]}
\newcommand{\sth}[1]{\left\{ #1 \right\}}

\newcommand{\norm}[1]{\|{#1} \|}

\newcommand{\iprod}[2]{\left \langle #1, #2 \right\rangle}
\newcommand{\Iprod}[2]{\langle #1, #2 \rangle}
\newcommand{\indc}[1]{{\mathbf{1}_{\left\{{#1}\right\}}}}

\newcommand{\calD}{{\mathcal{D}}}
\newcommand{\calE}{{\mathcal{E}}}
\newcommand{\calF}{{\mathcal{F}}}
\newcommand{\calG}{{\mathcal{G}}}

\newcommand{\calN}{{\mathcal{N}}}
\newcommand{\calO}{{\mathcal{O}}}
\newcommand{\calP}{{\mathcal{P}}}
\newcommand{\calQ}{{\mathcal{Q}}}

\newcommand{\calS}{{\mathcal{S}}}
\newcommand{\calT}{{\mathcal{T}}}

\newcommand{\orbit}{O}

\newcommand{\ER}{Erd\H{o}s-R\'enyi\xspace}

\newcommand{\Tr}{\mathsf{Tr}}

\renewcommand{\tilde}{\widetilde}
\renewcommand{\hat}{\widehat}
\newcommand{\mmse}{\mathrm{mmse}}

\newcommand{\overlap}{\mathsf{overlap}}
\newcommand{\Api}{A^{\pi}}
\newcommand{\tpi}{{\widetilde{\pi}}}
\newcommand{\hpi}{{\pi'}}

\newcommand{\iii}{{i}}

\newcommand{\stepa}[1]{\overset{\rm (a)}{#1}}
\newcommand{\stepb}[1]{\overset{\rm (b)}{#1}}
\newcommand{\stepc}[1]{\overset{\rm (c)}{#1}}
\newcommand{\stepd}[1]{\overset{\rm (d)}{#1}}

\newcommand{\piML}{\hat{\pi}_{\sf ML}}

\begin{document}
    \title{
    Settling the Sharp Reconstruction Thresholds \\of Random Graph Matching}
    \author{Yihong Wu, Jiaming Xu, and Sophie H. Yu \thanks{
    This paper was presented in part at the IEEE International Symposium on Information Theory, July, 2021. 
            Y.\ Wu is with Department of Statistics and Data Science, Yale University, New Haven CT, USA, 
            \texttt{yihong.wu@yale.edu}.
            J.\ Xu and S.\ H.\ Yu are with The Fuqua School of Business, Duke University, Durham NC, USA, 
            \texttt{\{jx77,haoyang.yu\}@duke.edu}.
            Y.~Wu is supported in part by the NSF Grant CCF-1900507, an NSF CAREER award CCF-1651588, and an Alfred Sloan fellowship.
            J.~Xu is supported by the NSF Grants IIS-1838124, CCF-1850743, and CCF-1856424.
    }}

    \date{\today}

    
    \maketitle
    \begin{abstract} 
     This paper studies the problem of recovering the hidden vertex correspondence between two edge-correlated random graphs.
     We focus on the Gaussian model where the two graphs are complete graphs with correlated Gaussian weights and the  \ER model where the two graphs are subsampled from 
        a common parent \ER graph $\calG(n,p)$.  For dense \ER graphs with $p=n^{-o(1)}$, we prove that there exists a sharp threshold, above which one can correctly match all but a vanishing fraction of vertices and below which correctly matching any positive fraction is impossible,    a phenomenon known as the ``all-or-nothing'' phase transition. 
     Even more strikingly, in the Gaussian setting, above the threshold all vertices can be exactly matched with high probability. In contrast, for sparse \ER graphs with $p=n^{-\Theta(1)}$, we show that the all-or-nothing phenomenon no longer holds and
     we determine the thresholds up to a constant factor.
      Along the way, we also derive the sharp threshold for exact recovery, sharpening the existing results in Erd\H{o}s-R\'enyi graphs~\cite{cullina2016improved,cullina2017exact}.
    
    The proof of the negative results builds upon a tight characterization of the mutual information based on the truncated second-moment computation in~\cite{wu2020testing}  and an ``area theorem'' that relates the mutual information to the integral of the reconstruction error. The positive  results follows from a tight analysis of the maximum likelihood estimator that takes into account the cycle structure of the induced permutation on the edges.


    \end{abstract}

    \tableofcontents

    \section{Introduction}

    The problem of graph matching (or network alignment) refers to 
    finding the underlying vertex correspondence between two graphs on the sole basis of their network topologies. 
    Going beyond the worst-case intractability of finding the optimal correspondence (quadratic assignment problem  \cite{pardalos1994quadratic,burkard1998quadratic}), an emerging line of research is devoted to the average-case analysis of graph matching under meaningful statistical models, focusing on either information-theoretic limits \cite{cullina2016improved,cullina2017exact,cullina2019partial,hall2020partial,wu2020testing,ganassali2020sharp} or computationally efficient algorithms \cite{feizi2016spectral,lyzinski2016graph,ding2018efficient,barak2019nearly,FMWX19a,FMWX19b,ganassali2020tree}. 
    Despite these recent advances, the sharp thresholds of graph matching remain not fully understood especially for approximate reconstruction. The current paper aims to close this gap.
    
    Following \cite{pedarsani2011privacy,ding2018efficient}, we consider the following probabilistic model for 
    two random graphs correlated through a hidden vertex correspondence. Let the ground truth
    $\pi$ be a uniformly random permutation on $[n]$.
    We generate two random weighted graphs on the common vertex set $[n]$ with (weighted) adjacency vectors
    $A=(A_{ij})_{1\le i<j\le n}$ and $B=(B_{ij})_{1 \le i<j \le n}$ 
    such that $\left(A_{\pi(i)\pi(j)}, B_{ij}\right)$ are i.i.d.\ pairs of correlated
    random variables with a joint distribution $P$, which implicitly depends on $n$.
Of particular interest are the following two special cases:
    \begin{itemize}
        \item (Gaussian model): $
        P= \calN  \Big( \left(\begin{smallmatrix} 0\\ 0\end{smallmatrix}\right) , \left(\begin{smallmatrix} 1 & \rho \\ \rho & 1 \end{smallmatrix}\right)  \Big )$ is the joint distribution of two standard Gaussian random variables with correlation coefficient $\rho>0$.       
        In this case, we have $B=\rho A^\pi + \sqrt{1-\rho^2} Z$, where $A$  and $Z$ are independent standard normal vectors and $A^{\pi}_{ij}=A_{\pi(i)\pi(j)}$. 
        
        \item (\ER random graph): $P$ denotes the joint distribution of two correlated $\Bern(q)$ random variables $X$ and $Y$
        such that $\prob{Y=1\mid X=1} =s$, where $q \leq s \leq 1$. In this case, $A$ and $B$ are the adjacency vectors of two \ER random graphs
        $G_1, G_2 \sim \calG(n,q)$, where $G_1^\pi$ (with the adjacency vector $A^\pi$) and $G_2$ are independently edge-subsampled from a common parent graph  $G\sim \calG(n,p)$ with $p=q/s$.
    \end{itemize}

    Given the observation $A$ and $B$, the goal is to recover the latent
    vertex correspondence $\pi$ as accurately as possible.
    More specifically, given two permutations $\pi,\hat{\pi}$ on $[n]$,
    denote the fraction of their overlap by
    \[
    \overlap(\pi, \hat{\pi}) \triangleq \frac{1}{n} \left| \left\{i \in [n]: \pi(i)=\hat{\pi}(i) \right\}\right|.
    \]
        \begin{definition}
        We say an estimator $\hat{\pi}$ of $\pi$ achieves, as $n\diverge$,
        \begin{itemize}
            \item \emph{partial}  recovery, if $\prob{ \overlap \left(  \hat{\pi}, \pi  \right) \ge \delta  } = 1- o(1)$ for some constant $\delta \in (0,1)$;
            \item \emph{almost exact}  recovery, if $\prob{ \overlap \left(  \hat{\pi}, \pi  \right) \ge \delta  } = 1- o(1)$ for any constant $\delta \in (0,1)$;
            \item \emph{exact}  recovery, if $\prob{ \overlap \left(  \hat{\pi}, \pi  \right) =1  } = 1- o(1)$.         
        \end{itemize}
    \end{definition}

    The information-theoretic threshold of exact recovery has been determined for the \ER graph model \cite{cullina2017exact} 
in the regime of $p= O\left( \log^{-3}(n) \right)$ and more recently for the Gaussian model \cite{ganassali2020sharp}; however, the results and proof techniques in \cite{cullina2017exact} do not hold for denser graphs. In contrast, approximate recovery are far less well understood. 
    Apart from the sharp condition for almost exact recovery in the sparse regime $p= n^{-\Omega(1)}$ \cite{cullina2019partial}, only upper and lower bounds are known for partial recovery \cite{hall2020partial}. See \prettyref{sec:prior} for a detailed review of these previous results.
    
    In this paper, we characterize the sharp reconstruction thresholds for both the Gaussian and dense \ER graphs with $p=n^{-o(1)}$. Specifically, we prove that there exists a sharp threshold, above which one can estimate all but a vanishing fraction of the latent permutation and below which recovering any positive fraction is impossible, 
    a phenomenon known as the ``all-or-nothing'' phase transition~\cite{reeves2019all}. This phenomenon is even more striking in the Gaussian model, in the sense that above the threshold the hidden permutation can be estimated error-free with high probability.
    In contrast, for sparse \ER graphs with $p=n^{-\Theta(1)}$, we show that the all-or-nothing phenomenon no longer holds. 
    To this end, we determine the threshold for partial recovery up to a constant factor and show that it is order-wise smaller than the almost exact recovery threshold found in~\cite{cullina2019partial}.    
    
    
    Along the way, we also 
derive a sharp threshold for exact recovery, 
    sharpening existing results in~\cite{cullina2016improved,cullina2017exact}.
    As a byproduct, the same technique yields an alternative proof of the result in~\cite{ganassali2020sharp} for the Gaussian model.

    \subsection{Main results}
    Throughout the paper, we let $\epsilon>0$ denote an arbitrarily small but fixed constant.
    Let     $\piML$ denote the maximum likelihood estimator, which reduces to 
    \begin{equation}
        \piML \in \argmax_{\pi'} \iprod{A^{\pi'}}{B} \, .
        \label{eq:ML}
    \end{equation}
    
    \begin{theorem}[Gaussian   model]\label{thm:Gaussian_negative}
    If 
    \begin{equation}
        \rho^2 \ge \frac{(4+\epsilon) \log n}{n},
        \label{eq:Gaussian_positive}
        \end{equation}
        then 
        $
        \prob{ \overlap \left( \hat{\pi}_{\rm ML}, \pi \right) =1 } = 1-o(1).
        $
        
Conversely, if 
        \begin{equation}
        \rho^2 \le \frac{(4-\epsilon) \log n}{n},
        \label{eq:Gaussian_negative}
        \end{equation}
 then for any estimator $\hat{\pi}$
        and any fixed constant $\delta>0$,
        $
        \prob{ \overlap \left( \hat{\pi}, \pi \right) \le \delta } = 1-o(1).
        $
    \end{theorem}
\prettyref{thm:Gaussian_negative} implies that in the Gaussian setting,
    the recovery of $\pi$ exhibits a sharp phase transition
    in terms of the limiting value of $\frac{n\rho^2}{\log n}$ at threshold $4$, above which exact recovery is possible
    and below which even partial recovery is impossible. 
    The positive part of~\prettyref{thm:Gaussian_negative} was first shown in~\cite{ganassali2020sharp}. Here we provide an alternative proof
    that does not rely on the Gaussian property and works for \ER graphs as well.

    The next result determines the sharp threshold for the \ER model in terms of the key quantity $nps^2$, the average degree of the intersection graph $G_1\wedge G_2$ (whose edges are sampled by both $G_1$ and $G_2$).

    \begin{theorem}[\ER graphs, dense regime]\label{thm:ER_dense}
        Assume $p$ is bounded away from $1$
        and $p=n^{-o(1)}$. 
        If 
        \begin{equation}        
        nps^2 \ge \frac{(2+\epsilon) \log n}{  \log \frac{1}{p} -1 + p },
        \label{eq:ER_dense1}
        \end{equation}
        then for any constant $\delta<1$, 
        $
        \prob{ \overlap \left( \hat{\pi}_{\rm ML}, \pi \right) \ge \delta  } = 1- o(1).
        $
        Conversely, if 
        \begin{equation}        
        nps^2 \le \frac{(2-\epsilon) \log n}{  \log \frac{1}{p} -1 + p },
        \label{eq:ER_dense2}
        \end{equation}
        then for any estimator $\hat{\pi}$ and any constant $\delta>0$, 
        $
        \prob{ \overlap \left( \hat{\pi}, \pi \right) \le \delta  } = 1-o(1).
        $
    \end{theorem}
    \prettyref{thm:ER_dense} implies that analogous to the Gaussian   model, 
    in dense \ER graphs, the recovery of $\pi$ exhibits an ``all-or-nothing'' phase transition
    in terms of the limiting value of $\frac{nps^2\left( \log \frac{1}{p} -1 + p \right) }{\log n}$ at threshold $2$, above which almost exact recovery is possible
    and below which even partial recovery is impossible. However, as we will see in~\prettyref{thm:exact_ER},
    unlike the Gaussian   model, this threshold differs from that of exact recovery. 

    
\begin{remark}[Entropy interpretation of the thresholds]
    \label{rmk:interp}
        The sharp thresholds in \prettyref{thm:Gaussian_negative} and \prettyref{thm:ER_dense} can be interpreted from an information-theoretic perspective.
        Suppose an estimator $\hat\pi=\hat\pi(A,B)$ achieves almost exact recovery with $\Expect[\overlap(\pi,\hat\pi)]=1-o(1)$, which, by a rate-distortion computation, implies that $I(\pi;\hat\pi)$ must be close to the full entropy of $\pi$, that is, $I(\pi;\hat\pi) = (1-o(1)) n\log n$. On the other hand, by the data processing inequality, we have $I(\pi;\hat\pi) \leq I(\pi;A,B)$. The latter can be bounded simply as (see  \prettyref{sec:technique-negative-partial})
        \begin{equation}
        \label{eq:MI-simple}
        I(\pi;A,B) \le \binom{n}{2} I(P),
        \end{equation}
        where $I(P)$ denote the mutual information between a pair of random variables with joint distribution $P$.
        For the Gaussian   model, we have 
        \begin{equation}
        I(P) = \frac{1}{2} \log \frac{1}{1-\rho^2}.
        \label{eq:IP-gaussian}
        \end{equation}      
        For the correlated \ER graph,
        \begin{equation}
        I(P) = q d(s\|q) + (1-q) d(\eta \|q),
        \label{eq:IP-ER}
        \end{equation}
        where $q=ps$, $\eta \triangleq \frac{q(1-s)}{1-q}$, and $d(s\|q) \triangleq D(\Bern(s)\|\Bern(q))$ denotes the binary KL divergence.
        By Taylor expansion, we have $I(P) = s^2 p \pth{p-1+\log\frac{1}{p}}(1-o(1))$.
        Combining these with $\binom{n}{2} I(P) \geq (1-o(1))n\log n$ shows
        the impossibility of 	almost exact recovery 
        under the conditions \prettyref{eq:Gaussian_negative} and \prettyref{eq:ER_dense2}.
        The fact that they are also necessary for partial recovery takes more effort to show, which we do in \prettyref{sec:impossibility}. 
    \end{remark}
    
    \begin{theorem}[\ER graphs, sparse regime]\label{thm:ER_sparse}
        Assume $p=n^{-\Omega(1)}$. 
        If 
        \begin{equation}
            \label{eq:partial-sparse-positive}
        np s^2 \ge (2+\epsilon) \max\left\{ \frac{\log n}{\log (1/p)} , 2 \right\},
        \end{equation}
        then there exists a constant $\delta>0$ such that 
        $
        \prob{ \overlap \left( \hat{\pi}_{\rm ML}, \pi \right) \ge \delta  } = 1- o(1).
        $
        Conversely, assuming $np=\omega(\log^2 n)$, if 
        \begin{align}
            nps^2 \le 1-\epsilon,
            \label{eq:2lowerbound_sparse_strong}
        \end{align} 
        then for any estimator $\hat{\pi}$ and any constant $\delta>0$, 
        $
        \prob{ \overlap \left( \hat{\pi}, \pi \right) \le \delta  } = 1-o(1).
        $
    \end{theorem}
    \prettyref{thm:ER_sparse} implies that for sparse \ER graphs with $p=n^{-\alpha}$ for a constant $\alpha \in (0,1)$, 
    the information-theoretic thresholds  for partial recovery is at $np s^2 \asymp 1$, 
    which is much lower than the  almost exact recovery threshold $nps^2=\omega(1)$ as established in~\cite{cullina2019partial}. Hence, interestingly the all-or-nothing phenomenon no longer holds for sparse \ER graphs. 
    Note that the conditions \prettyref{eq:partial-sparse-positive} and \prettyref{eq:2lowerbound_sparse_strong} differ by a constant factor. 
    Determining the sharp threshold for partial recovery in the sparse  regime remains an open question.

        Finally, we address the exact recovery threshold in the \ER graph model. 
	For ease of notation, we consider a general correlated \ER graph model specified by the joint distribution $P=(p_{ab}: a,b\in\{0,1\})$, so that 
    $
     \prob{A_{\pi(i)\pi(j)}=a, B_{ij}=b}=p_{ab}  \text{ for } a, b \in \{0,1\}. 
    $
     In this general \ER model, $\piML$ is again given by the maximization problem~\prettyref{eq:ML} if $p_{11}p_{00} > p_{01}p_{10}$ (positive correlation) and changes to minimization if $p_{11}p_{00} < p_{01}p_{10}$ (negative correlation). 
        The subsampling model is a special case with positive correlation, where 
     \begin{align}
      \label{eq:general-subsampling}
     p_{11}=ps^2, \quad p_{10}=p_{01}=ps(1-s), \quad p_{00}=1-2ps+ps^2.
     \end{align}
		\begin{theorem}[\ER graphs, exact recovery] \label{thm:exact_ER}
	Under the  subsampling model~\prettyref{eq:general-subsampling}, if 
	\begin{align}
	n \left( \sqrt{p_{00} p_{11}}-\sqrt{p_{01} p_{10}}\right)^2 \ge \left(1+ \epsilon\right) \log n ,
	\label{eq:exact_positive_er}
	\end{align}
	then 
		$
		\prob{ \overlap \left( \hat{\pi}_{\rm ML}, \pi \right) =1 } = 1-o(1).
		$

	Conversely, if 
	\begin{align}
	n \left( \sqrt{p_{00} p_{11}}-\sqrt{p_{01} p_{10}}\right)^2 \le \left(1- \epsilon\right) \log n,
	\label{eq:exact_negative_er}
	\end{align}
		then for any estimator $\hat{\pi}$, 
        $
		\prob{ \overlap \left( \hat{\pi}, \pi \right) =1 } = o(1).
		$
	\end{theorem}
If $p$ is bounded away from $1$, \prettyref{thm:exact_ER} implies that the exact recovery threshold is given by  $\lim_{n\diverge} \frac{n ps^2 \left( 1-\sqrt{p}  \right)^2}{\log n}=1$. Since $\log \frac{1}{p} -1 + p \ge 2 (1-\sqrt{p})^2$, with equality if and only if $p=1$, the threshold of exact recovery is strictly higher than that of almost exact recovery in the \ER graph model, unlike the Gaussian model. If $p = 1-o(1)$,  \prettyref{thm:exact_ER} implies that the exact recovery threshold is given by  $\lim_{n\diverge} \frac{n \rho^2}{\log n}=4$, where $\rho \triangleq \frac{s(1-p)}{1-ps}$ denotes the correlation parameter between $A_{\pi(i)\pi(j)}$ and $B_{ij}$ for any $i<j$ under the latent permutation $\pi$.

\subsection{Comparisons to Prior work}
    \label{sec:prior}

\paragraph*{Exact recovery}
    The information-theoretic thresholds for exact recovery have been determined for the Gaussian model and the general \ER graph model in certain regimes.
    In particular, for the Gaussian model, it is shown in~\cite{ganassali2020sharp} that 
    if $n\rho^2 \ge (4+\epsilon) \log n$ for any constant $\epsilon>0$, then the MLE achieves exact recovery;
    if instead
    $n\rho^2 \le (4-\epsilon) \log n$, then exact recovery is impossible.
    \prettyref{thm:Gaussian_negative} significantly strengthens this negative result by showing that under the same condition even partial recovery is impossible.

 Analogously, for \ER random graphs,
 it is shown in~\cite{cullina2016improved,cullina2017exact} that the MLE achieve exact recovery when $nps^2 =\log n +\omega(1)$ under the additional restriction that $p=O(\log^{-3}(n))$.\footnote{In fact, \cite{cullina2016improved,cullina2017exact} studied exact recovery condition in a more general correlated \ER model where $\prob{A_{\pi(i)\pi(j)}=a, B_{ij}=b}=p_{ab}$ for $a, b \in \{0,1\}$, which will also be the setting in \prettyref{sec:exact}.}
  Conversely,  exact recovery is shown  in~\cite{cullina2016improved} to be 
 information-theoretically impossible provided that 
  $nps^2 \le \log n -\omega(1)$, based on the fact the intersection graph $G_1 \wedge G_2 \sim \calG(n,ps^2)$ has many isolated nodes with high probability. 
 These two results together imply
  that when $p=O(\log^{-3}(n))$, the exact recovery threshold is
  given by $\lim \frac{nps^2}{\log n}=1$, coinciding with the connectivity threshold of 
  $G_1\wedge G_2$. In comparison, \prettyref{thm:exact_ER} implies that if $p$ is bounded away from $1$, the precise exact recovery threshold is instead given by  $\lim \frac{n ps^2 \left( 1-\sqrt{p}  \right)^2}{\log n}=1$, strictly higher than the connectivity threshold.
  In particular,
 deriving the tight condition~\prettyref{eq:exact_negative_er} requires more than eliminating isolated nodes.  
 In fact, our results show that when $p$ is bounded away from $1$ and $\log n< nps^2 < \frac{\log n}{(1-\sqrt{p})^2}$, exact recovery still fails even when the intersection graph is asymmetric (no non-trivial automorphism) with high probability \cite{wright1971graphs}. 
 See the discussions in~\prettyref{sec:exact_proof_techniques} for more details.

\paragraph*{Partial and almost exact recovery}  
    Compared to exact recovery, the understanding of partial and almost exact recovery is less precise.
    It is shown in~\cite{cullina2019partial} that in the sparse regime 
    $p= n^{-\Omega(1)}$, almost exact recovery is information-theoretically possible if 
    and only if $nps^2=\omega(1)$. 
    The more recent work~\cite{hall2020partial} further investigates partial recovery. It is shown that if 
    $nps^2 \ge C(\delta) \max\left\{ 1, \frac{\log n}{\log (1/p)} \right\}$,
    then there exists an exponential-time estimator $\hat{\pi}$ that achieves 
    $\overlap\left(\hat{\pi},\pi \right) \ge \delta$ with high probability,
    where $C(\delta)$ is some large constant that tends to $\infty$ as $\delta \to 1$; 
    conversely, if $\frac{I(P)}{\delta} = o\left(\frac{\log \left(n\right)}{n}\right)$ with $I(P)$ given in \prettyref{eq:IP-ER}, then no estimator can achieve 
    $\overlap\left(\hat{\pi},\pi \right) \ge \delta$ with positive probability.
    These conditions do not match in general and are much looser than the results in Theorems~\ref{thm:ER_dense} and~\ref{thm:ER_sparse}.

    \subsection{Proof Techniques}

    
    We start by introducing some preliminary definitions associated with
    permutations (cf.~\cite[Section 3.1]{wu2020testing} for more details and examples).
    
    \subsubsection{Node permutation, edge permutation, and cycle decomposition}
    \label{sec:orbit}
    
    Let $\calS_n$ denote the set of permutations on the node set $[n]$. 
    Each $\sigma\in \calS_n$ induces a permutation $\sigmae$ on the edge set $\binom{[n]}{2}$ of unordered pairs, according to
    \begin{equation}
        \sigmae((i,j)) \triangleq (\sigma(i),\sigma(j)).
        \label{eq:sigmae}
    \end{equation}
    We shall refer to $\sigma$ and $\sigmae$ as a node permutation and edge permutation. 
    Each permutation can be decomposed as disjoint cycles known as orbits.  Orbits of $\sigma$ (resp.~$\sigmae$) are referred as \emph{node orbits} (resp.~edge orbits).
    Let $n_k$ (resp.~$N_k$) denote the number of $k$-node (resp.~$k$-edge) orbits in $\sigma$ (resp.~$\sigmae$). 
    The cycle structure of $\sigmae$ is determined by that of $\sigma$. For example, we have
    \begin{align}
    N_1 = \binom{n_1}{2}+ n_2, \label{eq:fixed_edge_node}
    \end{align}
     because an edge $(i,j)$ is a fixed point of $\sigmae$ if and only if 
     either both $i$ and $j$ are fixed points of $\sigma$ or $(i,j)$ forms a $2$-node orbit of $\sigma$.
     Let  $F$ be the set of fixed points of $\sigma$ with $|F|=n_1$.     
     Denote $\calO_1=\binom{F}{2} \subset \calO$ as the subset of fixed points of edge permutation $\sigmae$,  where $\calO$ denotes the collection of all edge orbits of $\sigmae$. 

    \subsubsection{Negative results on partial recovery}
    \label{sec:technique-negative-partial}
    Let $\calP$ denote the joint distribution of $A$ and $B$ under the correlated model. 
    To prove our negative results, we introduce an auxiliary null model $\calQ$,
    under which $A$ and $B$ are independent with the same marginal as $\calP$. In other words, under $\calQ$, $(A_{ij}, B_{ij})$ are 
    i.i.d.\ pairs of independent random variables with a joint distribution $Q$ equal to the product of the marginals of $P$.

    As the first step, we leverage the previous truncated second moment computation in~\cite{wu2020testing} to conclude that
    the KL-divergence $\calD(\calP\|\calQ)$ is negligible under the desired conditions. 
    By expressing the mutual information as $I(\pi; A, B)=\binom{n}{2} D(P\|Q) - D(\calP\|\calQ)    $ where $D(P\|Q)=I(P)$, 
    this readily implies that $I(\pi;A, B) = \binom{n}{2} I(P)(1+o(1))$.
    Next, we relate the mutual information $I(\pi;A,B)$ to the integral of 
    the minimum mean-squared error (MMSE) of $A^\pi$, the weighted adjacency vector relabeled according to the ground truth.
    For the Gaussian model, this directly follows from the celebrated I-MMSE formula \cite{GuoShamaiVerdu05}. 
    For correlated \ER graphs, we introduce an appropriate interpolating model and obtain an analogous but more involved ``area theorem'', following \cite{measson2009generalized,deshpande2015asymptotic}. 
    These two steps together imply that the MMSE of $A^\pi$ given the observation $(A,B)$ is 
    asymptotically equal to the estimation error of the trivial estimator $\expect{A^\pi}$. 
    Finally, we connect the MMSE of $A^\pi$ to the Hamming loss of reconstructing $\pi$, concluding the impossibility of the partial recovery.
    
    Note that by the non-negativity of $D(\calP\|\calQ)$, we arrive at the simple upper bound \prettyref{eq:MI-simple}, that is,  $I(\pi;A, B) \le \binom{n}{2} I(P)$. 
    Interestingly, our proof relies on establishing an asymptotically matching lower bound to the mutual information $I(\pi;A, B) $.
    This significantly deviates from the existing results in~\cite{hall2020partial} based on Fano's inequality:
    $\prob{\overlap(\hat{\pi},\pi) \le \delta} \ge  1- \frac{I(\pi;A, B)+1}{\log (n!/m)}$
    with $m=|\{\pi: \overlap(\pi, \pi) \ge \delta\}|$, followed by applying the simple  bound \prettyref{eq:MI-simple}.

    \subsubsection{Positive results on partial and almost exact recovery}

    Our positive results follow from a large deviation analysis of the maximum likelihood estimator~\prettyref{eq:ML}.  A crucial observation is that the difference between the objective function in~\prettyref{eq:ML} evaluated at a given permutation $\pi'$ and that at the ground truth $\pi$ can be 
    decomposed 
    across the edge orbits of $\sigma\triangleq  \pi^{-1}\circ \hpi $ as 
    $$
     \iprod{A^{\pi'} - A^{\pi}}{B} = \sum_{O \in \calO \setminus \calO_1 } X_O - 
     \sum_{O \in \calO \setminus \calO_1} Y_O \; \triangleq X - Y,
    $$
    where  $F$ is the fixed point of $\sigma$,
    $\calO_1=\binom{F}{2} \subset \calO$ is a  subset of fixed points of the edge permutation $\sigmae$,
    $X_O\triangleq \sum_{(i,j)\in O}  A_{\hpi(i)\hpi(j)} B_{ij}$, and 
    $Y_O \triangleq \sum_{(i,j)\in O}  A_{\pi(i)\pi(j)} B_{ij}$, 
    are independent across  edge orbits $O$.
    Crucially, $Y$  depends on $\pi'$
    only through its fixed point set $F$, which has substantially fewer
    choices than $\pi'$ itself when $n-|F| \asymp n$.
    Therefore, for the purpose of applying the union bound it is beneficial to separately control $ X$ and $Y$. 
    Indeed, we show that $Y$ is highly concentrated on its mean. 
Hence, 
    it remains to analyze the large-deviation event of $ X$ exceeding $\expect{Y}$, which is accomplished by a careful computation of 
    the moment generation function (MGF)
    $M_{|O|} \triangleq \expect{\exp\left( t X_O\right)}$ and proving that \begin{align}
    M_{|O|}\le M_2^{|O|/2},     \text{ for }    |O|\ge 2.  
    \label{eq:mgf_bound}
    \end{align}
    Intuitively, it means that the contribution of longer edge orbits can be effectively bounded by that of the $2$-edge orbits.
    Capitalizing on this key finding and applying the Chernoff bound together with a union bound 
    over $\pi'$  yields the tight condition when $q=o(1)$.
    When $q=\Theta(1)$, it turns out that the correlation between $X$ and $Y$ can no longer be ignored so that separately controlling $X $ and $Y$ leads to suboptimal results. To remedy this issue, we need to center the adjacency vectors $A$ and $B$. More precisely, 
    define 
    $\overline{X}_O\triangleq \sum_{(i,j)\in O}  (A_{\hpi(i)\hpi(j)} - q) (B_{ij}-q )$, 
 $\overline{Y}_O \triangleq \sum_{(i,j)\in O}  (A_{\pi(i)\pi(j)} - q) (B_{ij}-q)$, $\overline{X}\triangleq \sum_{O \in \calO \setminus \calO_1 } \overline{X}_O$, and 
 $\overline{Y} \triangleq \sum_{O \in \calO \setminus \calO_1 }  \overline{Y}_O$. 
 By definition, $X-Y=\overline{X}-\overline{Y}$. Crucially, it can be verified that $\overline{X}_O$ and $\overline{Y}_O$ 
 (and hence $\overline{X}$ and $\overline{Y}$)
 are uncorrelated after the centering. Thus we can apply our aforementioned techniques to separately bound $\overline{X}$ and $\overline{Y}$, which yield the sharp conditions of recovery.

    We remark that the partial recovery results in~\cite{hall2020partial} are obtained  by analyzing an estimator slightly different from the MLE and the same MGF bound~\prettyref{eq:mgf_bound} is used. 
    However, there are two major differences that led to the looseness of their results. 
    First, their analysis does not separately bound $X$ and $Y$. Second, the tilting parameter in the Chernoff bound is suboptimal.

    \subsubsection{Exact recovery}\label{sec:exact_proof_techniques}
    
    For  exact recovery, we need to further consider $\pi'$ that is close to $\pi$, \ie, $n-|F|=o(n)$. 
    In this regime, the number of choices of $F$ is comparable to that of $\pi'$. Hence, instead of separately bounding $X$ and $Y$,  it is  more favorable to directly applying the Chernoff bound to the difference $X-Y$.  Crucially, the moment generation function $\expect{\exp\left( t (X_O-Y_O)\right)}$ continues to satisfy the relation~\prettyref{eq:mgf_bound} and the bottleneck for exact recovery happens at $|F|=n-2$, where $\pi'$ differs from $\pi$ from a $2$-cycle (transposition).

    Prompted by this observation, we prove a matching necessary condition of exact recovery 
    by considering all possible permutations $\sigma \triangleq  \pi^{-1}  \circ \hpi$ that consists of $n-2$
    fixed points and a 2-node cycle $(i,j)$ (transposition), in which case,
    \begin{align}
    \left \langle  A^{\pi'}- A^\pi, B\right\rangle = 
    - \sum_{k \in [n]\backslash \{ij\}} 
    \left( A^\pi_{ik} - A^\pi_{jk}\right)\left(B_{ik}-B_{jk}\right).\label{eq:obj_diff_2}
    \end{align}
    There remains two key challenges to conclude the existence of many choices of $(i,j)$ for which 
    $\langle A^{\pi'}, B\rangle\geq \langle A^\pi, B\rangle$. First, to derive a tight 
    impossibility condition, we need to obtain a tight large-deviation lower estimate for this event. Second, the RHS of~\prettyref{eq:obj_diff_2} for different pairs
    $(i,j)$ are correlated. 
    This dependency is addressed by restricting the choices of $(i,j)$ and applying a second moment computation.

    Note that the impossibility proof of exact recovery for the Gaussian model in \cite{ganassali2020sharp} also considers the permutations that consist of a single transposition. The difference is that the large-deviation lower estimate simply follows from Gaussian tail probability
    and the correlations among different pairs $(i,j)$ is bounded by a second-moment
    calculation using 
    densities of correlated Gaussians.
    

        \section{Impossibility of Partial Recovery}
    \label{sec:impossibility}
    

    To start, we characterize the asymptotic value of the mutual information $I(A, B; \pi)$ -- a key quantity that measures
    the amount of information about $\pi$ provided by the observation $(A,B)$. By definition, 
    $$
    I(A,B;\pi) \triangleq
    \expect{ D \left( \calP_{A,B \mid \pi} \| \calP_{A,B} \right)}
    =\expect{ D \left( \calP_{A,B \mid \pi} \| \calQ_{A,B} \right)} - D\left( \calP_{A,B}||\calQ_{AB} \right)
    $$
    for any joint distribution $\calQ_{A,B}$ of $(A,B)$ such that $D\left( \calP_{A,B}||\calQ_{AB} \right)<\infty$.
    Note that $\calP_{A,B \mid \pi}$ factorizes into a product distribution $\prod_{i<j} \calP_{A_{\pi(i)\pi(j)}, B_{ij}}=P^{\otimes \binom{n}{2}}$, where $P$ is the joint distribution of $(A_{\pi(i)\pi(j)}, B_{ij})$. 
    Thus, to exploit the tensorization property of the KL-divergence,
    we choose $\calQ_{A,B}$ to be a product distribution under which $A$ and $B$ are independent
    and $(A_{ij}, B_{ij})$ are i.i.d.\ pairs of independent random variables with a joint distribution
    $Q$ with the same marginals as $P$. (We shall refer to this $\calQ_{A,B}$ as the null model.)
    In particular, for the Gaussian   (resp.~\ER) model, $Q$ is the joint distribution
    of two independent standard normal (resp.~$\Bern(q)$) random variables.
    Under this choice, we have $D \left( \calP_{A,B \mid \pi} \| \calQ_{A,B} \right)=\binom{n}{2}D(P\|Q)=\binom{n}{2} I(P)$
    and hence
    $$
    I(A,B;\pi) 
    =\binom{n}{2} I(P)- D\left( \calP_{A,B}||\calQ_{AB} \right).
    $$
    By the non-negativity of the KL divergence, we have     $I(A, B; \pi) \le \binom{n}{2} I(P)$.  This bound turns out to be tight, as made precise by the following proposition. 
    \begin{proposition}\label{prop:MI}
        It holds that 
        $$
        I(A,B;\pi)  =\binom{n}{2} I(P) - \zeta_n,
        $$
        where
        \begin{itemize}
            \item $\zeta_n=o(1)$ in  the Gaussian   model
            with $\rho^2 \le \frac{(4-\epsilon)\log n}{n}$;
            \item $\zeta_n=o(1)$  in  the dense \ER graphs with $p=n^{-o(1)}$
            and $nps^2 \left( \log (1/p)-1+p \right) \le (2-\epsilon) \log(n)$;
            \item 
            $\zeta_n=O(\log n)$ in the sparse \ER graphs with
            $p=n^{-\Omega(1)}$ and $np=\omega(1)$ and $nps^2 \le 1-\epsilon$;
        \end{itemize}
        for some arbitrarily small but fixed constant $\epsilon>0$.
    \end{proposition}
    
    Given the tight characterization of  the mutual information in \prettyref{prop:MI}, we now relate it to the Bayes risk.
    Using the chain rule, we have
    $$
    I(A, B; \pi) = I(B;\pi \mid A) = I(B; A^\pi \mid A),
    $$
    where the second equality follows from the fact that $A\to A^\pi\to B$ forms a Markov chain.
    The intuition is that conditioned on $A$, $B$ is a noisy observation of $A^\pi$ (which is random owning to $\pi$).
    In such a situation, the mutual information can typically be related to an integral of the reconstruction error of the signal $A^\pi$.  
    To make this precise, 
    we first introduce a parametric model $\calP_\theta$ that interpolates
    between the planted model $\calP$
    and the null model $\calQ$ as $\theta$ varies. 
    We write $\Expect_\theta$ to indicate expectation taken with respect to the law $\calP_\theta$.

    For the Gaussian model, 
    let $\calP_\theta$ denote the model under which 
    $
    B= \sqrt{\theta} A^\pi + \sqrt{1-\theta} Z,
    $
    where $A, Z$ are two independent Gaussian   matrices and $\theta \in [0,1]$. 
    Then $\theta=\rho^2$ corresponds to the planted model $\calP$
    while $\theta=0$ corresponds to the null model $\calQ$.
    As $\theta$ increases from $0$ to $\rho^2$, $\calP_\theta$ interpolates between $\calQ$ and $\calP$.
    Let 
    \begin{equation}
        \mmse_\theta(\Api) \triangleq \Expect_\theta[\norm{A^\pi-\Expect_\theta[A^\pi|A,B]}^2]
        \label{eq:mmse}
    \end{equation}
    denote the minimum mean-squared error (MMSE) of estimating $A^\pi$ based on $(A,B)$ distributed according to $\calP_\theta$.  
    The following proposition follows from the celebrated I-MMSE formula~\cite{GuoShamaiVerdu05}. 
    \begin{proposition}[Gaussian   model]\label{prop:I_MMSE}
        $$
        I(A, B; \pi)= \frac{1}{2} \int_{0}^{\rho^2} \frac{\mmse_\theta(\Api)}{(1-\theta)^2}  d\theta \,.
        $$
    \end{proposition}
    
    The correlated \ER graph model requires more effort. Let us fix $q=ps$ and consider the following coupling $P_\theta$ 
    between two $\Bern(q)$ random variables with joint probability mass function $p_\theta$, 
    where $p_\theta(11)=q\theta$, 
    $p_\theta(01)=p_\theta(10)=q(1-\theta)$, and $p_\theta(00)=1-(2-\theta)q$, with $\theta \in [q,s]$. 
    Let $\calP_\theta$ denote the interpolated model under which $\left(A_{\pi(i)\pi(j)}, B_{ij} \right)$
    are i.i.d.\ pairs of correlated random variables with joint distribution $P_\theta$.
    As $\theta$ increases from $q$ to $s$, $\calP_\theta$ interpolates between the null model $\calQ=\calP_q$ and the planted model $\calP=\calP_s$.
    We have the following area theorem that relates $I(A, B; \pi)$ to the MMSE of $A^\pi$.
    
    \begin{proposition}[\ER random graph]\label{prop:I_MMSE_ER}
        It holds that 
        \begin{align*}
            I(A, B; \pi) \le \binom{n}{2} I(P)  + \binom{n}{2} qs^2 + \int_{q}^s \frac{\theta-q}{s(1-q)^2} 
            \left( \mmse_\theta(\Api) - \binom{n}{2} q (1-q) \right) d\theta.
        \end{align*}
    \end{proposition}
    
    Finally, we  relate the estimation error of $\Api$ to that of $\pi$.
    \begin{proposition}\label{prop:graph_to_permutation}
        In both the Gaussian   and \ER graph model, if
        \begin{align}
            \mmse_{\theta}(A^\pi) \ge \expect{\norm{A}^2} (1-\xi), \label{eq:triv_mmse_cond}
        \end{align}
        for some $\xi>0$, then for any estimator $\hat{\pi}=\hat{\pi}(A,B)$, 
        \begin{align*}
            \Expect_\theta[\overlap(\hat{\pi}, \pi)] &\le O\left(\xi^{1/4}  + \left(\frac{n\log n}{\expect{\norm{A}^2}} \right)^{1/4}  \right).
        \end{align*}
    \end{proposition}
    
    Now, we are ready to prove the negative results on partial recovery. 
    We start with the Gaussian case. 
    \begin{proof}[Proof of \prettyref{thm:Gaussian_negative}]
        In the Gaussian model, we have
        \[
        I(P) = D\pth{\calN  \Big( \left(\begin{smallmatrix} 0\\ 0\end{smallmatrix}\right) , \left(\begin{smallmatrix} 1 & \rho \\ \rho & 1 \end{smallmatrix}\right)  \Big )
        \Big\|\calN  \Big( \left(\begin{smallmatrix} 0\\ 0\end{smallmatrix}\right) , \left(\begin{smallmatrix} 1 & 0 \\ 0 & 1 \end{smallmatrix}\right)  \Big )} = \frac{1}{2} \log \frac{1}{1-\rho^2}. 
        \]
        Assume that $\rho^2 = (4-\epsilon/2)  \frac{\log n}{n}$. Fix some $\theta_0\in (0,\rho^2)$ to be chosen. Then
        \begin{align*}
            \binom{n}{2} \frac{1}{2} \log \frac{1}{1-\rho^2}-\zeta_n & \stepa{=}  I(A, B;\pi ) \\
            & \stepb{\le} \frac{1}{2(1-\rho^2)^2} \int_0^{\rho^2} \mmse_\theta(\Api) d\theta\\
            & \stepc{\le} \frac{1}{2(1-\rho^2)^2} 
            \pth{\binom{n}{2} \theta_0 + 
                \mmse_{\theta_0}(\Api) (\rho^2-\theta_0))} \\
            & = \frac{1}{2(1-\rho^2)^2} 
            \pth{\binom{n}{2} \rho^2 + 
                \left(\mmse_{\theta_0}(\Api) -\binom{n}{2} \right) (\rho^2-\theta_0)}
        \end{align*}
        where $\zeta_n=o(1)$ and $(a)$ holds by~\prettyref{prop:MI} ;
        $(b)$ follows from the I-MMSE formula given in~\prettyref{prop:I_MMSE};
        $(c)$ holds because $\mmse_\theta(\Api) \le  \expect{\|A\|_2^2} = \binom{n}{2}$ 
        and the fact that $ \mmse_\theta(\Api)$ is monotone decreasing in $\theta$. 
        Rearranging the terms in the last displayed equation, we get
        \begin{align*}
            \mmse_{\theta_0}(\Api) - \binom{n}{2}
            & \geq  \frac{(1-\rho^2)^2}{\rho^2-\theta_0 } 
            \left(
            \binom{n}{2}  \left( \log \frac{1}{1-\rho^2} - \frac{\rho^2}{\left(1-\rho^2\right)^2} \right)
            -2\zeta_n \right) \\
            &  \geq - \frac{(1-\rho^2)^2}{\rho^2-\theta_0 } \left( \binom{n}{2} \frac{2\rho^4}{(1-\rho^2)^2} + 2\zeta_n \right),
        \end{align*}
        where the last inequality holds because $\log(1+x) \ge x-x^2$ for $x\ge 0$. 
        Choosing $\theta_0=(4-\epsilon) \frac{\log n}{n}$, we conclude that 
        \begin{align}
            \mmse_{\theta_0}(\Api)  
            \ge \binom{n}{2} \left( 1- O\left( \rho^2 + \frac{\zeta_n}{n^2 \rho^2} \right) \right)
            = \binom{n}{2} \left( 1 -O \left( \frac{\log n}{n} \right) \right),\label{eq:trivial_mmse_gaussian}
        \end{align}
        where the last equality holds because $\rho^2=\Theta(\log (n)/n)$ 
        and $\zeta_n=o(1)$. 
        Since $\expect{\|A\|_2^2}=\binom{n}{2}$,
        it follows from~\prettyref{prop:graph_to_permutation} that 
        $$
        \Expect_{\theta_0}[\overlap(\hat{\pi}, \pi)] \le O\left( \left( \frac{\log n}{n}\right)^{1/4} \right).
        $$
        Finally, by Markov's inequality, for any $\delta_n=\omega\left(  \left( \frac{\log n}{n}\right)^{1/4} \right)$ (in particular,
        a fixed constant $\delta_n>0$), $       \calP_{\theta_0}\{ \overlap(\hat{\pi}, \pi) \ge \delta_n \} = o(1)$.
        Note that $\calP_{\theta_0}$ corresponds to the Gaussian model with squared correlation coefficient equal to $\theta_0 = (4-\epsilon)\frac{\log n}{n}$. By the arbitrariness of $\epsilon$, this completes the proof of \prettyref{thm:Gaussian_negative}.
    \end{proof}
    
    Next, we move to the \ER graph model. 
    \begin{proof}[Proof of the negative parts in Theorems \ref{thm:ER_dense}
        and~\ref{thm:ER_sparse}]
        Let $s^2=\frac{(2-\epsilon) \log n}{np (\log (1/p)-1+p)}$ in the dense
        regime and 
        $s^2=\frac{1-\epsilon}{np}$ in the sparse regime. 
        Then we get that 
        \begin{align*}
            -\zeta_n -\binom{n}{2} qs^2 
            & \stepa{\le}  \int_{q}^s \frac{\theta-q}{s(1-q)^2} 
            \left( \mmse_\theta(\Api) - \binom{n}{2} q (1-q) \right) d\theta \\
            & \stepb{\le}  \int_{(1-\epsilon)s}^s  \frac{\theta-q}{s(1-q)^2}   \left( \mmse_{(1-\epsilon)s} (\Api ) - \binom{n}{2} q (1-q) \right) d\theta \\
            & =    \underbrace{\frac{s^2(2\epsilon-\epsilon^2) -2\epsilon sq}{2s(1-q)^2}}_{=\Theta(s)}  \left( \mmse_{(1-\epsilon)s} (\Api ) - \binom{n}{2} q (1-q) \right) 
        \end{align*}
        where $\zeta_n=o(1)$ in the dense regime 
        and $\zeta_n=O(\log n)$ in the sparse regime;
        $(a)$ follows from~\prettyref{prop:MI} and~\prettyref{prop:I_MMSE_ER};
        $(b)$ holds because  
        $\mmse_\theta(\Api) \le  \expect{\|A-\expect{A} \|_2^2} = \binom{n}{2}q (1-q)$ 
        and $ \mmse_\theta(\Api)$ is monotone decreasing in $\theta$.\footnote{The fact that 
            $\mmse_\theta(\Api)$ is monotone decreasing in $\theta$ follows from a simulation argument. 
            Let $(A,B)\sim\calP_\theta$. Fix $\theta'$ such that $q<\theta'<\theta<s$.  
            Define $B'=(B'_{ij})$ by passing each $B_{ij}$ independently through the same (asymmetric) channel $W$ to obtain $B'_{ij}$, where
            $W(0|1) = \frac{(1-q)(\theta-\theta')}{\theta-q}$ and $W(1|0) = \frac{q(\theta-\theta')}{\theta-q}$ are well-defined. Then $(A,B')\sim\calP_{\theta'}$. 
        }       
        Rearranging the terms in the last displayed equation, we conclude that
        \begin{align*}
            \mmse_{(1-\epsilon)s} (\Api )  
            & \ge  \binom{n}{2} q (1-q) \left(  1 - O \left( \frac{\zeta_n}{n^2q s}  + s \right) \right) \\
            & \ge  \binom{n}{2} q \left( 1- O(s) \right),
        \end{align*}
        where the last inequality holds because $q\le s$, $\zeta_n = O(n^2 qs^2)$ and $s^2 \asymp \frac{\log n}{np\log(1/p)}$. 
        Since $\expect{\|A \|_2^2} = \binom{n}{2}q$, it follows from~\prettyref{prop:graph_to_permutation} that 
        $$
        \Expect_{(1-\epsilon)s}[\overlap(\hat{\pi}, \pi)] \le O \left( s^{1/4} + \left( \frac{\log n}{nq}\right)^{1/4}  \right)
        = O \left( \left( \frac{\log n}{nq}\right)^{1/4}  \right),
        $$
        where the last equality holds because $nqs=nps^2=O(\log n)$. 
        Note that in the dense regime, since $s^2 \asymp \frac{\log n}{np(\log\frac{1}{p}-1+p)}$ and $p=n^{-o(1)}$, 
        we have $nq=nps = \omega(\log n)$.
        This also holds in the sparse regime when $s^2 \asymp \frac{1}{np}$ under the extra assumption that $np = \left(\log^2 n \right)$.
        Thus, by Markov's inequality, for any $\delta_n=\omega\left(  \left( \frac{\log n}{nq}\right)^{1/4} \right)$, in particular,
        any fixed constant $\delta_n>0$, 
        we have $\calP_{(1-\epsilon)s}\{ \overlap(\hat{\pi}, \pi) \ge \delta_n \} = o(1)$.
        In other words, we have shown the desired impossibility result under the distribution
        $\calP_{(1-\epsilon)s}$, which corresponds to the correlated \ER model with parameters $p'=\frac{p}{1-\epsilon}$ and $s'=s(1-\epsilon)$. By the arbitrariness of $\epsilon$, this completes the proof.      
    \end{proof}

    \subsection{Proof of \prettyref{prop:MI}}
    In this subsection, we prove \prettyref{prop:MI}, which reduces to 
    bounding $D \left( \calP_{A,B} \| \calQ_{A,B} \right)$. 
    It is well-known that KL divergence can be bounded by the $\chi^2$-divergence (variance of the likelihood ratio). This method, however, is often too loose as the second moment can be derailed by rare events. Thus a more robust version is by means of \emph{truncated second moment}, which has been carried out in \cite{wu2020testing} to bound $\mathrm{TV} \left( \calP_{A,B}, \calQ_{A,B} \right)$ for studying the hypothesis testing problem in graph matching. Here we leverage the same result to bound the KL divergence.
    To this end, we first present a general bound then specialize to our problem in both Gaussian and \ER models. 
    
    \begin{lemma}
        \label{lmm:trunc2MI}        
        Let $\calP_{XY}$ denote the joint distribution of $(X,Y)$.
        Let $\calE$ be an event independent of $X$ such that $\calP(\calE)=1-\delta$.
        Let $\calQ_Y$ be an auxiliary distribution such that $\calP_{Y|X} \ll \calQ_Y$  $\calP_X$-a.s. 
        Then
        \begin{equation}
            D(\calP_Y\|\calQ_Y) \leq \log(1+\chi^2(\calP_{Y|\calE}\|\calQ_Y)) + \delta \pth{\log \frac{1}{\delta} + \Expect[D(\calP_{Y|X} \| \calQ_Y)]} + \sqrt{\delta \cdot \Var \pth{\log\frac{d\calP_{Y|X}}{d\calQ_Y}}},
            \label{eq:trunc2MI}     
        \end{equation}
        where $\calP_{Y|\calE}$ denote the distribution of $Y$ conditioned on $\calE$, 
        and the $\chi^2$-divergence is defined as $\chi^2(P\|Q) = \Expect_Q[(\frac{dP}{dQ})^2]$ if $P\ll Q$ and $\infty$ otherwise. 
    \end{lemma}
    \begin{proof}
        Note that $\calP_Y = (1-\delta) \calP_{Y|\calE} + \delta \calP_{Y|\calE^c}$. 
        Thanks to the convexity of the KL divergence,  Jensen's inequality yields
        \[
        D(\calP_Y\|\calQ_Y) \leq (1-\delta) D(\calP_{Y|\calE} \|\calQ_Y) + \delta D(\calP_{Y|\calE^c} \|\calQ_Y).
        \]
        The first term can be bounded using the generic fact that $D(P\|Q) \leq \log \Expect_Q[(\frac{dP}{dQ})^2]= \log(1+\chi^2(P\|Q))$. 
        Let $g(X,Y) = \log\frac{d\calP_{Y|X}}{d\calQ_Y}$. 
        Using the convexity of KL divergence and the independence of $\calE$ and $X$, we bound the second term as follows:
          \begin{align*}
            D(\calP_{Y|\calE^c} \|\calQ_Y)
            \leq & ~ \Expect[D(\calP_{Y|X,\calE^c} \|\calQ_Y)] \\
             = &~\expect{ \int d\calP_{Y|X,\calE^c} 
             \log \pth{ \frac{d\calP_{Y|X,\calE^c}}{d\calQ_Y} }} \\
            = & ~ \expect{
            \int d\calP_{Y|X} \frac{\indc{(X,Y)\in\calE^c}}{\calP(\calE^c)}
             \log \pth{ \frac{d\calP_{Y|X}}{d\calQ_Y\calP(\calE^c)} }} \\
             = & ~\log \frac{1}{\delta}
             + \frac{1}{\delta}\Expect[(g(X,Y)\indc{(X,Y)\in\calE^c}] \\
            = & ~  \log \frac{1}{\delta} + \Expect[g(X,Y)] + \frac{1}{\delta} \Expect\left[\left(g(X,Y) - \Expect[g(X,Y)]\right) \indc{(X,Y)\in\calE^c}\right].
        \end{align*}
        Applying Cauchy-Schwarz to the last term completes the proof.   
    \end{proof}

    Next we apply \prettyref{lmm:trunc2MI} in the context of the random graph matching by identifying $X$ and $Y$ with the latent $\pi$ and the observation $(A,B)$ respectively.
    Let $\calE$ be a certain high-probability event 
    independent of $\pi$. Then
    \[
    \calP_{A,B|\calE}= \frac{1}{\calP(\calE)\ n! } \sum_{\pi \in \calS_n} \calP_{A,B} \indc{(A,B,\pi) \in \calE}.
    \]
    Recall that the null model is chosen to be $\calQ_{A,B}=\calP_A \otimes \calP_B$. 
    As shown in \cite{wu2020testing}, for both the Gaussian   and \ER graph model, it is possible to construct a high-probability event $\calE$ satisfying the symmetry condition $\calP(\calE \mid \pi)=\calP(\calE)$, such that $\chi^2(\calP_{A,B|\calE}\|\calQ_{A,B}) = o(1)$. This bounds the first term in \prettyref{eq:trunc2MI}. For the second term,  since both $A$ and $B$ are individually independent of $\pi$, we have 
    \begin{equation}
        \expect{\log \frac{d\calP_{ A,B \mid \pi}}{d\calQ_{A,B}}} = I(A;B|\pi) = \binom{n}{2} D(P\|Q) = \binom{n}{2} I(P),
        \label{eq:IABpi}
    \end{equation}
    where the last equality holds because $Q$ is the product of the marginals of $P$.
    The third term in \prettyref{eq:trunc2MI} can be computed explicitly.
    Next we give details on how to complete the proof of \prettyref{prop:MI}.
    
    \paragraph{Gaussian   model}
    It is shown in \cite[Section 4.1, Lemma 1]{wu2020testing} that there exists an event $\calE$ independent of $\pi$ such that
    $
    \calP(\calE^c)=e^{-\Omega(n)}
    $
    and $
    \chi^2 \left( \calP_\calE \| \calQ \right) = o(1),
    $
    provided that $\rho^2 \le \frac{(4-\epsilon)\log n}{n}$.
    Furthermore, by~\prettyref{eq:IABpi} and \prettyref{eq:IP-gaussian}, we have 
    $\expect{\log \frac{d\calP_{ A,B \mid \pi}}{d\calQ_{A,B}}} = \binom{n}{2} \frac{1}{2}\log \frac{1}{1-\rho^2} = O(n \log n)$.
    To compute the variance,
    note that   
    $$
    \log \frac{d\calP_{ A,B \mid \pi}}{d\calQ_{A,B}}
    = - \frac{1}{2} \binom{n}{2} \log (1-\rho^2) -\frac{h(A, B, \pi)}{4(1-\rho^2)}   ,
    $$
    where 
    \[
    h(A, B, \pi) \triangleq \rho^2 \norm{A}^2 + \rho^2 \norm{B}^2  -2 \rho \iprod{A^\pi}{B}.
    \]      
    Thus $\Var( \log \frac{d\calP_{ A,B \mid \pi}}{d\calQ_{A,B}}) = \frac{1}{16(1-\rho^2)^2}\Var(h(A, B, \pi))$. 
    Write $B=\rho A^\pi + \sqrt{1-\rho^2}Z$ where $Z$ is an independent copy of $A$, we have 
    $h(A, B, \pi) = \rho^2 (\norm{B}^2-\norm{A}^2)  - 2  \rho \sqrt{1-\rho^2}\iprod{A^\pi}{Z}$.
    Here both $\norm{A}^2$ and $\norm{B}^2$ are distributed as $\chi^2_{\binom{n}{2}}$, with variance equal to $2 \binom{n}{2}$. Furthermore,
    $\Var(\iprod{A^\pi}{Z}) = \binom{n}{2}$. Thus   $\Var(h(A, B, \pi)) = O(n\log n)$.
    Applying \prettyref{lmm:trunc2MI}, we conclude that $D(\calP_{A,B}\|\calQ_{A,B})=o(1)$.

    \paragraph{\ER Graphs}
    In the dense regime of $p=n^{-o(1)}$ and $p=1-\Omega(1)$, it is shown in \cite[Section A.3, Lemma 6]{wu2020testing} that there exists an event $\calE$ such that
    $
    \calP(\calE^c)=e^{-n^{1-o(1)}}
    $
    and $   \chi^2 \left( \calP_\calE \| \calQ \right) = o(1)$, 
    provided that $nps^2 \left( \log (1/p)-1+p \right) \le (2-\epsilon) \log(n)$.
    In the sparse regime (see \cite[Section 5, Lemma 2]{wu2020testing}), it is possible to choose $\calE$ such that
    $
    \calP(\calE^c)=O(\frac{1}{n})
    $
    and $   \chi^2 \left( \calP_\calE \| \calQ \right) = o(1)$, 
    provided that $nps^2 \le 1-\epsilon$ and $np=\omega(1)$.

    By~\prettyref{eq:IABpi} and \prettyref{eq:IP-ER}, we have $\expect{\log \frac{d\calP_{ A,B \mid \pi}}{d\calQ_{A,B}}} = \binom{n}{2} I(P)$,
    where $ I(P) = q d(s\|q) + (1-q) d(\eta \|q)$, with $q=ps$ and $\eta = \frac{q(1-s)}{1-q}$. In both the dense and sparse regime, one can verify that 
    \[
    I(P) = s^2 p \pth{p-1+\log\frac{1}{p}}(1+o(1)).
    \]
    As a result, we have $\expect{\log \frac{d\calP_{ A,B \mid \pi}}{d\calQ_{A,B}}}=O(n \log n)$ in both cases.
    
    It remains to bound the variance in \prettyref{eq:trunc2MI}. 
    Note that 
    $$
    \log \frac{\calP( A,B \mid \pi)}{\calQ(A,B)}
    = \binom{n}{2} \log \frac{1-\eta}{1-ps} + 
    \frac{1}{2}  h(A,B, \pi),
    $$
    where 
    $$
    h(A, B, \pi) \triangleq \log \frac{1-s}{1-\eta} \left(  \norm{A}^2 + \norm{B}^2 \right)  +\log \frac{s(1-\eta)}{\eta(1-s)} \iprod{A^\pi}{B}.
    $$
    Since $p$ is bounded away from $1$ and $s=o(1)$ in both the dense regime ($p=n^{-o(1)}$) and sparse regime ($p=n^{-\Omega(1)}$ and $np=\omega(1)$), it follows that 
    \begin{align*}
        \log \frac{1-\eta}{1-s}  & = \log \left( 1 +  \frac{s(1-p) }{(1-s)(1-ps) } \right) =\left(1+o(1)\right) s (1-p)\\
        \frac{s(1-\eta)}{\eta(1-s)} &=  \frac{(1-\eta) (1-ps) }{ p (1-s)^2 } =\frac{1+o(1)}{p}. 
    \end{align*}
    Note that $\norm{A}^2,\norm{B}^2\sim\Binom(\binom{n}{2},ps)$ and 
    $ \iprod{A^\pi}{B} \sim\Binom(\binom{n}{2},ps^2)$. We have
    $\Var(h) = O(n^2 ps^2 \log^2 \frac{1}{p})$. Consequently, 
    $\sqrt{\Var(h)\calP(\calE^c)}=o(1)$   and 
    $O(\log n)$ in the dense and sparse case, respectively.
    Applying \prettyref{lmm:trunc2MI} yields the same upper bound on $D(\calP_{A,B}\|\calQ_{A,B})$.

    \subsection{Proof of~\prettyref{prop:I_MMSE}}
    The proof follows from a simple application of the I-MMSE formula
    for the additive Gaussian channel.
    %
    Note that under the interpolated model $\calP_\theta$, we have
    $$
    \frac{B}{\sqrt{1-\theta}}= \sqrt{\frac{\theta}{1-\theta}} A^\pi + Z,
    $$
    which is the output of an additive Gaussian channel with input $\Api$ and the standard Gaussian noise $Z$.
    Letting $I(\theta)=I(B; A^\pi \mid \pi)=I(A, B; \pi)$
    and using the I-MMSE formula~\cite[Theorem 2]{GuoShamaiVerdu05}, we have
    \begin{align}
        \frac{d I(\theta)}{d \left( \theta /(1-\theta) \right)}
        =\frac{1}{2} \mmse_{\theta}\left(\Api\right). \label{eq:I_MMSE_Gaussian}
    \end{align}
    Thus $\frac{d I(\theta)}{d \theta }
    =\frac{1}{2(1-\theta)^2 } \mmse_{\theta}\left(\Api\right)$. Integrating over $\theta$ from $0$ to $\rho^2$ and noting $I(0)=0$, we arrive at 
    $$
    I(\rho^2)= \int_{0}^{\rho^2}\frac{1}{2(1-\theta)^2 } \mmse_{\theta}\left(\Api\right) d\theta.
    $$

    \subsection{Proof of \prettyref{prop:I_MMSE_ER}}
    Note that under the interpolated model $\calP_\theta$, we have
    \begin{equation*}
        p_\theta(y|x) = 
        \begin{cases}
            \theta  &  x=1,y=1\\
            1-\theta & x=1,y=0\\
            \eta  &  x=0,y=1\\
            1-\eta & x=0,y=0\\
        \end{cases},
        \quad \eta = \frac{q(1-\theta)}{1-q}.
    \end{equation*}
    Let $g(\theta)\triangleq D(P_\theta\|Q)= q \cdot d(\theta\|q) + (1-q) d(\eta\|q)$. 
    Then $g(s)=D(P\|Q)$ and $g(q)=0$.
    Let $I(\theta)=I_\theta(A,B; \pi)$, where the subscript $\theta$ indicates that $A^\pi$ and $B$ are distributed according to $\calP_\theta$. 
    Then 
    \[
    I_s(A^\pi;B|A)
    = H_s(A^\pi|A) - H_s(A^\pi|B,A) = 
    H_{q}(A^\pi|B,A) - H_{s}(A^\pi|B,A)
    = - \int_q^s \frac{d}{d\theta}  H_\theta(A^\pi|B,A)
    d\theta,
    \]
    where the second equality holds because for a fixed $q$, $H_\theta(A^\pi | A)$ does not change with $\theta$
    and when $\theta=q$, $A^\pi$ and $B$ are independent
    and hence $I_q(A^\pi; B |A)=0$ so that $H_q(A^\pi |A)=H_q(A^\pi|B, A)$.
    By \cite[Lemma 7.1]{deshpande2015asymptotic}, we have
    \[
    \frac{d}{d\theta}  H_\theta(A^\pi|B,A) = \termI + \binom{n}{2} \frac{d}{d\theta} 
    \left( h(q) - g(\theta) \right) = \termI - \binom{n}{2} g'(\theta),
    \]
    where $h(q)=-q\log q - (1-q) \log (1-q)$ is the binary entropy function, 
    \[
    \termI = \sum_{e\in \binom{[n]}{2}} \sum_{x_e,y_e} \frac{\partial p_\theta(y_e|x_e)}{\partial \theta} \Expect \qth{\mu_e(x_e|B_{\backslash e},A) \log \sum_{x_e'} p_\theta(y_e|x_e') \mu_e(x_e'|B_{\backslash e},A)},
    \]
    $B_{\backslash e}$ denotes the adjacency vector $B$ excluding $B_e$, 
    and $\mu_e(\cdot \mid B_{\backslash e},A)$ is the distribution of $A_e^{\pi}$ conditional on $(B_{\backslash e}, A)$
    under $\calP_\theta$.
    Since $g(q)=0$, we have
    \begin{equation}
        I_s(A^\pi;B|\pi) = -\int_q^s \termI +\binom{n}{2} g(s) = -\int_q^s \termI + \binom{n}{2} D(P\|Q).
        \label{eq:I2}
    \end{equation}
    It remains to relate $\termI$ to the reconstruction error.
    Note that for $x,y\in\{0,1\}$,
    \begin{equation}
        p_\theta(y|x) = \alpha(x) y  + (1-\alpha(x)) (1-y), \quad \alpha(x) = \theta x+\eta(1-x)
        \label{eq:channel}
    \end{equation}
    and
    \[
    \frac{\partial p_\theta(1|x) }{\partial \theta} = \frac{\partial \alpha(x) }{\partial \theta}  = x + \frac{\partial \eta}{\partial \theta} (1-x) = - \frac{\partial p_\theta(0|x) }{\partial \theta}.
    \]
    Thus for each $x_e$,
    \begin{align*}
        & ~ \sum_{y_e=0,1} \frac{\partial p_\theta(y_e|x_e)}{\partial \theta} \Expect \qth{\mu_e(x_e|B_{\backslash e},A) \log \sum_{x_e'=0,1} p_\theta(y_e|x_e') \mu_e(x_e'|B_{\backslash e},A)}\\
        = & ~ \frac{\partial p_\theta(1|x_e)}{\partial \theta} \Expect \qth{\mu_e(x_e|B_{\backslash e},A) \log \frac{\sum_{x_e'=0,1} p_\theta(1|x_e') \mu_e(x_e'|B_{\backslash e},A)}{1-\sum_{x_e'=0,1} p_\theta(1|x_e') \mu_e(x_e'|B_{\backslash e},A)}  }\\
        = & ~ - \frac{\partial p_\theta(1|x_e)}{\partial \theta} \Expect \qth{\mu_e(x_e|B_{\backslash e},A) h'(y_e)},
    \end{align*}
    where we defined
    \[
    y_e \equiv y_e(B_{\backslash e},A) \triangleq \sum_{x_e=0,1}  p_\theta(1|x_e) \mu_e(x_e|B_{\backslash e},A).
    \]
    and used $h'(x) = \log \frac{1-x}{x}$. Then we have
    \begin{align*}
        & ~ \sum_{x_e=0,1} \sum_{y_e=0,1} \frac{\partial p_\theta(y_e|x_e)}{\partial \theta} \Expect \qth{\mu_e(x_e|B_{\backslash e},A) \log \sum_{x_e'=0,1} p_\theta(y_e|x_e') \mu_e(x_e'|B_{\backslash e},A)} = - \Expect \qth{  \frac{\partial y_e}{\partial \theta} h'(y_e)  }
    \end{align*}
    Let $\tilde x_e = \Expect[A^\pi_e |B_{\backslash e},A]$. Then
    $y_e = \theta \tilde x_e + \eta (1-\tilde x_e)$. Let 
    \[
    \Delta_e=y_e-q = (\theta-\eta) (\tilde x_e - q)  = \frac{\theta-q}{1-q} (\tilde x_e-q).
    \]
    Then $\Expect[\Delta_e]=0$ and $\Expect[\Delta_e^2]= (\frac{\theta-q}{1-q})^2 \Var(\tilde x_e)$.
    Furthermore,
    \[
    \frac{\partial y_e}{\partial \theta} = \frac{1}{1-q} (\tilde x_e-q).
    \]
    
    Using $h''(x)=-\frac{1}{x(1-x)}$, we get $h'(y_e) = h'(q) -\frac{1}{\xi(1-\xi)} \Delta_e$ for some $\xi$ between $y_e$ and $q$.
    Note that $\eta \leq q \leq \theta \leq s$ and $y_e \in [\eta,\theta]$. Thus $\eta \leq \xi \leq s$. 
    So
    \begin{align*}
        - \Expect \qth{  \frac{\partial y_e}{\partial \theta} h'(y_e)  } 
        = & ~  -  \frac{h'(q)}{1-q} \Expect \qth{\tilde x_e-q} + \frac{\theta-q}{(1-q)^2} \Expect \qth{\frac{(\tilde x_e-q)^2}{\xi(1-\xi)}}   \\
        \geq & ~ \frac{\theta-q}{s(1-q)^2} \Var_\theta(\tilde x_e).
    \end{align*}
    Integrating over $\theta$ we get
    \begin{align}
        \int_q^s \termI
        = & ~  \sum_e \int_q^s d\theta   \left(- \Expect \qth{  \frac{\partial y_e}{\partial \theta} h'(y_e)  }  \right) \\
        \geq & ~ \sum_e \int_q^s d\theta  \frac{\theta-q}{s(1-q)^2}  \Var_\theta(\tilde x_e)  \label{eq:int_ER}
    \end{align}
    
    Finally, note that the above bound pertains to $\tilde x_e=\Expect[A^\pi_e |B_{\backslash e},A]$, which we now relate to $\hat x_e=\Expect[A^\pi_e |B,A]$.
    Denote by $\mu_e(\cdot|B,A)$ the full posterior law of $A^\pi_e$. 
    Note that
    \[
    \hat x_e= \sum_{x_e}  x_e \mu_e(x_e|B,A) = \frac{\sum_{x_e} x_e \mu_e(x_e|B_{\backslash e},A) p_\theta(B_e|x_e) }{\sum_{x_e} \mu_e(x_e|B_{\backslash e},A) p_\theta(B_e|x_e)}
    \]
    By \prettyref{eq:channel}, we have
    \[
    p_\theta(y|x) = 1- y  + \left(\theta x+\eta(1-x) \right)(2y-1).
    \]
    After some simplification, we have
    \[
    \hat x_e= \sum_{x_e}  x_e  \mu_e(x_e|B,A) = \frac{\tilde x_e (\theta B_e + (1-\theta) (1-B_e))}{1-\eta-(1-2\eta) B_e + \tilde x_e (2B_e-1) ( \theta- \eta) }
    =\begin{cases}
        \frac{(1-\theta) \tilde x_e}{1-\eta - \tilde x_e (\theta-\eta) } & B_e=0 \\
        \frac{\theta \tilde x_e}{\eta + \tilde x_e (\theta-\eta) } & B_e=1 \\
    \end{cases}.
    \]
    Since $\eta\leq q\leq\theta\leq s$, we have
    \[
    \hat x_e \leq B_e \min\pth{1,\frac{s}{\eta} \tilde x_e}  + (1-B_e) \tilde x_e
    \]
    and hence
    \begin{align*}
        \Expect[\hat x_e^2] \leq & \expect{\min\pth{1,\frac{s}{\eta} \tilde x_e}^2 B_e} +  \Expect[\tilde x_e^2].
    \end{align*}
    Note that
    \begin{align*}
        \expect{\min\pth{1,\frac{s}{\eta} \tilde x_e}^2 B_e}
        \stepa{=} & ~ \expect{\expect{\min\pth{1,\frac{s}{\eta} \tilde x_e}^2 \Big| A^\pi_e} \Expect[B_e\mid A^\pi_e]} \\
        \stepb{=} & ~ \expect{\expect{\min\pth{1,\frac{s}{\eta} \tilde x_e}^2 \Big| A^\pi_e} (s A^\pi_e + \eta (1-A^\pi_e))} \\
        \stepc{\leq} & ~ s \expect{A^\pi_e} + s \expect{\expect{\tilde x_e\Big| A^\pi_e} (1-A^\pi_e)} \\
        \leq & ~ s q  + s \Expect[\tilde x_e] =2sq,
    \end{align*}
    where (a) follows from the conditional independence of $\tilde x_e$ (which depends on $(A,B_{\backslash e})$) and $B_e$ given $A^\pi_e$;
    (b) follows from \prettyref{eq:channel}; (c) follows by using $\min\pth{1,\frac{s}{\eta} \tilde x_e}^2 \le 1$
    to get the first term and $\min\pth{1,\frac{s}{\eta} \tilde x_e}^2 \le \frac{s}{\eta} \tilde x_e$ to get the second term.
    Combining the previous two displays yields that 
    \begin{align*}
        \Expect_\theta[\hat x_e^2]
        \le 2s q + \expect{ \tilde x_e^2}
        \le 2sq +q^2 + \Var\left(\tilde x_e^2 \right).
    \end{align*}
    It follows that 
    \begin{align}
        \mmse(A^\pi) = \sum_e \expect{ \left( x_e - \hat{x}_e \right)^2 } 
        =\sum_e \left(  \expect{x_e^2} - \expect{ \hat{x}^2_e} \right)
        \ge \binom{n}{2} q (1-q) - 2 \binom{n}{2} sq 
        - \sum_e \Var\left(\tilde x_e^2 \right)
        \label{eq:trivial_overlap_ER_1}. 
    \end{align} 
    Combining~\prettyref{eq:int_ER} with~\prettyref{eq:trivial_overlap_ER_1} yields that 
    \begin{align*}
        \int_q^s \termI 
        & \ge 
        \int_q^s d\theta  \frac{\theta-q}{s(1-q)^2}  \left( \binom{n}{2} q (1-q) - 2 \binom{n}{2} sq - \mmse_\theta(\Api) \right) d\theta \\
        & = \int_q^s d\theta  \frac{\theta-q}{s(1-q)^2}  \left( \binom{n}{2} q (1-q)  - \mmse_\theta(\Api) \right) d\theta
        - \binom{n}{2} \frac{(s-q)^2q}{(1-q)^2} \\
        & \ge \int_q^s d\theta  \frac{\theta-q}{s(1-q)^2}  \left( \binom{n}{2} q (1-q)  - \mmse_\theta(\Api) \right) d\theta
        - \binom{n}{2} qs^2
    \end{align*}
    The conclusion follows by combining the last display with~\prettyref{eq:I2}.

    \subsection{Proof of~\prettyref{prop:graph_to_permutation}}
    In this section, we prove~\prettyref{prop:graph_to_permutation}
    by connecting the MSE of $\Api$ to the Hamming risk of estimating $\pi$.
    In particular, assuming~\prettyref{eq:triv_mmse_cond}, that is,
    $\mmse(\Api)\ge\expect{\norm{A}^2} (1-\xi)$, we aim to show that 
    $\expect{\overlap(\pi, \hat{\pi})} =O\left(\xi^{1/4} + (\frac{n\log n}{\expect{\norm{A}^2}} )^{1/4} \right)$
    for any estimator $\hat{\pi}(A,B)$. 
    We first present a general program and then specialize the argument to 
    the Gaussian   and \ER graph model. 
    
    Recall that $\overlap(\pi, \hat{\pi})$ denotes the fraction of fixed points of $\sigma\triangleq  \pi^{-1} \circ \hat{\pi}$. 
    Let $\alpha(\pi, \hat{\pi})$ denote the fraction of fixed points of the edge permutation $\sigmae$ induced by the node permutation $\sigma$ (cf.~\prettyref{sec:orbit}).
    The following simple lemma relates $\alpha(\pi, \hat{\pi})$ to $\overlap(\pi, \hat{\pi})$. 
    
    \begin{lemma}\label{lmm:edge_to_node_fixed}
        It holds that
        $$
        \expect{ \overlap(\pi,\hat{\pi}) }  \le \sqrt{\expect{\alpha(\pi,\hat{\pi})}} + \frac{1}{n}.
        $$
    \end{lemma}
    \begin{proof}
        In view of~\prettyref{eq:fixed_edge_node},
        $$
        \binom{n\overlap(\pi,\hat{\pi})}{2} \le \binom{n}{2} \alpha (\pi,\hat{\pi}).
        $$
        By Jensen's inequality, 
        $$
        \binom{n \expect{ \overlap(\pi,\hat{\pi})} }{2}
        \le \expect{\binom{n\overlap(\pi,\hat{\pi})}{2} } \le 
        \binom{n}{2}  \expect{\alpha (\pi,\hat{\pi})}.
        $$
        The desired conclusion follows because 
        for $x,y\ge 0$, $\binom{nx}{2} \le \binom{n}{2} y$ $\iff$
        $nx^2 - x - (n-1) y \le 0$ $\implies$ $x \le \frac{1+\sqrt{1+4n(n-1)y}}{2n} \le \sqrt{y} + \frac{1}{n}$.
    \end{proof}
    
    In view of~\prettyref{lmm:edge_to_node_fixed} and the fact that $\expect{\norm{A}^2}\le n^2$, it suffices to show
    $\expect{\alpha(\pi, \hat{\pi})} =O\left(\xi^{1/2} + (\frac{n\log n}{\expect{\norm{A}^2}} )^{1/2} \right)$.
    Let $\alpha_0=\expect{\alpha(\pi,\hat{\pi})}$
    and define an estimator of $A^\pi$ by  $\hat{A}=\alpha_0 A^{\hat{\pi}}+(1-\alpha_0) \expect{A}$.
    This is well-defined since $\alpha_0$ is deterministic and $\hat{\pi}$ only depends on $(A,B)$.
    Intuitively,  $\hat{A}$ can be viewed as an interpolation between the ``plug-in'' estimator $A^{\hat{\pi}}$
    and the trivial estimator $\expect{A^\pi} = \expect{A}$. We remark that to derive the desired lower bound to $\alpha_0$, 
    it is crucial to use the interpolated estimator $\hat{A}$
    rather than the ``plug-in'' estimator, because we expect $\alpha_0$ to be small and $\hat{\pi}$ is only slightly
    correlated with $\pi$. 

    On the one hand, by definition of the MMSE and the assumption~\prettyref{eq:triv_mmse_cond},
    \begin{align}
        \expect{\norm{\Api-\hat{A}}^2} \ge
        \mmse(\Api) \ge \expect{ \norm{A}^2}  (1-\xi). \label{eq:pi_to_A_Gaussian_0}
    \end{align}
    On the other hand, we claim that in both the Gaussian   and \ER model,
    \begin{align}
        \expect{\iprod{\Api}{A^{\hat{\pi}}}}\ge \expect{\|A\|_2^2} 
        \alpha_0 - O\left( \sqrt{ \expect{\|A\|_2^2} n \log n} \right),
        \label{eq:pi_to_A_Gaussian_2}
    \end{align} 
    so that 
    \begin{align}
        \expect{\norm{\Api-\hat{A}}^2}
        & = \expect{\norm{\Api}^2} + \expect{ \norm{\hat{A} }^2} - 2 \expect{\iprod{\Api}{\hat{A}}} \nonumber \\
        & \stepa{=}  (1+\alpha_0^2) \expect{\norm{A}^2} - (1-\alpha_0)^2
        \norm{\expect{A}}^2  - 2 \alpha_0 \expect{\iprod{\Api}{A^{\hat{\pi}}}}  \nonumber \\
        & \le (1-\alpha_0^2)\expect{\norm{A}^2}
        +O\left(\alpha_0 \sqrt{ \expect{\|A\|_2^2} n \log n }\right),
        \label{eq:pi_to_A_Gaussian}
    \end{align}
    where in (a) we used the fact that $\expect{\Api}=\expect{A}$ is entrywise constant and 
    $\iprod{\expect{A}}{A^{\hat \pi}} = \Expect[A_{12}] \sum_{i<j} A_{\hat \pi(i),\hat \pi(j)} = \Expect[A_{12}] \sum_{i<j} A_{i,j} = \iprod{\expect{A}}{A}$    
    so that $\iprod{\expect{A}}{\Expect[A^{\hat \pi}]} = \|\Expect[A]\|^2$.
    Combining \prettyref{eq:pi_to_A_Gaussian_0} and~\prettyref{eq:pi_to_A_Gaussian}  yields
    that  
    \begin{align}
        \alpha_0^2  \le \xi + O\left( \alpha_0  \sqrt{ \frac{n  \log n}{  \expect{ \norm{A}^2} }} \right) \Longrightarrow \alpha_0 =O\left( 
        \xi^{1/2} +  \sqrt{ \frac{n  \log n}{  \expect{ \norm{A}^2} }}\right).
        \label{eq:trivial_overlap_gaussian} 
    \end{align}
    To finish the proof, it remains to prove the claim~\prettyref{eq:pi_to_A_Gaussian_2}. 
    
    \begin{proof}[Proof of Claim~\prettyref{eq:pi_to_A_Gaussian_2}]
        Let $C$ be a sufficiently large enough constant.
        For each permutation $\pi' \in \calS_n$, define an event 
        \begin{equation}
            \calF_{\pi'}=\sth{
                \Iprod{\Api}{A^{\pi'}} \geq \expect{\norm{A}^2} 
                \alpha(\pi,\pi') - C \sqrt{\expect{\norm{A}^2}  n \log n }}
            \label{eq:calFpi}
        \end{equation}
        and set $\calF \triangleq \cap_{\pi' \in \calS_n} \calF_{\pi'}$. 
        It follows that  
        \begin{align*}
            \expect{\iprod{\Api}{A^{\hat{\pi}}}} &= \expect{\iprod{\Api}{A^{\hat{\pi}}} \indc{\calF}}+ \expect{\iprod{\Api}{A^
                    {\hat{\pi}}} \indc{\calF^c}} \\
            & \ge \expect{\norm{A}^2}  \expect{\alpha(\pi,\hat{\pi}) \indc{\calF}} - C \sqrt{ \expect{\norm{A}^2} n \log n}
            - \expect{\norm{A}^2 \indc{\calF^c}} \\
            & \ge \expect{\norm{A}^2} \left( \alpha_0 -\prob{\calF^c}  \right)
            -  C \sqrt{\expect{\norm{A}^2} n \log n} 
            -\sqrt{ \expect{\norm{A}^4 } \prob{\calF^c}},
        \end{align*}
        where the last inequality holds because 
        $$
        \expect{\alpha(\pi,\hat{\pi}) \indc{\calF}}=
        \expect{\alpha(\pi,\hat{\pi}) }- \expect{\alpha(\pi,\hat{\pi}) \indc{\calF^c}}
        \ge \alpha_0- \prob{\calF^c} ,
        $$
        and 
        $
        \expect{\norm{A}^2 \indc{\calF^c}} \le \sqrt{ \expect{\norm{A}^4 } \prob{\calF^c}}
        $
        by the Cauchy-Schwarz inequality. 
        Note that $\expect{\norm{A}^4} = O(n^4)$, and $\expect{\norm{A}^2}$ is equal to $\binom{n}{2}$ in the Gaussian
        case and $\binom{n}{2} q$ in the \ER case (with $q \geq n^{-O(1)}$). 
        To get~\prettyref{eq:pi_to_A_Gaussian_2},
        it suffices to prove $\prob{\calF^c} \le e^{-n\log n}$,
        which, by union bound, further reduces to showing that $\prob{\calF_{\pi'}^c} \le e^{-2n\log n}$ 
        for any permutation $\pi' \in \calS_n$. 
        To this end, we consider the Gaussian   and \ER graph model separately.
        
        For the Gaussian Winger model,
        let $M \in \{0,1\}^{\binom{[n]}{2} \times \binom{[n]}{2}}$ denote the permutation matrix corresponding to the edge permutation $\sigmae$ induced by
        $\sigma=\pi^{-1}\circ \pi'  $. 
        Recalling that $A$ denotes the weighted adjacency vector, we have 
        $
        \Iprod{\Api}{A^{\pi'}} = A^\top M A.
        $
        By Hanson-Wright inequality~\prettyref{lmm:hw}, with probability at least $1-\delta$
        $$
        A^\top M A  \ge \Tr(M) - C' \left( \|M\|_F^2 \sqrt{\log (1/\delta)} + \|M\|_2 \log (1/\delta) \right),
        $$
        where $C'>0$ is a universal constant. 
        Since $\Tr(M)=\binom{n}{2} \alpha(\pi,\pi')$, $\|M\|_F^2=\binom{n}{2}$, and the spectral norm $\|M\|_2=1$,
        it follows that with probability at least $1-\delta$,
        $$
        A^\top M A  \ge \binom{n}{2} \alpha(\pi,\pi') - C'\left( n \sqrt{\log (1/\delta)} +  \log (1/\delta) \right).
        $$
        Choosing $\delta=e^{-2n\log n}$ and $C$ in \prettyref{eq:calFpi} to be a large enough constant, we get that 
        $
        \prob{\calF_{\pi'}} \ge 1-e^{-2n\log n}. 
        $
        Next, we move to the \ER graph model. 
        Fix any permutation $\pi' \in \calS_n$.
        Let $\orbit_1$ denote the set of fixed points of the edge permutation induced by 
        $\pi'\circ \pi^{-1}$. By definition, $|\orbit_1|=\binom{n}{2}\alpha(\pi,\pi') $
        and 
        $$
        \iprod{\Api}{A^{\pi'}}
        \ge \sum_{(i,j) \in \orbit_1} A_{ij} \sim \Binom \left(\binom{n}{2}\alpha(\pi,\pi'), q\right).
        $$
        By Bernstein's inequality, with probability at least $1-\delta$,
        \begin{align*}
            \iprod{\Api}{A^{\pi'}}
            & \ge \binom{n}{2}\alpha(\pi,\pi')q - C' \left( \sqrt{\binom{n}{2} \alpha(\pi,\pi')q \log (1/\delta) } 
            + \log (1/\delta) \right) \\
            & \ge \binom{n}{2}\alpha(\pi,\pi')q - C' \left( \sqrt{n^2 q \log (1/\delta) } 
            + \log (1/\delta) \right), 
        \end{align*}
        where $C'>0$ is a universal constant. 
        Choosing $\delta=e^{-2n\log n}$ 
        and $C$ to be a large enough constant, we get that 
        $
        \prob{\calF_{\pi'}} \ge 1-e^{-2n\log n}. 
        $
    \end{proof}

    \section{Possibility of Partial and Almost Exact Recovery}
    In this section, we prove the positive parts of \prettyref{thm:ER_dense} and \prettyref{thm:ER_sparse}.
    
    We first argue that we can assume $ps \le 1/2$ without loss of generality. 
    Note that the existence/absence of an edge is a matter of representation and they are mathematically equivalent. As a consequence, by flipping $0$ and $1$, the model with parameter $(n,p,s)$ is equivalent to that with parameter $(n, p', s')$ for an appropriate choice of $p'$ and $s'$ such that $p's'=1-ps$, 
    where   
          \begin{align}
          p'  = \frac{\left( 1-ps\right)^2}{1-2ps+ps^2} ,
          \quad      s' & = \frac{1-2ps+ps^2}{1-ps} 
          \, .  \label{eq:equivalent_model}
         \end{align}
    Therefore when $ps > \frac{1}{2}$, we can replace the model $(n,p,s)$ with $(n, p', s')$ so that $p's' = 1-ps < \frac{1}{2}$. 
   
   \medskip
    For any two permutations $\pi, \hpi\in \calS_n$, let $d(\pi, \hpi)$ denote the number of non-fixed points in the $ \pi' \circ \pi^{-1}$. 
    The following proposition provides sufficient conditions for $\piML$ defined in \prettyref{eq:ML}
    to achieve the partial recovery and almost exact recovery in \ER graphs.
    \begin{proposition}
    \label{prop:possibility_almost}
         Let $ps \le \frac{1}{2}$. If $p = 1-o(1)$, 
        \begin{align}
             \frac{n s^2(1-p)^2}{(1-ps)^2} \ge (4+\epsilon) \log n \label{eq:recovery_positive_cond_2}  \, ,
        \end{align}
        or if $p = 1-\Omega(1)$, 
        \begin{align}
            nps^2 \ge 
            \begin{cases}
                \frac{(2+ \epsilon ) \log n}{\log (1/p)-1+p } & \text{ if } p \ge n^{-\frac{1}{2}} \\
                4+ \epsilon & \text{ if } p < n^{-\frac{1}{2}}  
            \end{cases} \label{eq:recovery_positive_cond_1}\, ,
        \end{align}
        for any arbitrarily small constant $\epsilon >0$,
        there exists a constant $0<\delta<1$ such that 
        $$
        \prob{ d\left(\piML, \pi\right) < \delta n} \ge 1-n^{-1+o(1)},
        $$
        that is, $\piML$ achieves partial recovery.
        
        If in addition $nps^2(1-p)^2 = \omega(1) $, then for any constant $\delta>0$,
        $$
        \prob{ d\left(\piML, \pi\right) < \delta n} \ge 1-n^{-1+o(1)},
        $$
        that is, $\piML$ achieves almost exact recovery.
    \end{proposition}

\begin{remark}
We explain how to 
prove the positive parts of \prettyref{thm:ER_dense} and \prettyref{thm:ER_sparse}
using~\prettyref{prop:possibility_almost}.
\begin{itemize}
    \item In the dense regime of $p=n^{-o(1)}$,
    either  \prettyref{eq:recovery_positive_cond_2} or  \prettyref{eq:recovery_positive_cond_1} already implies that $nps^2(1-p)^2 =\omega(1) $
    and hence the MLE achieves almost exact recovery under the condition \prettyref{eq:recovery_positive_cond_2} when $p= 1-o(1)$ or  $nps^2 \ge \frac{(2+ \epsilon ) \log n}{\log (1/p)-1+p }$ when $p= 1-\Omega(1)$; this proves the positive part of \prettyref{thm:ER_dense}. 
    \item In the sparse regime of $p=n^{-\Omega(1)}$, condition~\prettyref{eq:partial-sparse-positive} implies~\prettyref{eq:recovery_positive_cond_1} and hence the MLE achieves partial recovery; this proves the positive part of \prettyref{thm:ER_sparse}.
    Furthermore,  
    since $nps^2=\omega(1)$ implies~\prettyref{eq:recovery_positive_cond_1}, 
    the MLE achieves the almost exact recovery provided that $nps^2=\omega(1)$, which is in fact needed for any estimator to succeed \cite{cullina2019partial}. 
\end{itemize}
  \end{remark}

    To prove \prettyref{prop:possibility_almost}, we need the following intermediate lemma, which bounds the probability that the ML estimator \prettyref{eq:ML} makes a given number of errors.
    \begin{lemma}
    \label{lmm:ER_fixing_k}
      Let  $\epsilon \in (0,1)$ be an arbitrary constant and $ps \le \frac{1}{2}$. 
        \begin{itemize}
            \item For the \ER model, suppose that either \prettyref{eq:recovery_positive_cond_2} or \prettyref{eq:recovery_positive_cond_1} holds. Then there exists some constant $0< \delta <1$ such that for any $k\ge \delta n$,
        \begin{align}
			& \prob{ d\left(\piML, \pi\right)  = k}
			 \le  2 \exp\left( -  n h \left( \frac{k}{n}  \right) \right) \indc{k \le n-1} \nonumber\\
			&~~~~ + 
			e^{-2\log n} \indc{k=n}  + \exp\left( - \frac{1}{64}  \epsilon k \log n\right),
			\label{eq:ER_fixing_k}
		\end{align}
        where $h(x)=  - x \log x  - (1-x) \log (1-x)$ is the binary entropy function.
        
        If in addition $ n ps^2(1-p) = \omega(1) $, then \prettyref{eq:ER_fixing_k} holds
        for any constant $0<\delta<1 $ and all $k \ge \delta n.$

        
            \item For the Gaussian model, suppose that $n\rho^2 \ge (4+\epsilon) \log n$.  Then \prettyref{eq:ER_fixing_k} holds for any constant $0<\delta<1 $ and all $k \ge \delta n$.
        \end{itemize}
    \end{lemma}
    Note that \prettyref{lmm:ER_fixing_k} also includes the
    Gaussian case. In fact, analogous to~\prettyref{prop:possibility_almost}, 
    we can apply~\prettyref{lmm:ER_fixing_k} to 
    show that the MLE attains almost exact recovery 
    when $n\rho^2 \ge (4+\epsilon)\log n$. We will not do it here; instead in the next section, we will directly prove
    a stronger result, showing that the  MLE attains exact recovery
    under the same condition.

    Now, we proceed to prove \prettyref{prop:possibility_almost}.
    \begin{proof}[Proof of \prettyref{prop:possibility_almost}]
        Applying \prettyref{lmm:ER_fixing_k} with a union bound yields that 
        \begin{align}
            \prob{ d\left(\piML, \pi\right) \ge \delta n } 
            & \le \sum_{k \ge \delta  n}^{n }   \prob{ d\left(\piML, \pi\right)  = k} \nonumber \\
            & \le \exp\left(-2 \log n\right) + 2 \sum_{k \ge \delta  n}^{n-1 } 
            \exp\left( -  n h \left( \frac{k}{n}  \right) \right) +  
            \sum_{k \ge \delta  n } \exp\left( - \frac{1}{64}  \epsilon k \log n\right) \nonumber \\
            & = n^{-1+o(1)}, \label{eq:summation_k}
        \end{align}
        where the last inequality follows from 
        $\sum_{k \ge \delta  n } \exp\left( - \frac{1}{64}  \epsilon k \log n\right)
        \le  \frac{\exp\left( - \frac{\epsilon }{64} \delta  n \log n\right)}{1- \exp \left( -\frac{\epsilon}{64} \log n \right) }
        =n^{-\Omega(n)}$ for any fixed constant $\delta >0$, and 
        \begin{align*}
            \sum_{k \ge 1}^{n-1 } \exp\left( -  n h \left( \frac{k}{n}  \right) \right)
            \stepa{\leq} & ~ 2 \sum_{1 \leq k \leq n/2} \exp\left( -  k \log \frac{n}{k} \right) \\
            \leq & ~  2 \sum_{k=1}^{10 \log n} \exp\left( -  k \log \frac{n}{k} \right) +  2\sum_{10 \log n \leq k \leq n/2} 2^{ -  k} \\
            \leq & ~ 2  e^{-\log n} \times 10 \log n  + 4\times 2^{-10\log n} =n^{-1+o(1)},
        \end{align*}    
        where (a) follows from $h(x)=h(1-x)$ and $h(x) \geq x \log \frac{1}{x}$. 
    \end{proof}

    \subsection{Proof of \prettyref{lmm:ER_fixing_k}}
    \label{sec:pf-ER_fixing_k}
    Without loss of generality, we assume $\epsilon<1$. 
    Fix $k\in[n]$.  Let $\calT_k$ denote the set of permutations $\hpi$ such that $d(\pi, \hpi)=k$. 
    Recall that  $F$ is the set of fixed points of $\sigma \triangleq   \pi^{-1} \circ \hpi  $ with $|F|=n-k$ and $\calO_1=\binom{F}{2}$ is a subset of fixed points of edge permutation $\sigmae$. Let $\overline{A} = (\overline{A}_{ij})_{1\le i<j\le n} = (A_{ij}-\expect{A_{ij}})_{1\le i<j\le n}$ and $\overline{B} =  (\overline{B}_{ij})_{1\le i<j\le n} = (B_{ij}-\expect{B_{ij}})_{1\le i<j\le n}$ denote the centered adjacency vectors of $A$ and $B$ respectively. 
    It follows from the definition of MLE in \prettyref{eq:ML} that for
    any $\tau \in \reals$,
    \begin{align*}
        \left\{ d\left(\piML, \pi\right)  = k \right\} 
        &  \subset \left\{ \exists \hpi \in \calT_k : \left \langle A^{\pi}, B \right \rangle - \left \langle A^{\pi'}, B \right \rangle\le 0 \right\} \\
        &  \overset{(a)}{\subset} \left\{ \exists \hpi \in \calT_k : \sum_{i<j}  \overline{A}_{\pi(i)\pi(j)} \overline{B }_{ij}- \sum_{i<j}  \overline{A}_{\hpi(i)\hpi(j)} \overline{B}_{ij} \le 0 \right\} \\
        &= \left\{ \exists \hpi \in \calT_k : \sum_{(i,j) \notin \calO_1 } \overline{A}_{\pi(i)\pi(j)} \overline{B}_{ij}  -  \sum_{(i,j) \notin \calO_1 } \overline{A}_{\hpi(i)\hpi(j)} \overline{B}_{ij}  \le 0 \right\} \\
        & \subset \left\{ \exists \hpi \in \calT_k : \sum_{(i,j) \notin \calO_1 } \overline{A}_{\pi(i)\pi(j)}\overline{B}_{ij} < \tau\right\}
        \cup \left\{ \exists \pi \in \calT_k : \sum_{(i,j) \notin \calO_1 }  \overline{A}_{\hpi(i)\hpi(j)}\overline{B}_{ij} \ge \tau \right\},
    \end{align*}
    where $(a)$ holds because $\left \langle A^{\pi},B\right \rangle - \left \langle A^{\pi'},B\right \rangle = 
    \sum_{i<j}  \overline{A}_{\pi(i)\pi(j)} \overline{B }_{ij}- \sum_{i<j}  \overline{A}_{\hpi(i)\hpi(j)} \overline{B}_{ij}$. 
    Note that 
    $$
    \left\{ \exists \hpi \in \calT_k : \sum_{(i,j) \notin \calO_1 } \overline{A}_{\pi(i)\pi(j)} \overline{B}_{ij} < \tau\right\}
    =\left\{ \exists F \subset [n]: |F|=n-k, \; \sum_{(i,j) \notin \binom{F}{2} } \overline{A}_{\pi(i)\pi(j)} \overline{B}_{ij} < \tau \right\}.
    $$
    Thus, by the union bound,
    \begin{align*}
        \prob{ d\left(\piML, \pi\right)  = k} 
        & \le  \sum_{F\subset [n]: |F|=n-k} \prob{\sum_{(i,j) \notin \binom{F}{2} } \overline{A}_{\pi(i)\pi(j)}  \overline{B}_{ij} < \tau }  + \sum_{\hpi \in \calT_k} \prob{ 
            \sum_{(i,j) \notin \calO_1 }  \overline{A}_{\hpi(i)\hpi(j)} \overline{B}_{ij} \ge \tau }. 
    \end{align*}
    Let $Y = \sum_{(i,j) \notin \calO_1 } \overline{A}_{\pi(i)\pi(j)} \overline{B}_{ij} $ and  $X_\hpi=\sum_{(i,j) \notin \calO_1 }  \overline{A}_{\hpi(i)\hpi(j)}\overline{B}_{ij} $.
    Then
    \begin{align}
        \prob{ d\left(\piML, \pi\right)  = k} 
        \le \binom{n}{k} \prob{Y \le  \tau }  + \sum_{\hpi \in \calT_k} \prob{X_\hpi \ge \tau} \triangleq  \termI + \termII.  \label{eq:two_terms}
    \end{align}

    For $\termI$, we claim that for both the \ER random graph and Gaussian   model and appropriately  choice of $\tau$,
    \begin{align}
        \termI \le  \exp \left( - n h\left(\frac{k}{n} \right)\right) + \exp \left( -2 \log n\right) \indc{k=n} \label{eq:boundtermI} \, .
    \end{align}
    
    For $\termII$, applying the Chernoff bound together with the union bound yields that for any $t>0$,
    \begin{align}
        \termII
        \le \sum_{\hpi \in \calT_k } \prob{ X_\hpi \ge  \tau} 
        \le n^k \exp\left({-t \tau }\right) \expect{\exp\left({t X_\hpi}\right)}. \label{eq:large_k_upperbound_0}
    \end{align}
    Note that
    $$
    X_\hpi= \sum_{O\in  \calO \backslash \calO_1} \underbrace{\sum_{(i,j)\in O} \overline{A}_{\hpi(i)\hpi(j)} \overline{B}_{ij}}_{\triangleq \overline{X}_O},
    $$
    where $\calO$ denotes the collection of all edge orbits of $\sigma$.
    Since edge orbits are disjoint, it follows that $\overline{X}_O$ are mutually independent across different $O$. Therefore,
    \begin{align}
        \expect{ \exp\left({t X_\hpi}\right) } = \prod_{O\in  \calO \backslash \calO_1} \expect{\exp\left(t \overline{X}_O \right)}.\label{eq:x_pi_MGF}   
    \end{align}
    It turns out that the MGF of $\overline{X}_O$ can be explicitly computed for both \ER random graph and Gaussian   model.
    In particular, applying \prettyref{lmm:expect_exp_AB_pi} yields that $
    \expect{e^{t \overline{X}_O }} = M_{|O|},
    $
    and $ M_\ell \le M_2^{\ell/2} $ for $\ell \ge 2$. 
    It follows that         
    $$
    \expect{ \exp\left({t X_\hpi}\right) }
    \stepa{=} M_1^{n_2} \prod_{\ell=2}^{\binom{n}{2}} M_\ell^{N_\ell}
    \stepb{\le}  M_1^{n_2} \prod_{\ell=2}^{\binom{n}{2}} M_2^{\ell N_\ell /2}
    \stepc{=} \left( \frac{M_1^2}{M_2}\right)^{\frac{1}{2}n_2} M_2^{\frac{\binom{n}{2}-\binom{n_1}{2}}{2}}
    \stepd{=} \left( \frac{M_1^2}{M_2}\right)^{\frac{k}{2}} M_2^{\frac{m}{2}} \, , 
    $$
    where (a) follows from \prettyref{eq:x_pi_MGF} and $N_1 = \binom{n_1}{2}+n_2$ in view of \prettyref{eq:fixed_edge_node};
    (b) follows from $M_\ell \le M_2^{\ell/2}$ for $\ell \ge 2$;
    (c) follows from $\sum_{\ell = 1}^{\binom{n}{2}} \ell N_\ell = \binom{n}{2}$;
    (d) follows from $n_2\le k$, $n_1=n-k$, and $m \triangleq \binom{n}{2} - \binom{n-k}{2}$.
    Combining the last displayed equation with \prettyref{eq:large_k_upperbound_0} yields that 
    \begin{align*}
        \termII 
        & \le \exp \left( k \log n -t \tau + \frac{k}{2} \log  \frac{M_1^2}{M_2} + \frac{m}{2} \log M_2  \right) \\
        & \overset{(a)}{\le}  \exp \left( k \log n -t \tau + \frac{k}{2} \log  \frac{M_1^2}{M_2} + \frac{m}{2} \left( M_2 -1\right) \right) 
       \overset{(b)}{\le}  \exp \left( -\frac{\epsilon}{64} k \log n \right),
    \end{align*}
    where $(a)$ holds because $\log(1+x) \le x$ for any $x > -1$; $(b)$ follows from the claim that 
    \begin{align}
        \inf_{t\ge 0} 
        \left\{ -t \tau + \frac{k}{2} \log  \frac{M_1^2}{M_2} + \frac{m}{2} \left(M_2-1\right) \right\} 
        \le -\left( 1+ \frac{\epsilon}{64} \right) k \log n.
        \label{eq:boundtermII}
    \end{align}

    It remains to specify $\tau$ and 
    verify~\prettyref{eq:boundtermI} and~\prettyref{eq:boundtermII}
    for \ER and Gaussian models separately.
        In the following, we will use the inequality 
         \begin{align}
            \frac{1}{2}\left(1-\frac{1}{n}\right) kn\le m \le kn  \, ,  \label{eq:m}
         \end{align}
         which follows from  $ m = \binom{n}{2}- \binom{n-k}{2} = \frac{1}{2} kn \left(2- \frac{k+1}{n}\right)$.

    \begin{itemize}
        \item For \ER random graphs, $\overline{A} = (A_{ij} - q)_{1\le i<j \le n}$ and $\overline{B} = (B_{ij} - q )_{1\le i<j \le n}$, where $q = ps$. 
        Then for any $F \subset [n]$ with $|F|=n-k$, 
        $$
        Y = \sum_{(i,j) \notin \binom{F}{2} } (A_{\pi(i)\pi(j)}-q ) (B_{ij}-q) \, .
        $$
        Note that  $\expect{(A_{\pi(i)\pi(j)}-q ) (B_{ij}-q)} = ps^2(1-p)$ any $i<j$. 
        
        Let $\mu \triangleq \expect{Y}  = m  ps^2(1-p)$ and  set $\tau=(1-\gamma)\mu$ where
               
          \[
        \gamma \triangleq  \begin{cases}
            \sqrt{\frac{16 h(k/n)}{k ps^2(1-p)^2}} & k \le n-1 \\
            \sqrt{\frac{16 \log n}{n(n-1) ps^2(1-p)^2}} & k= n 
        \end{cases}  \, .
        \]
        We next choose a constant $0<\delta<1$  such that $ h\left(\delta\right)/\delta \le \frac{\epsilon^2}{2^{12}}nps^2 (1-p)^2$. Since $h(x)/x$ is monotone decreasing in $x$ and converges to $0$ as $x \to 1$, 
        under the condition~\prettyref{eq:recovery_positive_cond_2} or \prettyref{eq:recovery_positive_cond_1}, there exists some $0<\delta<1$ such that $ h\left(\delta\right)/\delta \le \frac{\epsilon^2}{2^{12}}nps^2(1-p)^2$. If further $nps^2(1-p)^2 = \omega(1)$
        then for any constant $0< \delta<1$, $h\left(\delta\right)/\delta \le \frac{\epsilon^2}{2^{12}}nps^2(1-p)^2$. 
        Hence, for all $\delta n \le k \le n-1$,  we have  $\frac{n}{k} h\left(\frac{k}{n}\right) \le  h\left(\delta\right)/\delta \le \frac{\epsilon^2}{2^{12}}nps^2(1-p)^2$, and then $\gamma \le  \sqrt{\frac{16 h(\delta)/\delta}{n ps^2(1-p)^2}} < \frac{\epsilon}{16}$; if $k=n$, since $nps^2(1-p)^2  = \Omega(1)$  under the condition~\prettyref{eq:recovery_positive_cond_2} or \prettyref{eq:recovery_positive_cond_1},
        $\gamma < \frac{\epsilon}{16} $ for all sufficiently large $n$. In conclusion, we have $\gamma< \frac{\epsilon}{16}$. 

        Applying the Bernstein inequality yields
        \begin{align}
        \prob{Y \le \tau} 
        \overset{(a)}{\le} \exp\left(-\frac{ \frac{1}{2}\gamma^2 \mu^2}{mps^2 +\frac{1}{3} \gamma \mu } \right) 
        & \overset{(b)}{\le} \exp\left(- \frac{1}{2+\frac{\epsilon}{24}}\gamma^2 \mu \left(1-p\right)\right)  \nonumber \\
        & \le \exp\left(-\frac{1}{4} \gamma^2 \mu (1-p)\right)\, ,  \label{eq:Y_tau}
        \end{align}
        where $(a)$ holds because for any $i<j$
        \begin{align*}
        \expect{(A_{\pi(i)\pi(j)}-q)^2 (B_{ij}-q)^2} 
        & = ps^2 \left(1-4ps+2p^2s+4p^2s^2-3p^3s^2\right) \\
        & \le ps^2 \left(1-2ps+4p^2s^2 \right) \le ps^2 \, ,
        \end{align*} 
        in view of $ps\le \frac{1}{2}$;
        $(b)$ holds because $mps^2 = \frac{\mu }{ 1-p}$ and 
        $\gamma < \frac{\epsilon}{16} $.
        By \prettyref{eq:m}, 
        we have
        \begin{align*}
            \gamma^2 \mu  (1-p)
            \ge \begin{cases}
                8 (n-1) h(k/n). & k \le n-1 \\
                8 \log n  & k = n  
            \end{cases} \, .
        \end{align*}
        Finally, applying \prettyref{eq:Y_tau} together with $\binom{n}{k} \le e^{n h\left(\frac{k}{n}\right)}$, 
        we arrive at the desired claim~\prettyref{eq:boundtermI}. 
        
       
        Next, we verify~\prettyref{eq:boundtermII} by choosing $t>0$ appropriately. Recall that
        $q = ps $. 
        In view of  \prettyref{eq:ER_expect_exp_AB_pi} in \prettyref{lmm:expect_exp_AB_pi}, 
         $M_2$ depends on $t$ and $q$. When $q$ is small and $p$ is bounded away from $1$, we will choose $t \le \log \frac{1}{p}$; otherwise, we will choose $t \le \frac{1}{2^{10}}$. As such, we bound $M_2$ using its Taylor expansion around $q=0$ and $t=0$, respectively, resulting in the following lemma. 
         \begin{lemma} \label{lmm:M_2}
        Assume $ps \le \frac{1}{2}$. 
        \begin{itemize}
            \item If $t \le \log \frac{1}{p}$, then $\frac{M_1^2}{M_2} \le e^2$, and
                \begin{align}
                    M_2 
                    & \le 1 
                     + q^2s^2 - 2q^2   + 10 q^3 t(1+t)  + 2 e^{t}q^2 (1-s^2) + e^{2t}  q^2s^2 
                   \, . \label{eq:M2_q_small}
        \end{align}
        \item If $t \le \frac{1}{2^{10}}$, then $\frac{M_1^2}{M_2} \le e^2$,  and
         \begin{align}
           M_2 \le 1 + t^2 \sigma^4 \left(1+ \rho^2 \right) +  8 t^3qs    \, , \label{eq:M2_t_small}
         \end{align}
         where $\rho = \frac{s-q}{1-q}$ and  $\sigma^2 = q(1-q)$.
        \end{itemize}
        \end{lemma}
        The proof of \prettyref{lmm:M_2} is deferred to \prettyref{app:M_2}. 
        With \prettyref{lmm:M_2}, we separately consider two cases, depending on whether $q$ is small and $p$ is bounded away from $1$. 
        Pick $q_0$ and $c_0$ to be sufficiently small constants such that 
         \[ 
            c_0 \le \frac{\epsilon}{2^{13}}, \quad q_0  \le  \frac{\epsilon}{480} \phi\left(1-c_0\right)
            \,  , 
        \]
        where for any $x \in (0,1)$,  we define
        \begin{align}
        \phi\left(x\right) \triangleq 
          \frac{\log \frac{1}{x} -1+x }{x \left(\log \frac{1}{x}\right) \left(1+ \log \frac{1}{x} \right) } \,  \label{eq:phix}
        \end{align} 
This function is monotonically decreasing on $(0,1)$ (see \prettyref{app:phix} for a proof). 
        
        Then we separately consider two cases.
        
        \textbf{Case 1}: $q \le q_0$ and $p \le 1- c_0$.  
        Here we will pick 
        $t \le  \log \frac{1}{p}$. Then, by \prettyref{eq:M2_q_small} in \prettyref{lmm:M_2}, 
        \begin{align*}
         -t \tau + \frac{k}{2} \log  \frac{M_1^2}{M_2} + \frac{m}{2} \left( M_2 -1 \right) 
        & \overset{(a)}{\le} 
         - t (1-\gamma) mps^2(1-p) + k + \frac{m}{2} \left(M_2-1\right) \\
        & \overset{(b)}{\le} k +  \frac{1}{2} m p s^2 \left( f(t) + 10 p^2 s t(1+t)\right)  \, , 
        \end{align*}
        where
        \[
        f(t) \triangleq  -2 \left(1 - \gamma  \right) t + e^{2t} p s^2 + 2 e^t p(1-s^2)+ ps^2 -2 p  \,;
        \]
        $(a)$ holds by $\frac{M_1^2}{M_2} \le e^2$; $(b)$ holds by \prettyref{eq:M2_q_small}.
        Note that for $t \le \log\frac{1}{p}$, 
       \begin{align}
          10 p^2 st(1+t) \le 10 q p \left(\log \frac{1}{p} \right) \left(1+\log \frac{1}{p} \right) \le \frac{\epsilon}{48} \left( \log \frac{1}{p} -1+ p\right)  \, , \label{eq:extra_term}
       \end{align}
       where the last inequality  holds because
       $q\le q_0 \le \frac{\epsilon}{480} \phi (1-c_0) \le \frac{\epsilon}{480} \phi (p)$, and  $\phi(x)$ is a monotone decreasing function on $x \in (0,1)$ by \prettyref{lmm:phix}. 
        
        Therefore, to verify~\prettyref{eq:boundtermII}, it suffices to check
        \begin{align}
            \frac{1}{2}   mps^2\left(
            f(t)    + \frac{\epsilon}{48} \left( \log \frac{1}{p} -1+ p\right) \right)\le -\left( 1+ \frac{\epsilon}{64} \right) k \log n -k. 
            \label{eq:boundtermII_2}
        \end{align}
       To this end, we need the following simple claim. 
       \begin{claim} \label{clm:claim1}
       Suppose $nps^2\left(\log \frac{1}{p} -1 +p\right) \ge \left(2+\epsilon \right)\log n$. Then, for all sufficiently large $n$ and all $0<p \le 1-c_0$, 
       \begin{align}
            \gamma (1-p) 
             \le \frac{\epsilon}{48} \left(\log\frac{1}{p}-1+p\right) \, . \label{eq:delta_1-p} 
         \end{align}
       \end{claim}
        We defer the proof of \prettyref{clm:claim1}  and first finish the proof of~\prettyref{eq:boundtermII_2}, 
        by choosing different values of $t \in [0, \log \frac{1}{p}]$ according to different cases.
        \begin{itemize}
        \item  {\bf Case 1(a)}:  $n p s^2  \left(\log (1/p)-1+p \right)\ge (4+\epsilon )\log n$.
        We pick 
        $ t = \frac{1}{2} \log \frac{1}{p}$. 
        Then 
        \begin{align*}
            f(t) 
            & =  - (1-\gamma)\log \frac{1}{p} + s^2\left(\sqrt{p}-1\right)^2 + 2\sqrt{p}- 2 p\\
            & \overset{(a)}{\le}  -(1-\gamma)  \left(\log \frac{1}{p} + 1- p  \right) + \gamma (1-p)\\
            & \overset{(b)}{\le} -  \left(1- \gamma - \frac{\epsilon}{48} \right) \left( \log \frac{1}{p} - 1 + p \right),
        \end{align*}
         where $(a)$ holds because $s^2 \le 1$; $(b)$ holds by \prettyref{eq:delta_1-p} in \prettyref{clm:claim1}.
      Hence, 
        \begin{align*}
            \frac{1}{2} mps^2 \left( f(t) + \frac{\epsilon}{48} \left( \log \frac{1}{p} - 1 + p \right) \right)
            & \le - \frac{1}{2} mps^2  \left(1- \gamma -\frac{\epsilon}{24} \right) \left( \log \frac{1}{p} - 1 + p \right)\\
            & \le - \left(1+ \frac{\epsilon}{4} \right)  \left(1- \frac{1}{n} \right) \left(1- \gamma -\frac{\epsilon}{24}   \right) k \log n. 
        \end{align*}
      where the last inequality follows from  \prettyref{eq:m}
      and our assumption $n p s^2  \left(\log (1/p)-1+p \right)\ge (4+\epsilon )\log n$.

        \item  {\bf Case 1(b)}: $  (2+ \epsilon )\log n \le n p s^2 \left(\log (1/p)-1+p \right) \le (4+ \epsilon )\log n$ and $p \ge n^{-\frac{1}{2}}$. 
         We pick $t = \log \frac{1}{p}$ and get 
        \begin{align*}
            f(t) & \le -2 \left(1-\gamma \right) \log \frac{1}{p} + \frac{s^2}{p}  + 2(1-s^2) + ps^2 -2 p  \\
            & = -2 \left(1- \gamma \right) \left(  \log \frac{1}{p}  - 1 + p\right) + \frac{s^2}{p} + 2 \gamma \left(1- p\right) + s^2(p-2)\\
            & \overset{(a)}{\le}  - 2 \left(1- \gamma  - \frac{\epsilon}{48}  \right)\left(  \log \frac{1}{p}  - 1 + p\right) +  2 \gamma  \left(1-p\right) \\
            & \overset{(b)}{\le}  - 2 \left(1-  \gamma - \frac{\epsilon}{24} \right)\left(  \log \frac{1}{p}  - 1 + p\right),
        \end{align*}
       where $(a)$ holds because $ \frac{s^2}{p}  \le \frac{(4+\epsilon)\log n}{n p^2 \left(\log (1/p)-1+p \right)}\le \frac{\epsilon}{48}\left(\log (1/p)-1+p \right)$ for all sufficiently large $n$ 
        when  $ n^{-\frac{1}{2}} \le p \le 1- c_0 $; $(b)$ holds in view of \prettyref{eq:delta_1-p} in \prettyref{clm:claim1}.
        Then under the assumption that $nps^2 \left(\log (1/p)-1+p \right) \ge (2+\epsilon) \log n$,  and by \prettyref{eq:m}, 
        \begin{align*}
            \frac{1}{2} m  ps^2  \left( f(t) + \frac{\epsilon}{48} \left( \log \frac{1}{p} - 1 + p \right) \right)
            & \le -\left( 1 + \frac{\epsilon}{2} \right) \left(1-\frac{1}{n}\right)\left(1 - \gamma- \frac{\epsilon}{16}\right) k \log n . 
        \end{align*}

        \item {\bf Case 1(c)}:  $n ps^2 \ge 4 + \epsilon $ and $p<n^{-\frac{1}{2}}$. 
        We pick $t = \frac{1}{2} \log \frac{1}{ps^2} \le \log \frac{1}{p}$, where the inequality follows because $s^2 \ge p $, in view of $n ps^2 \ge 4 + \epsilon $ and $p<n^{-\frac{1}{2}}$. Then
        \begin{align*}
            f(t)
            & = - (1-\gamma) \log \frac{1}{ps^2} + 1 + 2 \sqrt{\frac{p}{s^2}}(1-s^2)-2p+ps^2  \\
            &\le - (1-\gamma) \log \frac{1}{ps^2} + 1 + 2 \sqrt{\frac{p}{s^2}} \\
            & \le- \left(1-\gamma - \frac{\epsilon}{48}\right) \log \frac{1}{ps^2} \, ,
        \end{align*}
        where the last inequality holds for $n$ sufficiently large because $\frac{p}{s^2}  \le \frac{np^2}{4+\epsilon} \le \frac{1}{4+\epsilon}$ and 
        $\log \frac{1}{ps^2} \ge \frac{1}{2} \log n$ when $nps^2 \ge 4+\epsilon$ and $p < n^{-\frac{1}{2}}$.  
        Then, 
        \begin{align*}
            \frac{1}{2} m  ps^2  \left( f(t) + \frac{\epsilon}{48} \left( \log \frac{1}{p} - 1 + p \right) \right)
            & \overset{(a)}{\le} - \frac{1}{4}\left(1- \frac{1}{n}\right)\left(1-\gamma - \frac{\epsilon}{24}\right)  k n ps^2 \log \frac{1}{ps^2} \\
            & \overset{(b)}{\le} - \left(1+ \frac{\epsilon}{4} \right) \left(1- \frac{1}{n}\right)\left(1-\gamma  - \frac{\epsilon}{24} \right) k \log \frac{n}{4+\epsilon} \, ,
        \end{align*}
        where $(a)$ holds by \prettyref{eq:m} and
        $
        \log \frac{1}{p} - 1+p \le \log  \frac{1}{ps^2} ; 
        $
        $(b)$ holds due to $nps^2 \ge 4+\epsilon$. 

        \end{itemize}
        Hence, in view of $ \gamma < \frac{\epsilon}{16}$ for $\delta n \le k \le n$,  combining the three cases, we get
        \begin{align*}
            \frac{1}{2} mps^2\inf_{ 0 \le t \le \log (1/ps^2) } f(t)
            & <   - \left(1+\frac{\epsilon }{4} \right) \left(1- \frac{1}{n}\right)\left(1 - \frac{\epsilon}{8}\right) k \log \frac{n}{4+\epsilon}  \\
            & \le   - \left( 1+ \frac{\epsilon}{64} \right) k \log n  - k ,
        \end{align*}
        where the last inequality holds for all sufficiently large $n$.
        Thus we arrive at the desired~\prettyref{eq:boundtermII_2} which further implies~\prettyref{eq:boundtermII}.
    
    Now we are left to prove~\prettyref{clm:claim1}. 
      \begin{proof}[Proof of \prettyref{clm:claim1}]
           Note that for any $0< p \le 1-c_0$, 
        \begin{align*}
        \gamma(1-p) \stepa{\le}  \sqrt{\frac{16 h(\delta)/\delta}{ nps^2}} 
        \stepb{\le} \sqrt{\frac{8 h ( \delta )/\delta \left(\log \frac{1}{p}-1+p\right)}{ \log n}} 
        & \overset{(c)}{\le}  \sqrt{\frac{16 h ( \delta )/\delta }{  \left(\log \frac{1}{1-c_0} - c_0\right) \log n}}\left(\log \frac{1}{p}-1+p\right) \\
        &  \overset{ (d)}{\le}  \frac{\epsilon}{48} \left(\log \frac{1}{p}-1+p\right)  \, , 
        \end{align*}
        where $(a)$ holds because if $\delta n \le k \le n-1$, $h(x)/x$ is a monotone decreasing function on $0< x \le 1$ and $\frac{n}{k} h\left(\frac{k}{n}\right) \le  h\left(\delta\right)/ \delta$, 
        and if $k=n$, $\frac{\log n}{n-1} \le   h(\delta)/\delta $ for $n$ sufficiently large; $(b)$ holds by assumption $ nps^2 \left(\log \frac{1}{p} - 1+ p\right) \ge (2 +\epsilon) \log n$; 
        $(c)$ holds because $\log \frac{1}{x} -1+x$ is monotone decreasing function on $x\in (0,1)$ and we have $p \le 1-c_0$; $(d)$ holds for $n$ sufficiently large. 
        Hence, the claim~\prettyref{eq:delta_1-p} follows. 
\end{proof}
        \textbf{Case 2}: $q \ge q_0$ or $p \ge 1 -c_0$.        
        Here we will pick 
        $t \le \frac{1}{2^{10}}$.
        Then, by \prettyref{lmm:M_2}, we get
        \begin{align*}
            -t \tau + \frac{k}{2} \log  \frac{M_1^2}{M_2} + \frac{m}{2} \left( M_2 -1\right)   
            &  \overset{(a)}{\le} k 
            -t \tau + \frac{m}{2} \left( M_2 -1\right)    \\
            & \overset{(b)}{\le}  k + m f(t),
        \end{align*}
        where 
        \[
        f(t) \triangleq   -  
        t \rho \sigma^2 \left( 1 - \frac{\epsilon}{16} \right) +  
        \frac{1}{2} \left( M_2 -1\right) \, ;
        \]
        $(a)$ holds by $\frac{M_1^2}{M_2} \le e^2$, and \prettyref{eq:M2_t_small} given $t \le \frac{1}{2^{10}}$; $(b)$ holds because $ \tau \ge \rho \sigma^2 m \left( 1-\frac{\epsilon}{16} \right) $, in view of $\tau =  (1-\gamma) mps^2(1-p) $, $\rho = \frac{s(1-p)}{1-ps}$, $\sigma^2 = ps(1-ps)$, and $\gamma \le \frac{\epsilon}{16}$. 
        
        Therefore, to verify~\prettyref{eq:boundtermII}, it suffices to check 
        \begin{align}
           m  f(t)   \le -\left( 1+ \frac{\epsilon}{64} \right)k  \log n - k \, . 
            \label{eq:ER_boundtermII_2}
        \end{align}
        \begin{itemize}
            \item  {\bf Case 2(a)}: $p > 1-c_0$.
            We pick $t = \frac{\rho}{\sigma^2\left(1+\rho^2\right)} =  \frac{(1-p)}{p(1-ps) (1+\rho^2)} \overset{(a)}{<} \frac{2 c_0}{1-c_0} \le \frac{1}{2^{10}}$, where $(a)$ holds in view of $p > 1-c_0$ and $ps \le \frac{1}{2}$. By \prettyref{eq:M2_t_small},
            \[
            M_2 \le 1+ \frac{\rho^2}{1+\rho^2}  +  \frac{8 \rho^3 qs}{\sigma^6 (1+\rho^2)^3} \le 1 + \frac{\rho^2}{1+\rho^2} \left(1+ \frac{\epsilon}{16}\right)\, ,
            \]
            where the last inequality holds because $\frac{8 \rho qs}{\sigma^6 (1+\rho^2)^2}  = \frac{8 (1-p) }{p^2(1-ps)^4 (1+\rho^2)^2} \le \frac{\epsilon}{32} $ given
            $1- p \le c_0 \le \frac{\epsilon}{2^{13}}$ and $ ps \le \frac{1}{2}$.
            Then, we have
            \begin{align*}
                f(t) 
                & \le  - \frac{\rho^2}{1+\rho^2}  \left(  1 - \frac{\epsilon}{16}  - \frac{1}{2 }
                \left(1+ \frac{\epsilon}{32}\right)\right)\\
                & =  - \frac{\rho^2}{1+\rho^2} \left(\frac{1}{2}-\frac{5 \epsilon}{64}\right) \, .
            \end{align*}   
            Therefore, 
            \begin{align*}
                m f(t) 
                & \overset{(a)}{\le} - \frac{1}{2} \left(1-\frac{1}{n}\right) k n \left(\frac{ \rho^2}{1+\rho^2}\right) \left(\frac{1}{2}-\frac{5 \epsilon}{64}\right) \overset{(b)}{\le} -\left( 1+ \frac{\epsilon}{64} \right)k  \log n - k \, ,
            \end{align*}
            where $(a)$ holds by \prettyref{eq:m}; $(b)$ holds under the assumption that $ \rho^2  = \frac{s^2(1-p)^2}{(1-ps)^2}\ge (4+\epsilon) \frac{\log n}{n}$ in \prettyref{eq:recovery_positive_cond_2} for $n$ sufficiently large.
            
            \item \textbf{Case 2(b)}: $q \ge q_0$ and $p \le 1- c_0$.
            We pick $t = \frac{4 \log n}{ \rho \sigma^2 n}  =\frac{4\log n}{ps^2(1-p) n} \le \frac{1}{2^{10}}$, where the last inequality holds for $n$ sufficiently large. 
            By \prettyref{eq:M2_t_small}, 
            \begin{align*}
                 M_2 \le 1+  t^2 \sigma^4 \left(1+ \rho^2 + \frac{8 t qs}{\sigma^4 (1+\rho^2) }\right) 
                 & \le 1 +t^2 \sigma^4 (1+\rho^2) \left(1+\frac{\epsilon}{64}\right)\, ,
            \end{align*}
            where the last inequality holds because $ \frac{8 t qs}{\sigma^4 (1+\rho^2)^2 } = \frac{32 \log n }{ (1-p) q^2 (1-q)^2 \left(1+\rho^2\right)^2 n } \le \frac{\epsilon}{64} $ 
            for $n$ sufficiently large, given $ q_0 \le q \le \frac{1}{2}$ and $p \le 1- c_0$.  
            Then, we have
            \begin{align*}
                f(t) 
                & \le - \frac{4 \log n}{n} \left(1-\frac{\epsilon}{32}\right) + \frac{8 \log^2 n}{ n^2} \frac{1+\rho^2}{\rho^2} \left(1+\frac{\epsilon}{64}\right) \\
                & \le - \frac{4 \log n}{n} \left(1-\frac{\epsilon}{32}\right) + \frac{32 \log^2 n}{ \rho^2 n^2} \\
                & \le - \frac{3 \log n}{n} \, ,
            \end{align*}
            where the last inequality holds for $n$ sufficiently large because
            $\rho = \frac{s(1-p)}{1-ps} \ge \frac{q_0 c_0}{1-q_0} $.
            Therefore,  
            \begin{align*}
            m f(t) 
            & \overset{(a)}{\le} -  \frac{3}{2}  \left(1-\frac{1}{n}\right)  k \log n  
            \overset{(b)}{\le}  - \left( 1+ \frac{\epsilon}{64} \right) k \log n - k \, ,
            \end{align*}
            where $(a)$ holds by \prettyref{eq:m};
            $(b)$ holds for $n$ sufficiently large.
            \end{itemize}

      
        \item 
        For Gaussian   model,  set $\tau = \rho m - a_k$, where  
        \[
        a_k = \begin{cases}
            C \sqrt{ 2  h\left(\frac{k}{n} \right) n m} & k \le n-1\\
            C n \sqrt{\log n} & k=n
        \end{cases} \,
        \]
        for some universal constant $C>0$ to be specified later. 
       Recall that  $h(x)/x$ is monotone decreasing in $x$ and converges to $0$ as $x \to 1$,  
        under the condition~$n\rho^2\ge (4+\epsilon)\log n$, for any constant $0< \delta <1$, we have
        $ h\left(\delta \right)/ \delta \le \frac{\epsilon^2  }{2^{13}C^2} n \rho^2$ when $n$ is sufficiently large. 
        Therefore, for $\delta n \le k \le n-1$,   
        $\frac{n}{k} h\left(\frac{k}{n}\right) \le  h\left(\delta \right)/ \delta \le \frac{\epsilon^2  }{2^{13}C^2} n \rho^2 $. 
        Since $m = kn\left (1- \frac{k+1}{2n}\right)\ge \frac{1}{4} kn$, we have 
        $a_k \le \frac{\epsilon}{64} \rho\sqrt{ k n m} \le \frac{\epsilon}{32} \rho m$.
        For $k=n$, by the assumption that $n\rho^2 \ge (4+\epsilon) \log n$, 
        $a_k \le \frac{\epsilon}{32} \rho m  $. In conclusion, we have
        $\tau \ge \rho m \left( 1-\frac{\epsilon}{32} \right)$.

         Note that $(A_{\pi(i)\pi(j)},B_{ij})$ are i.i.d.\ pairs of standard normal random variables with zero mean and correlation coefficient $\rho$. 
        Therefore, $\overline{A} = A $ and 
        $\overline{B} = B$. It follows that  $ Y = \sum_{ (i,j) \not\in \binom{F}{2}} A_{\pi(i)\pi(j) } B_{ij}$.

        First, we verify~\prettyref{eq:boundtermI} using the Hanson-Wright inequality 
        stated in \prettyref{lmm:hw}. Pick $C = 2c$, where $c>0$ is the universal constant in \prettyref{lmm:hw}. 
        If $k=n$, applying  \prettyref{lmm:hw} with $M=I_m$ to $ Y = \sum_{ (i,j) \not\in \binom{F}{2}} A_{\pi(i)\pi(j) } B_{ij} $ yields that 
        $
        \calP\pth{ Y \le \tau } \le e^{- 2 \log n}.
        $
        If $k \le n-1$, applying \prettyref{lmm:hw} again yields that 
        \begin{align*}
            \binom{n}{k} \calP\pth{ Y \le \tau } 
            \le \binom{n}{k}  \exp \left( - 2 n h\left(\frac{k}{n} \right) \right)
            \le \exp \left( - n h\left(\frac{k}{n} \right) \right).
        \end{align*} 
        Next, we check~\prettyref{eq:boundtermII}. 
        In view of \prettyref{eq:GW_expect_exp_AB_pi} in \prettyref{lmm:expect_exp_AB_pi},
        $M_1=\frac{1}{\lambda_2}$ and
        $
        M_2= \frac{1}{\lambda_1\lambda_2},
        $ 
        where $\lambda_1 = \sqrt{\left(1+\rho t \right)^2 - t^2 }$ and $\lambda_2 = \sqrt{\left(1- \rho t \right)^2 - t^2 }$
        for $0 \le t \le \frac{1}{1+\rho}$. 
        Thus for $0 \le t \le \frac{\rho}{1+\rho^2}$, $\frac{M_1}{M_2}=\lambda_1 \le 1+\rho^2 \le 2$.
        It follows that 
        \begin{align*}
            \inf_{t\ge 0} \left(-t \tau + \frac{k}{2} \log  \frac{M_1^2}{M_2} + \frac{m}{2}\left( M_2-1\right)    \right)
            & =  \inf_{t \ge 0}
            \left( -t \tau + k \log  \frac{M_1}{M_2} + \frac{m+k}{2} \left( M_2-1\right)    \right)\\
            & \le k +
            \inf_{0 \le t\le \frac{\rho}{1+\rho^2} } \left( -t \tau + \frac{m+k}{2} \left( M_2-1\right)  \right) \\
            & \le k + \frac{1}{2} k (n-1) \rho \inf_{0 \le t\le\frac{\rho}{1+\rho^2} } f(t),
        \end{align*}
        where
        \[
        f(t) \triangleq   -  
        t \left( 1 - \frac{\epsilon}{32} \right) +  
        \frac{1}{2\rho} \left( 1+ \frac{2}{n-1} \right) \left( M_2-1\right) , 
        \]
        and the last inequality holds because  $\tau \ge \rho m \left( 1-\frac{\epsilon}{32} \right)$
        and \prettyref{eq:m}.
        
        Therefore, to verify~\prettyref{eq:boundtermII}, it suffices to check
        \begin{align}
            \frac{1}{2} (n-1) \rho 
            \inf_{0 \le t\le \frac{\rho}{1+\rho^2} } 
            f(t)   \le -\left( 1+ \frac{\epsilon}{64} \right)  \log n -1. 
            \label{eq:GW_boundtermII_2}
        \end{align}
        {\bf Case 1}: $\rho^2 \ge \frac{\epsilon^2}{256}$. We pick $t= \frac{\rho}{1+2\rho^2}$. Then, 
        \[
         M_2 = \frac{\left(1+2\rho^2\right)^2}{\sqrt{1+\rho^2+\rho^4} \sqrt{1+5\rho^2+9\rho^4}} \le   \left(\frac{1+2\rho^2}{1+\frac{3}{2}\rho^2}\right)^2 \le 1+ \frac{ 2\rho^2}{2+3\rho^2} \, ,
        \]
        where the first inequality holds because $(1+\rho^2+\rho^4)(1+5\rho^2+9\rho^4) \ge \left( 1+\frac{3}{2}\rho^2\right)^4$.
        Therefore, 
        \begin{align*}
            f(t)
            & = -   \frac{\rho}{1+2\rho^2} \left(1- \frac{\epsilon}{32}\right)  +  \frac{\rho}{2+3\rho^2} \left(1+ \frac{2}{n-1}\right) \\
            & =  -   \frac{\rho}{2+3\rho^2} \left( \left(2 - \frac{\rho^2}{1+2\rho^2}\right) \left(1 - \frac{\epsilon}{32}\right)-\left(1+ \frac{2}{n-1}\right)\right) \\
            & \le -  \frac{\rho}{2+3\rho^2} \left( \frac{5}{3}\left(1 - \frac{\epsilon}{32}\right)-\left(1+ \frac{2}{n-1}\right)\right)\\  
            &\le -  \frac{\rho}{5} 
            \left( \frac{2}{3} -\frac{5\epsilon}{96} - \frac{2}{n-1}\right),
        \end{align*}
        where the first inequality holds because $ \frac{\rho^2}{1+2\rho^2} \le \frac{1}{3}$ and the last inequality follows
        from $2+3\rho^2 \le 5$. Hence, 
        \[
        \frac{1}{2} (n-1) \rho  f(t)  
        \le  -  (n-1) \frac{\rho^2}{10}
        \left( \frac{2}{3} -\frac{5\epsilon}{96} - \frac{2}{n-1}\right).
        \]
        Thus it follows from $\rho^2 \ge \frac{\epsilon^2}{256}$ that 
        \prettyref{eq:GW_boundtermII_2} holds for all sufficiently large $n$.
        
        {\bf Case 2}: $(4+\epsilon) \frac{\log n}{n} \le \rho^2 < \frac{\epsilon^2}{256}$. We pick $t = \frac{\rho}{1+\rho^2}$. Then,
        \[
        M_2 = \frac{\left(1+\rho^2\right)^2}{\sqrt{1+3\rho^2 + 4 \rho^4}\sqrt{1-\rho^2 } } 
        \le \frac{\left(1+\rho^2\right)^2 }{ 1+ \rho^2 -\rho^4 } 
            \le 1 + \rho^2 + 2\rho^4,
        \]
        where the first inequality holds because
        $(1+3\rho^2 + 4 \rho^4)(1-\rho^2) - (1+\rho^2-\rho^4)^2=\rho^4 (2 - 2\rho^2+\rho^4) \ge 0$
        under the assumption that $\rho^2 \le \epsilon^2/256$. Therefore,
        \begin{align*}
            f(t) \le - \frac{\rho}{1+\rho^2}  \left(1- \frac{\epsilon}{32}\right)  +  \frac{\rho(1+2\rho^3)}{2 }\left(1+ \frac{2}{n-1}\right) 
            \le - \rho \left(\frac{1}{2}-\frac{\epsilon}{16}\right)
            \left(1 - \frac{4}{n-1}\right) ,
        \end{align*}   
        where the first inequality holds
        because $\rho^2 < \frac{\epsilon^2}{256}$.
        Hence,
        \begin{align*}
            \frac{1}{2} (n-1) \rho f(t) 
            \le - \frac{1}{4} (n-1) \rho^2 \left(1-\frac{\epsilon}{32}\right) \left(1- \frac{4}{n-1}\right)  
            \le - \left(1
            + \frac{\epsilon}{32} \right) \left(1-\frac{1}{n}\right) \left(1- \frac{4}{n-1}\right)   \log n,
        \end{align*} 
        where the last inequality follows from the assumption that $ \rho^2 \ge (4+\epsilon) \frac{\log n}{n}$. 
        Thus \prettyref{eq:GW_boundtermII_2} holds for all  sufficiently large $n$. 
    \end{itemize}

    \section{Exact Recovery} 
    \label{sec:exact}

     \subsection{Possibility of Exact Recovery}
     \label{sec:exact-positive}
     Building upon the almost exact recovery
     results in the preceding section, we now analyze the MLE for exact recovery. Improving upon \prettyref{lmm:ER_fixing_k}, the following lemma gives a tighter bound on
     the probability that the MLE makes a small number of errors
     with respect to the general correlated \ER graph model specified by the joint distribution $P=(p_{ab}: a,b\in\{0,1\})$. 
        \begin{lemma}
        \label{lmm:small_k_prob}
        Suppose that for any constant $0<\epsilon<1$,
        \begin{itemize}
            \item for general \ER random graphs, if $n \left(\sqrt{p_{00}p_{11}}-\sqrt{p_{01}p_{10}}\right)^2 \ge \left(1+ \epsilon\right) \log n$;
            \item for Gaussian   model, if $n\rho^2  \ge (4+\epsilon) \log n$;
        \end{itemize}
        then for any $k\in [n]$ such that $k\le \frac{\epsilon}{16} n $, 
        \begin{align}
            \prob{ d\left(\piML, \pi\right)  = k} \le \exp\left(-\frac{\epsilon}{8} k \log n\right).  \label{eq:small_k}
        \end{align}
    \end{lemma}
Note that~\prettyref{lmm:small_k_prob} only holds when $k/n$ 
is small but requires less stringent conditions
     than \prettyref{lmm:ER_fixing_k}. 
     Inspecting the proof
     of~\prettyref{lmm:small_k_prob}, one can see that 
     if the condition is strengthened by a factor of $2$ to 
     $n \left(\sqrt{p_{00}p_{11}}-\sqrt{p_{01}p_{10}}\right)^2 \ge \left(2+ \epsilon\right) \log n$ for the \ER graphs (which is the assumption of \cite{cullina2016improved}),
     then \prettyref{eq:small_k} holds for any $k\in [n]$.
    Conversely, we will prove in \prettyref{sec:exact-negative} that exact recovery is impossible if $n \left(\sqrt{p_{00}p_{11}}-\sqrt{p_{01}p_{10}}\right)^2 \le \left(1- \epsilon\right) \log n$. Closing this gap of two for general \ER graphs is an open problem.


In the following, we apply~\prettyref{lmm:small_k_prob}
for small $k$ and~\prettyref{lmm:ER_fixing_k} (or \prettyref{prop:possibility_almost})  for large $k$
to show the success of MLE in exact recovery. 


\begin{proof}[Proof of positive parts in~\prettyref{thm:Gaussian_negative} and~\prettyref{thm:exact_ER}] 

Without loss of generality, we assume $\epsilon < 1$.
        For all $k \le \frac{\epsilon}{16} n$, by \prettyref{lmm:small_k_prob}, for both \ER random graphs and Gaussian  model, 
        \begin{align*}
            \sum_{k=1}^{\frac{\epsilon}{16} n} \prob{ d\left(\piML, \pi\right) = k} \le  \sum_{k=1}^{\frac{\epsilon}{16} n} \exp\left(-\frac{\epsilon}{8} k \log n\right) \le  \frac{\exp\left(-\frac{\epsilon}{8}\log n\right)}{1-\frac{\epsilon}{8}\log n } = n^{- \frac{\epsilon}{8} +o(1)}. 
        \end{align*}
        Moreover, 
        \begin{itemize} 
       \item   For the Gaussian model, pick $\delta = \frac{\epsilon}{16}$. Thus by \prettyref{eq:ER_fixing_k} in \prettyref{lmm:ER_fixing_k} and \prettyref{eq:summation_k}, \prettyref{eq:exact_large_k} follows. 
         
             \item  For the subsampling \ER random graphs,  we divide the proof into two cases depending on whether $ps \le 1/2$. 
    
             \textbf{Case 1: $ps\le \frac{1}{2}$.} Note that condition \prettyref{eq:exact_positive_er} is equivalent to
            \begin{align}
                n \left(\sqrt{p_{11}p_{00}} - \sqrt{p_{01}p_{10}} \right)^2
                & = n ps^2 f(p,s) \ge (1+\epsilon) \log n \, ,  \label{eq:exact_positive_er_1}
            \end{align}
            where  \[f(p,s) \triangleq \left(\sqrt{1-2ps+ps^2}-\sqrt{p}(1-s)\right)^2 \, .\]
            In the sequel, we aim to show that if~\prettyref{eq:exact_positive_er_1} holds,
            then $nps^2(1-p) = \omega(1)$,  \prettyref{eq:recovery_positive_cond_2} holds for $p = 1-o(1) $, and \prettyref{eq:recovery_positive_cond_1}  holds for $p = 1-\Omega(1)$.  To this end, we will use the  inequality
            \begin{align}
               0\le   \frac{\partial f(p,s)}{\partial s} \le 2 \sqrt{p} \, , \label{eq:f_derivative}
            \end{align}
           which follows from
            \[
            \frac{\partial f(p,s)}{\partial s}= \dfrac{2 \sqrt{p} \left(\sqrt{1-2ps+ps^2}-\sqrt{p}\left(1-s\right)\right)^2 }{\sqrt{1-2ps+ps^2}} \, .  
            \]
            \begin{itemize}
             \item    
            By \prettyref{eq:f_derivative}, $f(p,s)$ is a monotone increasing on $s\in (0,1)$. Therefore, 
            $
            f(p,s) \le 1-p.
            $
            Hence, if \prettyref{eq:exact_positive_er_1} holds, then 
            $
            nps^2 \left(1-p\right) = \omega(1) \, .
            $
            \item Suppose $p = 1-o(1)$. We have
            \begin{align*}
            f(p,s) & = p(1-s)^2 \left(\sqrt{1+ \frac{1-p}{p(1-s)^2}} -1\right)^2 \le \frac{(1-p)^2}{4 p(1-s)^2} \, ,    
            \end{align*}
            where the last inequality holds because $\sqrt{1+x} \le 1+ \frac{x}{2}$. 
            Thus, if \prettyref{eq:exact_positive_er_1} holds, then 
            \begin{align*}
                 \frac{ n \left(1-p\right)^2}{ \left( 1-p s \right)^2  }  \ge \frac{ n \left(1-p\right)^2}{ \left(1 + \frac{\epsilon}{2}\right) p \left( 1-s \right)^2  } \ge (4 + \epsilon) \log n \, ,
            \end{align*} 
            where the first inequality holds because $p = 1-o(1)$ and $ps \le \frac{1}{2}$. Hence, \prettyref{eq:recovery_positive_cond_2} follows. 
                \item Suppose $p = 1-\Omega(1)$. 
                By the
                mean value theorem and \prettyref{eq:f_derivative},
                we have
                \[
                f(p,s) \le f(p,0) + s \sup_{s:ps \le 1/2} \frac{\partial f(p,s)}{\partial s} \le \left(1-\sqrt{p}\right)^2 +2 \sqrt{p} s \, ,
                \]
                Now, suppose $s = o(1)$ or $p = o(1)$. Then
                \[
                \left(1-\sqrt{p}\right)^2 +2 \sqrt{p} s  \le \left(1-\sqrt{p}\right)^2\left(1+ \frac{\epsilon}{4} \right) \, , 
                \]
                for all sufficiently large $n$. Therefore, if \prettyref{eq:exact_positive_er_1} holds, then
                \begin{align}
                    nps^2 (1-\sqrt{p})^2 \ge \left(1+\frac{\epsilon}{2}\right) \log n \, . \label{eq:p_bounded_away_1} 
                \end{align}
                
                If instead $s = \Omega(1)$ and $p = \Omega(1)$, then \prettyref{eq:p_bounded_away_1} follows directly. 
                
                In conclusion, for both cases, since $\log \frac{1}{p} -1+ p \ge 2 (1-\sqrt{p})^2$,  
                we get that \prettyref{eq:recovery_positive_cond_1} holds.   
                
                
            \end{itemize}
          Finally, by
          \prettyref{prop:possibility_almost}, we have that
            \begin{align}
                \prob{ d\left(\piML, \pi\right) > \frac{\epsilon}{16} n } \le n^{-1 +o(1)} \, .
                \label{eq:exact_large_k}
            \end{align}
          
          \textbf{Case 2: $ps > \frac{1}{2}$. } In this case, we consider the equivalent model $(n,p',s')$ as specified in~\prettyref{eq:equivalent_model} with $p's'=1-ps\le 1/2$ by flipping $0$ and $1$. Note that  after flipping $0$ and $1$, both the value of 
          $\sqrt{p_{11}p_{00}} - \sqrt{p_{01}p_{10}}$
          and the MLE are clearly unchanged. 
        Therefore, applying the Case $1$ of~\prettyref{thm:exact_ER} to 
        to the model $(n,p',s')$, we get that \prettyref{eq:exact_large_k} holds for the model $(n,p,s).$

        \end{itemize}
        Hence, for both the \ER random graph and Gaussian   model, 
        \[
        \prob{ d\left(\piML, \pi\right)>  0} \le  \sum_{k=1}^{\frac{\epsilon}{16} n}  \prob{ d\left(\piML, \pi\right) = k} + \prob{ d\left(\piML, \pi\right) > \frac{\epsilon}{16} n }  \le n^{-\frac{\epsilon}{8} + o(1)}.
        \]
    \end{proof}

    \subsubsection{Proof of \prettyref{lmm:small_k_prob}}

    In this proof
    we focus on the case with positive correlation, \ie,
    $p_{00}p_{11} \ge p_{01}p_{10}$ in the general \ER graphs
    and $\rho\ge 0$ in the Gaussian model. 
    The case with negative correlation can be proved analogously
    by bounding the probability $\iprod{A^{\pi'}-A^{\pi}}{B} \le 0$.
    
    Fix $k \in [n]$ and let $\calT_k$ denote the set of permutation $\hpi$ such that $d(\pi,\hpi) = k$. 
    Let $\calO_1'$ is the set of fixed points of edge permutation $\sigmae$,  where in view of~\prettyref{eq:fixed_edge_node},
    \[|\calO_1'|= N_1 = \binom{n_1}{2}+n_2= \binom{n-k}{2}+n_2.\] 
     Then, applying the Chernoff bound together with the union bound yields that for any $t\ge 0$,
        \begin{align}
            \prob{ d\left(\piML, \pi\right)  = k} 
            & \le \sum_{\hpi \in \calT_k } \prob{ Z_\hpi \ge  0} 
            \le n^k \expect{\exp\left({t Z_\hpi}\right)}, \label{eq:small_k_upperbound}
        \end{align}
    where
    \begin{align*}
        Z_\hpi 
        & \triangleq \sum_{i<j}A_{\hpi(i)\hpi(j)}  B_{ij}  -\sum_{i<j} A_{\pi(i)\pi(j)} B_{ij} \\
        & = \sum_{\calO\backslash \calO_1'} \left(\sum_{(i,j)\in O}A_{\hpi(i)\hpi(j)}B_{i,j}-\sum_{(i,j)\in O}A_{\pi(i)\pi(j)}B_{i,j} \right)  \triangleq  \sum_{\calO\backslash \calO_1'} Z_O \, ,
    \end{align*}
    where 
    \[
    Z_O = X_O - Y_O \, . 
    \]
    Since edge orbits are disjoint, it follows that $Z_O$ are mutually independent across different $O$. Therefore, for any $t>0$,
    $$
    \expect{\exp\left(t \sum_{O \in \calO \backslash \calO_1'} Z_O\right)}  = \prod_{O\in  \calO \backslash \calO'_1} \expect{\exp\left(t Z_O\right)}.
    $$ 
    It turns out that the MGF of $Z_O$ can be explicitly computed in both  \ER random graph and Gaussian  model. In particular, applying  \prettyref{lmm:expect_exp_AB_pi_minus_AB_pi_*} yields that
    $\expect{e^{t Z_O}} = L_{|O|}$ and $L_\ell \le L_2^{\ell/2}$ for $\ell \ge 2$.
      It follows that 
    $$
    \expect{ \exp\left({t Z_\hpi}\right) } =  \prod_{\ell=2}^{\binom{n}{2}} L_\ell^{N_\ell}
    \le \prod_{\ell=2}^{\binom{n}{2}} L_2^{\ell N_\ell /2}
    \le L_2^{\frac{m}{2}} \, ,
    $$
    where  $m \triangleq \binom{n}{2} -  \binom{n-k}{2} - n_2$ and 
    the last inequality follows from $\sum_{\ell = 1}^{\binom{n}{2}} \ell N_\ell = \binom{n}{2}$ and $N_1 = \binom{n_1}{2} + n_2 = \binom{n-k}{2} + n_2$.
    Hence, 
    \begin{align*}
        \prob{ d\left(\piML, \pi\right)  = k} 
        & \stepa{\le} \exp\left(k \log n  + \frac{1}{2}  k n \left(1-\frac{k+2}{2n}\right) \log L_2 \right) \\
        & \stepb{\le} \exp\left(-  \left(\frac{\epsilon}{4}- \frac{k+2}{2 n} \left(1+\frac{\epsilon}{4}\right)\right)k\log n\right)  \\
        & \stepc{\le}  \exp\left(- \frac{\epsilon}{8} k \log n  \right),
    \end{align*}
    where $(a)$ holds because $n_2 \le k/2$ and then $m \ge k n \left(1-\frac{k+2}{2n} \right)$;
    $(b)$ holds by the claim that $ n \log L_2 \le  -2 (1+\epsilon/4) \log n$ for appropriately chosen $t$ ; 
    $(c)$ holds for all sufficiently large  $n$
    given $k/n \le \frac{\epsilon}{16}$ and $0<\epsilon < 1$. 
    It remains to check  $ n \log L_2 \le  -2 (1+\epsilon/4) \log n$ by choosing appropriately $t$ for \ER random graphs and Gaussian   model separately.   
    \begin{itemize}
        \item For \ER random graphs, 
        in view of \prettyref{eq:ER_expect_exp_AB_pi_minus_AB_pi_*} in \prettyref{lmm:expect_exp_AB_pi_minus_AB_pi_*},
        \[
        L_2 =
        1 + 2 \left(p_{01}p_{10}\left(e^t-1\right) +p_{00}p_{11}\left(e^{-t}-1\right) \right).
        \]
        Since $p_{00}p_{11} \ge p_{01}p_{10}$, by choosing the optimal $t \ge 0$ such that $e^t =\sqrt{\frac{p_{00}p_{11}}{p_{01}p_{10}}} $,
        $
        L_{2}
        = 1- 2 \left(\sqrt{p_{00}p_{11}} - \sqrt{p_{01}p_{10}} \right)^2, 
        $
        and hence 
        $$
        n \log L_2 \le - 2n \left(\sqrt{p_{00}p_{11}} - \sqrt{p_{01}p_{10}} \right)^2 \le- 2(1+\epsilon) \log n,
        $$
        where the last inequality holds by the assumption that $ n \left(\sqrt{p_{00}p_{11}} - \sqrt{p_{01}p_{10}} \right)^2 \ge  (1+\epsilon) \log n$;

        \item For Gaussian   model, 
        in view of \prettyref{eq:GW_expect_exp_AB_pi_minus_AB_pi_*} in \prettyref{lmm:expect_exp_AB_pi_minus_AB_pi_*},
        \[
        L_2 
        =  \frac{1}{\sqrt{1 + 4 t \rho -4 t^2 \left(1-\rho^2 \right)}}.
        \]
        By choosing the optimal $t = \frac{\rho }{2\left(1-\rho^2\right)} \ge 0$, 
        $
        L_2 = \sqrt{1-\rho^2}. 
        $
        and hence
        $$
        n\log L_2 \le - \frac{1}{2} n \rho^2 \le- 2(1+\epsilon/4) \log n,
        $$
        where the last inequality holds by the assumption that  $ n \rho^2 \ge  (4+\epsilon) \log n$.
    \end{itemize}

    \subsection{Impossibility of Exact Recovery} 
    \label{sec:exact-negative}
    

In this subsection we prove the negative result in \prettyref{thm:exact_ER}. 
 As in \prettyref{sec:exact-positive}, we consider a  general correlated \ER graph model, 
 where 
$\prob{A_{\pi(i)\pi(j)}=a, B_{ij}=b}=p_{ab}$. 
We aim to show that if
    \begin{align}
    \label{eq:exact_negative_er-general}
    n \left( \sqrt{p_{00} p_{11}}-\sqrt{p_{01} p_{10}}\right)^2 \le \left(1- \epsilon\right) \log n,
    \end{align}
then the exact recovery of the latent permutation $\pi$ is impossible.
Particularizing to the subsampling model parameterized by~\prettyref{eq:general-subsampling}
the condition \prettyref{eq:exact_negative_er-general}  reduces to 
\[
    nps^2\left(\sqrt{1-2ps+ps^2}-\sqrt{p}(1-s)\right)^2  \leq (1-\epsilon) \log n,
\]
which, under the assumption that $p=1-\Omega(1)$ in \prettyref{thm:exact_ER}, is further equivalent to 
the desired \prettyref{eq:exact_negative_er}.


Since the true permutation $\pi$ is uniformly distributed, the MLE $\piML$ minimizes the error probability among all estimators. 
In the sequel, we focus on the case of positive correlation as the other case is entirely analogous. Without loss of generality, the latent permutation $\pi$ is assumed to be the identity permutation $\id$.


    To prove the failure of the MLE, it suffices to show the existence of a  permutation 
    $\pi'$ that achieves a higher likelihood than the true permutation $\pi=\id$.
    To this end, we consider
    the set $\calT_2$ of permutation $\pi'$ such that $d(\pi',\pi)=2$; in other words, $\pi'$ coincides with $\pi$ except for swapping two vertices. Then 
    the cycle decomposition of $\pi'$ consists of $n-2$ fixed points and 
    a $2$-node orbit $(i,j)$ for some pair of vertices $(i,j)$. It follows that 
    \begin{align*}
        \left\langle A,B\right\rangle - \langle A^{\hpi},B\rangle 
        = \sum_{k\in [n] \backslash \{i,j\}} \left(A_{i k}- A_{j k}\right)\left(B_{i k}- B_{j k}\right)
        \triangleq \sum_{k\in [n] \backslash \{i,j\}} X_{ij,k},
    \end{align*} 
    where $ X_{ij,k}  \iiddistr a\delta_{+1}+ b \delta_{-1}+ \left(1-a-b\right) \delta_{0}$ for $k  \in [n] \backslash \{i,j\}$
    with $a \triangleq 2 p_{00}p_{11}$, $b \triangleq 2 p_{01}p_{10}$ and $a \geq b$ by assumption.
We aim to show the existence of $\pi' \in \calT_2$ such that 
    $\left\langle A,B\right\rangle \le \langle A^{\hpi},B\rangle$,
    which further implies that $\pi'$ has a higher likelihood than $\pi$.
    We divide the remaining analysis into two cases depending on whether 
    $na \ge (2-\epsilon)\log n$.
    
    First we consider the case  where $na \le (2-\epsilon)\log n$.
    We show that with probability at least
    $1-n^{-\Omega(\epsilon) }$, there are at least 
    $n^{\Omega(\epsilon)}$ distinct $\pi' \in \calT_2$
    such that $\iprod{A}{B} \le \langle A^{\pi'},B\rangle$. 
    This implies that the MLE $\hat{\pi}_{\rm ML}$
    coincides with $\pi$ with probability at most $n^{-\Omega(\epsilon)}$.
    Specifically, define $\chi_{ij}$ as the indicator random variable which equals to $1$
    if $  X_{ij,k}\le 0$ 
    for all $k \neq i,j$, and $0$ otherwise. 
    Then
    \begin{align*}
        \prob{\chi_{ij}=1} 
        =\prod_{k\neq i,j} 
        \prob{ X_{ij,k} \le 0 }  =\left( 1-a \right)^{n-2} 
        \ge n^{-2+\epsilon-o(1)},
    \end{align*}
    where the last inequality holds because $an \le (2-\epsilon)\log n$.
    Let $I=\sum_{1 \le i \le n/2}
    \sum_{n/2<j\le n}
    \chi_{ij}$.
    Then $\expect{I} \ge n^{\epsilon-o(1)}$.
    Next, we show $\var(I)=o\left(\expect{I}^2\right)$,
    which, by Chebyshev's inequality, implies that 
    $I \ge \frac{1}{2} \expect{I}$ with high probability. 
    Note that 
    $$
    \var(I)
    =\sum_{1 \le i, i' \le n/2 }
    \sum_{n/2<j, j' \le n}
    \prob{\chi_{ij} =1, \chi_{i'j'} =1}
    -\prob{\chi_{ij} =1}
    \prob{\chi_{i'j'} =1}.
    $$
    We break the analysis into the following four cases depending on the choice of $(i,j)$
    and $(i',j')$.
    
    Case 1: $i=i'$, and $j=j'$. In this case,
    $$
    \prob{\chi_{ij} =1, \chi_{i'j'} =1}
    -\prob{\chi_{ij} =1}
    \prob{\chi_{i'j'} =1}
    \le \prob{\chi_{ij} =1}.
    $$
    
    Case 2: $i=i'$, and $j\neq j'$.
    In this case, note that
    \begin{align*}
        \prob{\chi_{ij'}=1 \mid \chi_{ij}=1}
        & =\prod_{k \neq i,j'}
        \prob{ X_{ij,k}\le 0 
            \mid \chi_{ij}=1} \\
        & = \prob{X_{ij',j} \le 0 
            \mid X_{ij,j'} \le 0}  \times  \prod_{k \neq i,j,j'}
        \prob{ X_{ij',k}\le 0 
            \mid  X_{ij,k} \le 0} 
        \\
        & \le \left( 1- p_{00}p_{11} \right)^{n-2},
    \end{align*}
    where the last inequality holds because 
    \begin{align*}
        \prob{ X_{ij',k} \le 0 
            \mid  X_{ij,k} \le 0} 
        &=\frac{ p_{11} \left( 1-p_{00}\right)^2 
            + p_{01} + p_{10} + p_{00} \left(1- p_{11} \right)^2 }{1-2p_{00}p_{11}} \\
        &= \frac{1-4p_{00}p_{11} + p_{00}p_{11}
            \left(p_{00}+p_{11}\right) }{1-2p_{00}p_{11}}\\
        & \le \frac{1-3p_{00}p_{11}  }{1-2p_{00}p_{11}} \le 1- p_{00}p_{11},
    \end{align*}
    and similarly for $ \prob{X_{ij',j} \le 0 
        \mid X_{ij,j'} \le 0} $.
    

    Case 3: $i\neq i'$, and $j= j'$.
    Analogous to Case 2, we have
    $$
    \prob{\chi_{i'j}=1 \mid \chi_{ij}=1}
    \le \left( 1- p_{00}p_{11} \right)^{n-2}
    $$
    
    Case 4: $i \neq i'$, and $ j\neq j'$.
    In this case,
    \begin{align*}
        \prob{\chi_{i'j'}=1 \mid \chi_{ij}=1}
        & =\prod_{k \neq i',j'}
        \prob{ X_{i'j',k} \le 0 
            \mid \chi_{ij}=1} \\
        & \le \prod_{k \neq i,i',j,j'}
        \prob{ X_{i'j',k} \le 0 
        } \\
        &= \left( 1- 2p_{00}p_{11}\right)^{n-4},
    \end{align*}
    where the first inequality holds because
    for all $k \neq i,i',j,j'$, $X_{i'j',k}$
    are independent from  $\{ X_{ij,k}\}_{k \neq i,j}$.
    Therefore,
    \begin{align*}
        \prob{\chi_{ij} =1, \chi_{i'j'} =1}
        -\prob{\chi_{ij} =1}
        \prob{\chi_{i'j'} =1}
        & = \left( 1- 2p_{00}p_{11}\right)^{2n-6}
        \left( 1 -  \left( 1- 2p_{00}p_{11}\right)^2\right) \\
        & \le 4 p_{00}p_{11}
        \left( 1- 2p_{00}p_{11}\right)^{2n-6}.
    \end{align*}
    
    Combining all four cases above, we get that 
    \begin{align*}
        \var(I) \le \expect{I}
        \left(1 + n \left( 1-  p_{00}p_{11}\right)^{n-2} 
        +n^2 p_{00}p_{11} \left( 1-  2p_{00}p_{11}\right)^{n-4}\right)
    \end{align*}
    Therefore,
    $$
    \frac{  \var(I)}{\expect{I}^2}
    \le 
    \frac{1 + n \left( 1-  p_{00}p_{11}\right)^{n-2}
        +n^2 p_{00}p_{11}\left( 1-  2p_{00}p_{11}\right)^{n-4}}
    {n^2 \left( 1-  2p_{00}p_{11}\right)^{n-2}}
    =O\left( n^{-\Omega(\epsilon)}\right),
    $$
    where the last inequality holds because $np_{00}p_{11} \le (1-\epsilon/2)\log n$.
    In conclusion, we have shown that 
    $I \ge n^{\Omega(\epsilon)}$
    with probability at least $1-n^{-\Omega(\epsilon)}$.
    This implies that the MLE exactly recovers the true permutation with probability at most $n^{-\Omega(\epsilon)}$.

    \medskip
    Next, we shift to the case where $na \ge (2-\epsilon)\log n$.
    The assumption \prettyref{eq:exact_negative_er-general} translates to  $n(\sqrt{a}-\sqrt{b})^2 \le 2(1-\epsilon)\log n$,
    we have 
    $$
    (2-\epsilon)\log n \left( 1- \sqrt{\frac{b}{a}}\right)^2 \le n a \left( 1- \sqrt{\frac{b}{a}}\right)^2
    \le 2(1-\epsilon)\log n.
    $$
    It follows that 
    $$
    \sqrt{\frac{b}{a}} \ge 1- \sqrt{\frac{2(1-\epsilon)}{2-\epsilon}}
    =1- \sqrt{1 - \frac{\epsilon}{2-\epsilon}}
    \ge \frac{\epsilon}{2(2-\epsilon)} \ge \frac{\epsilon}{4}
    $$
    and hence $\frac{\epsilon^2}{16} \le \frac{b}{a} \le 1 $. 
    Let $T$ be a set of $2m$ vertices where $m = \floor{n/\log^2 n}$. We further partition $T$ into $T_1 \cup T_2$ where $|T_1|=|T_2| = m$.
    Let $\calS$ denote the set of permutations $\hpi$ that consists of $n-2$ fixed points and a $2$-node orbit $(i,j)$ for some $i\in T_1$ and $j\in T_2$. Then $|\calS| = m^2$. 
    Next we show that
    \begin{align}
        \prob{\exists \pi'\in \calS \text{ s.t. }  \left\langle A,B\right\rangle - \left\langle A^{\hpi},B\right\rangle
            <0  } = 1-o(1). \label{eq:desired_negative_exact}
    \end{align}
    Fix a $\hpi \in \calS$ with $(i,j)$ as its $2$-node orbit, \ie, $\hpi(i)=j$,  $\hpi(j)=i$, and $\hpi(k)=k$ for any $k \in [n]\backslash \{i,j\}$. 
    Then $
    \left\langle A,B\right\rangle - \left\langle A^{\hpi},B\right\rangle 
    = X_{ij} + Y_{ij}, 
    $
    where
    \[
    X_{ij} \triangleq \sum_{k\in T^c}  X_{ij,k} ,\quad \text{ and } \quad  Y_{ij} \triangleq  \sum_{k\in T \backslash \{i,j\}}  X_{ij,k}. 
    \]
    Note that 
    $\expect{Y_{ij}} = \left(2m-2\right) (a-b)$ 
   and $\Var{\left[Y_{ij}\right]} \le \left(2m-2\right) \left(a+b\right)$. 
       Letting 
       \[
       \tau \triangleq \left(2m-2\right) (a-b) +  \sqrt{ \left(2m-2\right) \left(a+b\right) \log n},
       \] 
       by Chebyshev's inequality, 
       \begin{align}
               \prob{Y_{ij} \ge \tau} \le \frac{1}{\log n}. \label{eq:Y_tail_bobound}
       \end{align}
    Define 
    \[
    T' = \left\{(i,j) \in T_1 \times T_2:  Y_{ij} < \tau \right\}. 
    \]
    Then $\expect{\left| (T_1 \times T_2) \backslash T'\right| } \le m^2/\log n$.
    Hence, by Markov's inequality, 
    $\left| (T_1 \times T_2) \backslash T'\right|  \ge \frac{1}{2} m^2$
    with probability at most $ \frac{2}{\log n}$.
    Hence, we have $|T'|\ge \frac{1}{2} m^2$  with probability at least $1- \frac{2}{\log n}$.

    Note that crucially $T'$ is independent of $\{X_{ij}\}_{i \in T_1, j \in T_2}$. Thus, we condition on the event that 
    $|T'| \ge \frac{1}{2} m^2$ in the sequel.
    Define $I_n = \sum_{(i,j)\in T'} \indc{ X_{ij} \le -\tau}$. 
    To bound the $\Prob\{X_{ij} \le -\tau\}$ from below, we need the following reverse large-deviation estimate (proved at the end of this subsection):
    
    \begin{lemma}\label{lmm:inverse_Chernoff}
        Suppose $\{X_k\}_{i\in [n]} \iiddistr a \delta_{+1}+b\delta_{-1}+(1-a-b) \delta_0$ with some $a,b \in[0,1]$ such that $0\le a+b \le 1$, $1\le \frac{a}{b} = \Theta(1)$, $an = \omega(1)$, and $\left(\sqrt{a}-\sqrt{b}\right)^2 \le 2\frac{\log n}{n}$. For all $\tau$ such that $\tau=o\left( \sqrt{a n\log n}\right)$ and $\tau = o(an)$, and any constant $0<\delta<1$, 
        there exists $n_0$ such that for all $n\geq n_0$,
        \begin{align*}
            \prob{\sum_{k=1}^n X_k \le - \tau} \ge  \exp\left(-n \left(\sqrt{a}-\sqrt{b} \right)^2 - \frac{\delta}{2} \log n\right).
        \end{align*} 
    \end{lemma}
    
    To apply this lemma, 
    note that $\tau=O\left( \frac{n(a-b)}{\log^2 n} + \sqrt{\frac{na}{\log n} } \right)$, so 
    \begin{align*}
    \frac{\tau}{ \sqrt{a  n\log n }} & =  O\left( \left( \sqrt{a}-\sqrt{b} \right) \sqrt{\frac{n}{\log^5 n}}
    + \frac{1}{\log n} \right)
    =O\left( \frac{1}{\log n} \right) \\
    \frac{\tau}{ a n} & = O\left(\frac{1}{\log^2 n} + \frac{1}{\sqrt{a n \log n}}\right) =O\left( \frac{1}{\sqrt{\log n}}\right) ,
    \end{align*} 
    where we used the fact that $m  \le n / \log^2 n$, $b \leq a$,
    $\sqrt{a}-\sqrt{b} = O(\sqrt{\log (n)/n})$, 
    and $an=\omega(1)$.
    Then, applying  \prettyref{lmm:inverse_Chernoff} with $\delta = \frac{\epsilon}{4}$
    yields
    \begin{align}
        \expect{I_n} 
        \ge  
        \frac{1}{2} m^2 \exp\left(-|T^c| \left(\sqrt{a}-\sqrt{b} \right)^2 - \frac{\epsilon}{8} \log n\right) 
        \ge n^{\epsilon -o(1)}, 
        \label{eq:EIn}
    \end{align} 
    where the last inequality holds by assumption that $\left(\sqrt{a}-\sqrt{b} \right)^2 \le 2 \left(1-\epsilon\right)\frac{\log n}{n}$.
    Next, we show that $\Var{\left[I_n\right]}/\expect{I_n}^2=o(1)$,
    which, by Chebyshev's inequality, implies that $I_n \ge \frac{1}{2} \expect{I_n}$ with high probability. 
    
    Write
    \begin{align*}
        \Var{\left[I_n\right]} 
        & = \sum_{(i,j),(i',j')\in T'} \left( \prob{X_{ij}\le -\tau, X_{i'j'}\le -\tau }-\prob{X_{ij}\le -\tau }\prob{ X_{i'j'}\le -\tau} \right) \leq \termI +\termII,
    \end{align*}
     where
    \begin{align*}
        \termI & =  \sum_{(i,j)\in T'}  \prob{X_{ij}\le  -\tau } = \expect{I_n}, \\
        \termII & =  \sum_{(i,j),(i,j')\in T', j \neq j'}   \prob{X_{ij}\le  0, X_{ij'}\le 0} 
        +  \sum_{(i,j),(i',j)\in T', i \neq i'}  \prob{X_{ij}\le  0, X_{i'j}\le 0}.
    \end{align*}
    To  bound $\termII$, fix $(i,j),(i,j')\in T'$ such that $j\neq j'$. 
    Note that $ \{X_{ij,k}\}_{j\in T_2} \iiddistr \Bern\left(p_{00} \right)$ conditional on $ A_{ik}= 1, B_{ik} =1$; $  \{- X_{ij,k}\}_{j\in T_2} \iiddistr \Bern\left(p_{10} \right)$ conditional on $ A_{ik}= 1, B_{ik} =0$; $  \{-X_{ij,k}\}_{j\in T_2}\iiddistr \Bern\left(p_{01} \right)$ conditional on $A_{ik}= 0, B_{ik} =1 $; $  \{X_{ij,k}\}_{j\in T_2} \iiddistr \Bern\left(p_{11} \right) $ conditional on $A_{ik}= 0, B_{ik} =0$. 
    Then, for $\ell,\ell'\in \{0,1\}$, we define
    \[
    M_{\ell \ell'} = \left|\left\{ k\in T^c | A_{ik}= \ell, B_{ik} = \ell'\right\} \right|,
    \]
    and get that for any $\lambda \ge 0$,
    \begin{align*}
        &\prob{\sum_{k \in T^c} X_{ij,k} \le 0, \sum_{k \in T^c} X_{ij',k} \le 0}\\
        & \overset{(a)}{=} \expect{ \prob{\sum_{k \in T^c} X_{ij,k} \le 0 \bigg| M_{11}, M_{10}, M_{10}, M_{00}}  \prob{\sum_{k \in T^c} X_{ij',k} \le 0 \bigg| M_{11} , M_{10} , M_{10}, M_{00}} } \\\
        & \overset{(b)}{=}  \expect{ \gamma_{00}^{2M_{11}} \gamma_{01}^{2M_{10}} \gamma_{10} p_{10}^{2M_{10}}
            \gamma_{11}^{2M_{00}} }\\
        & \overset{(c)}{=} \left(
        \gamma_{00}^2 p_{11} + \gamma_{01}^{2} p_{10} + \gamma_{10}^{2} p_{01} +  \gamma_{11}^{2} p_{00}
        \right)^{|T^c|},
    \end{align*}
    where $\gamma_{\ell,\ell'}=1- p_{\ell\ell'}+ p_{\ell \ell'} e^{(2 | \ell-\ell'| -1) \lambda }$;
    $(a)$ holds because $\sum_{k \in T'} X_{ij,k}$ and $\sum_{k \in T'} X_{ij',k}$ are independent  conditional on $M_{11}, M_{10}, M_{10}, M_{00}$; $(b)$ holds by applying the Chernoff bound; $(c)$ holds by applying the MGF of the multinomial distribution, since $M_{11},M_{10},M_{10},M_{00} \sim \mathrm{Multi} (|T^c|, p_{11},p_{10},p_{01},p_{00})$.
Choosing  $e^{\lambda} = \sqrt{\frac{p_{00}p_{11}}{p_{01}p_{10}}} = \sqrt{\frac{a}{b}}$ where $a = 2 p_{00}p_{11}$ and $b = 2 p_{01}p_{10}$, we have
\begin{align*}
        \prob{\sum_{k \in T^c} X_{ij,k} \le 0, \sum_{k \in T^c} X_{ij',k} \le 0}\leq \left(1-\frac{3}{2} \left(\sqrt{a}-\sqrt{b}\right)^2\right)^{|T^c|}\,.
    \end{align*}
The same upper bound applies to  any $(i,j),(i',j)\in T'$ such that $i\neq i'$. 
    
    Therefore, we get that 
    \begin{align*}
        \termII 
        \le 2m^3  \left(1-\frac{3}{2} \left(\sqrt{a}-\sqrt{b}\right)^2\right)^{|T^c|}
        & \le 2 m^3 \exp \left( -\frac{3}{2} \left(\sqrt{a}-\sqrt{b}\right)^2 |T^c| \right).
    \end{align*}
    Hence, by Chebyshev's inequality,
    \begin{align*}
        \prob{I_n \le \frac{1}{2}\expect{I_n}}
        & \le \frac{\Var{\left[I_n\right]}}{\frac{1}{4}\left(\expect{I_n}\right)^2}  \\
        & \le \frac{4}{\expect{I_n} } + \frac{8 \times \termII}{\left(\expect{I_n}\right)^2} \\
        & \overset{(a)}{\le} n^{-\epsilon+o(1) } +  \frac{32}{m} \exp\left( \frac{1}{2} \left(\sqrt{a}-\sqrt{b}\right)^2 |T^c|
        + \frac{\epsilon}{8} \log n \right) \\
        & \overset{(b)}{\le} n^{-7\epsilon/8 +o(1)}
    \end{align*}
    where $(a)$ is due to \prettyref{eq:EIn}; $(b)$ holds by the assumption that $\left(\sqrt{a}-\sqrt{b}\right)^2 \le 2 \left(1-\epsilon\right)\frac{\log n}{n}$. 
    
    Therefore, with probability $1-n^{-\Omega(\epsilon)}$ there exists some $(i,j) \in T'$ such that $X_{ij} \le -\tau$. By definition of  $ T'$, we have $Y_{ij} < \tau$ and hence $X_{ij}+Y_{ij}<0$. Thus, we arrive at the desired claim~\prettyref{eq:desired_negative_exact}. 

    \begin{proof}[Proof of \prettyref{lmm:inverse_Chernoff}]
        Let $ E_n=\left\{\sum_{k=1}^n X_k \le - \tau \right\} $. Let $Q$ denote the distribution of $X_k$. 
        The following large-deviation lower  estimate based on the data processing inequality is well-known (see \cite[Eq.~(5.21), p.~167]{ckbook} and \cite[Lemma 3]{HajekWuXu_one_info_lim15}):
        For any $Q'$,
        \begin{align}
            Q \left(E_n \right) \ge \exp\left(\frac{- n D\left(Q'\| Q \right) -\log 2}{Q'(E_n)}\right). 
            \label{eq:Q_E_n}
        \end{align}

        Choose
        \[
        Q'= \frac{\alpha-\beta}{2} \delta_{+1} + \frac{\alpha +\beta}{2} \delta_{-1} + (1-\alpha) \delta_0 
        \,,
        \]
        where
        \[
        \alpha \triangleq \frac{2 \sqrt{ab}}{1-\left(\sqrt{a}-\sqrt{b}\right)^2}, \quad \text{ and } \quad  
        \beta \triangleq \min \left\{ \frac{\alpha}{2}, \, \frac{\delta}{8} \sqrt{\frac{ b \log n }{ n}} \right\}\, .
        \]
        Note that under the assumption that $ 1\le a/b= \Theta(1)$ and $\left(\sqrt{a}-\sqrt{b}\right)^2 \le \frac{\log n}{n}$,
        we have that $2b \le \alpha=\Theta(a)$. Moreover, since $\tau=o(\sqrt{an \log n})$ and $\tau=o(an)$, it follows that $\tau=o(\beta n)$.
        Then, 
        \begin{align*}
            Q' (E_n^c) =  Q'  \left(\sum_{k=1}^n Y_k \ge - \tau \right) 
            & \stepa{\le}  Q'  \left(\sum_{k=1}^n \left(Y_k - \expect{Y_k} \right)  \ge \beta n- \tau \right)\\
            & \stepb{\le} \frac{\sum_{k=1}^n \Var{\left[Y_k\right]}}{\left(\beta n - \tau  \right)^2 } \\
            &\stepc{\le}  \frac{  \alpha n }{\left(\beta n - \tau  \right)^2 }  \stepd{=} o(1) \, ,
        \end{align*}
        where $(a)$ holds because $ \expect{Y_k}= -  \beta $; $(b)$ follows from by Chebyshev's inequality as $\{Y_k\}_{i\in [n]}$ are independent and $\tau=o(\beta n)$; $(c)$ holds because $  \sum_{k=1}^n\Var{\left[Y_k\right]} \le \sum_{k=1}^n \expect{Y_k^2} \le  \alpha n$; $(d)$ holds because $\tau =o\left(\beta n \right)$, 
        and $ \frac{\alpha }{\beta^2 n} = \max\left\{ \frac{64  \alpha }{\delta^2 b \log n}, \frac{4}{\alpha n} \right\}= o(1)$, in view of that $\delta=\Theta(1)$, $\alpha= \Theta(a)=\Theta(b)$ and $\alpha n = \Theta(a n) = \omega(1)$.
        
        Next, we upper bound $ D \left(Q'\| Q \right)$. We get
        \begin{align*}
            D \left(Q'\| Q \right) 
            & = \frac{\alpha-\beta}{2} \log \frac{\alpha-\beta}{2a} +\frac{\alpha+\beta}{2} \log  \frac{\alpha+\beta}{2b} + (1-\alpha) \log \frac{1-\alpha}{1-a-b}  \\
            & \stepa{=} - \log \left(1- \left(\sqrt{a}-\sqrt{b}\right)^2\right) +\frac{\alpha}{2} \log \frac{\left(\alpha^2-\beta^2\right)}{\alpha^2}+\frac{\beta}{2} \log \frac{a(\alpha+\beta)}{b(\alpha-\beta)} \\
            & \stepb{\le} - \log \left(1- \left(\sqrt{a}-\sqrt{b}\right)^2\right)  +  \frac{\beta^2}{\alpha-\beta} +\beta \frac{ \sqrt{a}-\sqrt{b}}{\sqrt{b}} \\
            & \stepc{\le} \left(\sqrt{a}-\sqrt{b}\right)^2 + \frac{\delta \log n}{ 4n}
            \, ,
        \end{align*}
        where $(a)$ holds because $\alpha = \frac{2 \sqrt{ab}}{1-\left(\sqrt{a}-\sqrt{b}\right)^2}$; $(b)$ holds 
        because 
        $ \log \frac{\alpha+\beta}{\alpha-\beta} \le \frac{2\beta}{\alpha-\beta}$
        and $\frac{1}{2} \log \frac{a}{b} \le \frac{\sqrt{a}-\sqrt{b}}{\sqrt{b}}$,
        in view of $\log(1+x) \le x$; 
        $(c)$ holds for all sufficiently large $n$ because $\left(\sqrt{a}-\sqrt{b}\right)^2 \le 2\frac{\log n}{n}$
        so that $\log(1-(\sqrt{a}-\sqrt{b})^2)=-(\sqrt{a}-\sqrt{b})^2 
        +O\left(\log^2(n)/n^2\right)$,
        $\beta \le \frac{1}{2}\alpha$ and $\beta \le \frac{\delta }{8} \sqrt{\frac{b\log n}{n}}$ so that 
        $\frac{\beta^2}{\alpha-\beta} \le \frac{2\beta^2}{\alpha} \le \frac{\delta^2 b\log n}{32\alpha n} \le \frac{\delta^2 \log n}{32 n} $,
        and $\beta \frac{ \sqrt{a}-\sqrt{b}}{\sqrt{b}} \le  \frac{\delta }{8 } \sqrt{\frac{\log n}{n}} \left(\sqrt{a} -\sqrt{b}\right) \le  \frac{\sqrt{2}\delta\log n }{8n }$. 
        
        
        
        In conclusion,    it follows from~\prettyref{eq:Q_E_n} that   
        \begin{align*}
            Q \left(E_n \right) &\ge \exp\left( - \left(1+o(1) \right)   \left( n \left(\sqrt{a}-\sqrt{b}\right)^2  + \frac{\delta}{4} \log n+ \log 2 \right) \right) \\
            &\ge  \exp\left( -  n \left(\sqrt{a}-\sqrt{b}\right)^2  - \frac{\delta}{2} \log n\right),
        \end{align*}
        where the last inequality holds for all sufficiently large $n$ in view of $n \left(\sqrt{a}-\sqrt{b}\right)^2 \le 2\log n $.
    \end{proof}

    \appendix 
    

    \section{Moment Generating Functions Associated with Edge Orbits} \label{app:expect_exp_cal}

            \begin{lemma}\label{lmm:expect_exp_AB_pi}
        Fixing $\pi$ and $\hpi$, where $\pi$ is the latent permutation under $\calP$. 
        For any edge orbit $\orbit$ of $\sigma= \pi^{-1} \circ \hpi $ with $|\orbit|=k$ 
        and  $t\ge0$, 
        let $M_k \triangleq \Expect \left[ \exp\left(t \sum_{(i,j)\in O} \overline{A}_{\hpi(i)\hpi(j) } \overline{B}_{ij}\right) \right]$,
        where 
        $\overline{A}_{ij}=\overline{A}_{ij}-\Expect A_{ij}$ and
        $\overline{B}_{ij}=\overline{B}_{ij}-\Expect B_{ij}$.
        \begin{itemize}
            \item For \ER random graphs, 
            \begin{align}
                M_k=\left(\frac{T - \sqrt{T^2 -4D}}{2}\right)^{k} + \left(\frac{T + \sqrt{T^2 -4D}}{2}\right)^{k}, \label{eq:ER_expect_exp_AB_pi}
            \end{align}
            where
            \begin{align*}
            T &= e^{t(1-q)^2} qs + 2 e^{-tq(1-q)}q(1-s) + e^{tq^2}(1-2q+qs) \\ D &= \left(e^{t(1-q)^2+tq^2}-e^{-2tq(1-q)}\right)q(s-q)    
            \end{align*}
            and $q\triangleq ps$.   
            \item For Gaussian   model, for $t \le \frac{1}{1+\rho}$,
            \begin{align}
                M_k = \left[\left(\frac{\lambda_1+\lambda_2}{2}\right)^k - \left(\frac{\lambda_1-\lambda_2}{2}\right)^{k}\right]^{-1} \label{eq:GW_expect_exp_AB_pi}
            \end{align}
            where $\lambda_1 = \sqrt{ (1+\rho t)^2 - t^2 }$ and $\lambda_2 = \sqrt{ (1- \rho t)^2 - t^2 }$.
        \end{itemize}
        Moreover, $M_k \le M_2^{k/2}$ for all $k\ge 2$.
    \end{lemma}
    \begin{proof}
        For ease of notation, let $\{(a_{\iii},b_{\iii})\}_{\iii=1}^k$ be independently and identically distributed pairs of random variables with joint distribution $P$.
        Let $a_{k+1}=a_1$ and $b_{k+1}=b_1$. 
        Since $O$ is an edge orbit of $\sigmae$, we have $\{A_{\pi(i)\pi(j)}\}_{(i,j)\in \orbit}=\{A_{\hpi(i)\hpi(j)}\}_{(i,j)\in \orbit}$ and $(\hpi(i),\hpi(j)) =  (\pi(\sigma(i)), \pi(\sigma(j) ) )$. 
        Then, we have that 
            \begin{align*}
            M_k
            & = \Expect\left[\exp\left(\sum_{\iii =1}^k t (a_{\iii +1} - \expect{a_{\iii +1}})  (b_{\iii}-\expect{b_{\iii +1}})\right)\right]\\
            & = \Expect\left[ \expect{\exp\left(\sum_{\iii=1}^k t (a_{\iii +1} - \expect{a_{\iii +1}})  (b_{\iii}-\expect{b_{\iii +1}})\right)\bigg | a_1,a_2,\cdots, a_k}\right]\\
            & = \Expect\left[ \prod_{\iii=1}^k  \expect{\exp \bigg( t \left( a_{\iii +1} - \expect{a_{\iii +1}}\right)  \left(b_{\iii}-\expect{b_{\iii +1}}\right)\bigg)\bigg | a_{\iii}, a_{\iii+1}}\right].
        \end{align*}
        \begin{itemize}
            \item For \ER random graphs,
            $\Expect[a_i]=\Expect[b_i]=q$, and
            $
            M_k  =  \mathrm{tr}\left(M^{k}\right),
            $
            where $M$ is a $2\times 2$ row-stochastic matrix with rows and columns indexed by $\{0,1\}$ and 
            \begin{align*}
                M(\ell,m) & = \Expect \left[\exp\bigg( t \left( a_{\iii +1} - q\right)  \left(b_{\iii}-q\right)\bigg)  \bigg | a_{\iii}= \ell, a_{\iii+1} = m \right] \Prob\left(a_{\iii+1} = m \right). \\
                &= \Expect \left[\exp\bigg( t \left( m - q\right)  \left(b_{\iii}-q\right)\bigg)  \bigg || a_{\iii} = \ell \right] \Prob\left(a_{\iii+1} = m \right). 
            \end{align*}
            Explicitly, we have
            \[
            M = \begin{pmatrix}
               e^{tq^2} \left(1-\eta\right)\left(1-q\right) + e^{-tq(1-q) } 
               \eta\left(1-q\right)  & e^{-tq(1-q) }   \left(1-\eta\right)q + e^{t(1-q)^2}\eta q  \\
                e^{tq^2}  s \left(1-q\right) + e^{-tq(1-q) } 
               \left(1-s\right)\left(1-q\right)   & e^{-tq(1-q) }  \left(1-s\right)  q + e^{t(1-q)^2} sq
            \end{pmatrix} \, ,
            \]
            where $ \eta=\frac{q(1-s)}{1-q} $. 
            The eigenvalues of $M$ are $\frac{T - \sqrt{T^2 -4D}}{2} $ and $\frac{T + \sqrt{T^2 -4D}}{2}$, where
            \begin{align*}
                T  & \triangleq \mathrm{Tr}(M) =e^{t(1-q)^2} qs  + 2 e^{-tq(1-q)}q(1-s) + e^{tq^2}(1-2q+qs) \\
                D & \triangleq  \det(M) =\left(e^{t(1-q)^2+tq^2}-e^{-2tq(1-q)}\right)q(s-q) 
            \, .
            \end{align*}
        %
            \item For Gaussian   model,
            $\Expect[a_i]=\Expect[b_i]=0$, and
            \begin{align*}
                M_k
                = \Expect\left[ \prod_{\iii=1}^k  \exp\left( t \rho  a_{\iii} a_{\iii+1}\, +  \frac{1}{2} t^2 \left(1-\rho^2\right)
                 a_{\iii+1}^2  \right)\right],
            \end{align*}
            where the equality follows from  $b_{\iii} \sim \calN(\rho a_{\iii}, 1-\rho^2)$ conditional on $a_{\iii} $ 
            and $\Expect\left[\exp\left(t Z\right)\right] = \exp\left(t \mu + t^2\nu^2/2\right)$ for $Z\sim \calN(\mu,\nu^2)$. 

            Let  $\lambda_1 = \sqrt{\left(1+\rho t \right)^2 - t^2 }$ and $\lambda_2 = \sqrt{\left(1- \rho t \right)^2 - t^2 }$, where $t \le \frac{1}{1+\rho}$.
                By change of variables, 
            \begin{align*}
                M_k
                & = \int \cdots \int \prod_{\iii=1}^{k} \frac{1}{\sqrt{2\pi} } \exp\left(-\frac{\left(\frac{\lambda_1+\lambda_2}{2} a_{\iii} + \frac{\lambda_1-\lambda_2}{2} a_{\iii+1}\right)^2}{2} \right)\mathrm{d} a_1 \cdots \mathrm{d} a_{k}\\
                & \stepa{=} \int \cdots \int \prod_{\iii=1}^{k} \frac{1}{\sqrt{2\pi} } \exp\left(-\frac{X_{\iii}^2}{2} \right)\mathrm{d} X_1 \cdots \mathrm{d} X_{k} \det(J^{-1})\\
                & \stepb{=} \frac{1}{\left(\frac{\lambda_1+ \lambda_2}{2}\right)^{k}-\left(\frac{\lambda_1- \lambda_2}{2}\right)^{k}}, 
            \end{align*}
            where (a) holds by letting $X_{\iii} = \frac{\lambda_1+\lambda_2}{2} a_{\iii}+ \frac{\lambda_1-\lambda_2}{2} a_{\iii+1}$, and denoting $J$ as the Jacobian matrix with $J_{ij} = \frac{\partial X_k }{\partial a_j}$ whose inverse matrix is
            \begin{align*}
                J = 
                \begin{pmatrix} \frac{\lambda_1+\lambda_2}{2}      & \frac{\lambda_1- \lambda_2}{2}      & 0     & \cdots     & 0 \\
                    0   &   \frac{\lambda_1+\lambda_2}{2}          & \frac{\lambda_1- \lambda_2}{2}     & \cdots     &0 \\
                    0  &  0     &    \frac{\lambda_1+\lambda_2}{2}        & \cdots     & 0 \\
                    \vdots  & \vdots     & \vdots     & \ddots     & \frac{\lambda_1- \lambda_2}{2} \\
                    \frac{\lambda_1- \lambda_2}{2}  & 0 &  \cdots  0 & 0 &   \frac{\lambda_1+\lambda_2}{2}    \\
                \end{pmatrix} \, ;  
            \end{align*}
            (b) holds because $\det(J^{-1}) = \det(J)^{-1}$, where $\det(J) = \left(\frac{\lambda_1+ \lambda_2}{2}\right)^{k}-\left(\frac{\lambda_1- \lambda_2}{2}\right)^{k}$. 
        \end{itemize}

        Finally, we prove that $M_k \le (M_2)^{k/2}$ for $k \ge 2$.
        Indeed, for the \ER graphs, this simply follows from $x^k + y^k \le(x^2 +y^2 )^{k/2}$ for $x, y \ge 0$ and $k \in \naturals$. 
        Analogously, for the Gaussian   model, this follows from $(a+b)^k - (a-b)^k \ge (4ab)^{k/2}$, which holds by rewriting $(a+b)^2=(a-b)^2 + 4 ab$ and letting $x=(a-b)^2$ and $y=4ab$. 
    \end{proof}
        
        
        \begin{lemma}\label{lmm:expect_exp_AB_pi_minus_AB_pi_*}
        Fixing $\pi$ and $\tpi$, where $\pi$ is the latent permutation under $\calP$. For any edge orbit $\orbit$ of $\sigma=   \pi^{-1} \circ \hpi$ with $|O|=k$ and $t>0$, let $L_k \triangleq\Expect \left[ \exp\left(t \sum_{(i,j)\in O}A_{\hpi(i)\hpi(j)}B_{ij}- \sum_{(i,j)\in O}A_{\pi(i)\pi(j)}B_{ij}\right)  \right]$.
        \begin{itemize}
                \item For general \ER random graphs model, 
            \begin{align}
                L_k
                & = \left(\frac{T - \sqrt{T^2-4D}}{2}\right)^{k} + \left(\frac{T + \sqrt{T^2 -4D}}{2}\right)^{k},  \label{eq:ER_expect_exp_AB_pi_minus_AB_pi_*}
            \end{align}
            where $T=1$ and $D =- \left(p_{01}p_{10}\left(e^t-1\right) +p_{00}p_{11}\left(e^{-t}-1\right) \right)$. 
            
                
            \item For Gaussian   model, for $ t \le  \frac{1}{2(1-\rho)}$, 
            \begin{align}
                L_k  = \left[\left(\frac{\lambda_1+ \lambda_2}{2}\right)^{k}-\left(\frac{\lambda_1- \lambda_2}{2}\right)^{k}\right]^{-1},  \label{eq:GW_expect_exp_AB_pi_minus_AB_pi_*}
            \end{align}
            where $\lambda_1 = 1$ and $\lambda_2 = \sqrt{1 + 4 t \rho -4 t^2 \left(1-\rho^2 \right)}$.
        \end{itemize}
        Moreover, $L_k \le L_2^{k/2}$ for all $k\ge 2$.
    \end{lemma}
        \begin{proof}
        For ease of notation, let $\{(a_{\iii},b_{\iii})\}_{\iii=1}^k$ be independently and identically distributed pairs of random variables with joint distribution $P$. Let $a_{k+1}=a_1$ and $b_{k+1}=b_1$. 
        Since $O$ is an edge orbit of $\sigmae$, we have $\{A_{\pi(i)\pi(j)}\}_{(i,j)\in \orbit}=\{A_{\hpi(i)\hpi(j)}\}_{(i,j)\in \orbit}$ and $(\hpi(i),\hpi(j)) =  \pi(\sigma(i), \sigma(j))$. Then, we have that 
        \begin{align*}
            L_k
            & = \Expect \left[ \exp\left(t \sum_{(i,j)\in O} \left(A_{\hpi(i)\hpi(j)}  -A_{\pi(i)\pi(j)} \right)\right)B_{ij} \right] \\
            & = \Expect\left[ \expect{\exp\left(\sum_{\iii=1}^k t \left(a_{\iii+1} -a_{\iii} \right) b_{\iii}\right)\bigg | a_1,a_2,\cdots, a_k}\right]\\
            & = \Expect\left[ \prod_{\iii=1}^k  \expect{\exp\left( t \left( a_{\iii+1} -a_{\iii}  \right) b_{\iii}\right) \bigg | a_{\iii}, a_{\iii+1} } \right].
        \end{align*}
        \begin{itemize}
            \item For \ER random graphs, $L_k=\mathrm{tr}\left(L^{k}\right) $,
            where $L$ is a $2\times 2$ row-stochastic matrix with rows and columns indexed by $\{0,1\}$ and 
            \begin{align*}
                L(\ell,m) & = \Expect \left[\exp\left(t \left(a_{\iii+1} -a_{\iii} \right)  b_{\iii} \right) | a_{\iii} = \ell, a_{\iii+1}= m \right] \Prob\left(a_{\iii+1}= m \right). \\
                &= \Expect \left[\exp\left(t (m-\ell) b_{\iii} \right) | a_{\iii} = \ell \right] \Prob\left(a_{\iii+1} = m \right).
            \end{align*}
            Explicitly, we have
            \[
            L = \begin{pmatrix}
            1-p_{10}-p_{11} & \left(\frac{p_{01}}{  1-p_{10}-p_{11} } e^t + \frac{p_{00}}{ 1-p_{10}-p_{11} }\right)(p_{10}+p_{11})\\
            \left(\frac{p_{11}}{p_{10}+p_{11}} (e^t-1)+1\right)(p_{10}+p_{11})&     p_{10}+p_{11} 
            \end{pmatrix} \, .
            \]
            The eigenvalues of $L$ are $\frac{T- \sqrt{T^2-4D}}{2} $ and $\frac{T + \sqrt{T^2 -4D}}{2}$, where $T=1$ and 
            \[ D 
            =- \left(p_{01}p_{10}\left(e^t-1\right) +p_{00}p_{11}\left(e^{-t}-1\right) \right). \] 
            
            \item For Gaussian   model, 
            \begin{align*}
                L_k
                & =\Expect\left[ \prod_{\iii=1}^k  \exp\left( t \left(a_{\iii+1} -a_{\iii} \right) \rho a_{\iii} +  t^2 \left(a_{\iii+1}-a_{\iii} \right)^2 \left(1-\rho^2\right)/2  \right)\right]\\
                & = \Expect\left[ \exp\left(\left( t^2  \left(1-\rho^2\right) -  t \rho \right) \sum_{\iii=1}^k\left( a_{\iii}^2 - a_{\iii} a_{\iii+1} \right)\right)\right] \\
                &= \frac{1}{\left(\frac{\lambda_1+ \lambda_2}{2}\right)^{k}-\left(\frac{\lambda_1- \lambda_2}{2}\right)^{k}}. 
            \end{align*}
            where the first equality follows from  $b_{\iii} \sim \calN(\rho a_{\iii}, 1-\rho^2)$ conditional on $a_{\iii}$ 
            and $\Expect\left[\exp\left(t Z\right)\right] = \exp\left(t \mu + t^2\nu^2/2\right)$ for $Z\sim \calN(\mu,\nu^2)$ 
            ;
            the last equality holds by change of variables and
            Gaussian integral analogous to the proof of \prettyref{lmm:expect_exp_AB_pi}. 
        \end{itemize}
        Finally, $L_k \le (L_2)^{k/2}$ for $k \ge 2$ follows from the same reasoning as in \prettyref{lmm:expect_exp_AB_pi}. 
    \end{proof}

  \section{Proof of \prettyref{lmm:M_2}}\label{app:M_2}
     In view of  \prettyref{eq:ER_expect_exp_AB_pi} in \prettyref{lmm:expect_exp_AB_pi}, 
        $M_1=T$  and 
        \begin{align}
             M_2=T^2-2D \, , \label{eq:M_2}
        \end{align}
        where 
        by letting $a = e^{t(1-q)^2}$, $b=e^{-tq(1-q)}$, and $c =e^{tq^2}$, we get
        \[
         T = a qs + 2 b q(1-s) + c (1-2q+qs), \quad D =  \left(ac -b^2 \right)q(s-q) \, . 
        \]
        \begin{itemize}
            \item  Suppose $t \le \log \frac{1}{p}$. Since $ 1+ x \le  e^{x} \le 1+ x + x^2$ for $ 0 \le x \le 1$ and $ 1+ x+x^2/2 +x^3 \le  e^{x} \le 1+ x + x^2/2$ for $-1 \le x < 0$, we get  
         \begin{align*}
             1- tq(1-q) + \frac{1}{2} t^2 q^2(1-q)^2 -  t^3 q^3 & \le b \le 1- tq(1-q) + \frac{1}{2} t^2 q^2(1-q)^2 \\
             1+ tq^2  
             & \le c \le 1+ tq^2 + t^2 q^4 \,.
         \end{align*}
        It follows that
        \begin{align*}
             T 
            & \le e^{t(1-q)^2}  qs + 2 \left(1- tq(1-q) + \frac{1}{2} t^2 q^2(1-q)^2 \right) q(1-s) + \left(1+ tq^2 + t^2 q^4\right) (1-2q+qs) \\
            & = \alpha_0 + \alpha_1 t + \alpha_2 t^2 
            \, ,
        \end{align*}
        where
        \begin{align}
                \alpha_0 
                & =  2 q(1-s) + (1-2q+qs) + e^{t(1-q)^2}  qs\label{eq:alpha_10}  \, ,\\
                \alpha_1
                & =- 2 q (1-q)q (1-s) +  q^2 (1-2q+qs) \label{eq:alpha_11}  \, ,\\
                \alpha_2 
                & =  q^2(1-q)^2 q(1-s) +  q^4 (1-2q+qs)  \, . \label{eq:alpha_12}
        \end{align}
        Moreover, 
        \begin{align*}
            D 
            & \ge e^{t(1-q)^2}\left(1+ tq^2 \right) q(s-q)  -\left(1- tq(1-q) + \frac{1}{2} t^2 q^2(1-q)^2 \right)^2 q(s-q)\\
            & = 
            \beta_0 + \beta_1 t + \beta_2 t^2+  \beta_3 t^3 + \beta_4 t^4 
            \, , 
        \end{align*}
        where
        \begin{align}
            \beta_0 
            & = - q(s-q) +  e^{t(1-q)^2} q(s-q) \label{eq:beta_10} \, ,\\
            \beta_1 
            & = 2 q^2  (1-q) (s-q) +  e^{t(1-q)^2}q^3 (s-q)\label{eq:beta_11} \, ,\\
            \beta_2 
            & =- 2 q^3 (1-q)^2(s-q) \label{eq:beta_12} \, ,\\
            \beta_3
            & = q^4 \left(1-q\right)^3 (s-q) \label{eq:beta_13} \, ,\\
            \beta_4
            & = -\frac{1}{4} q^5\left(1-q\right)^4(s-q) \label{eq:beta_14}  \, . 
        \end{align}
        By \prettyref{eq:M_2}, we get 
        \begin{align*}
            M_2  = 
            T^2 - 2D 
            & \le \left(\alpha_0 + \alpha_1 t + \alpha_2 t^2\right)^2  - 2\left(\beta_0 + \beta_1 t + \beta_2 t^2+  \beta_3 t^3 + \beta_4 t^4 \right) \\
            & = \gamma_0 + \gamma_1 t + \gamma_2 t^2 +\gamma_3 t^3 +\gamma_4 t^4 
        \end{align*}
        where 
        \begin{align}
              \gamma_0 
            & = \alpha_0^2 -2 \beta_0  \label{eq:gamma_10}\\
            \gamma_1
            & = 2 \alpha_0\alpha_1 -2 \beta_1  \label{eq:gamma_11}\\
              \gamma_2 
            & = \alpha_1^2 + 2\alpha_0\alpha_2 -2\beta_2  \label{eq:gamma_12}\\
             \gamma_3 
            & =  2 \alpha_1 \alpha_2 - 2\beta_3  \label{eq:gamma_13}\\
            \gamma_4 
            & = \alpha_2^2 - 2\beta_4   \label{eq:gamma_14}
        \end{align}
        By \prettyref{eq:alpha_10}, \prettyref{eq:beta_10} and \prettyref{eq:gamma_10},  we get
           \begin{align*}
            \gamma_0 
            & = \alpha_0^2 -2 \beta_0   \\
             & =  \left(  1-qs + e^{t(1-q)^2}  qs \right)^2 - 2 \left(- q(s-q) +  e^{t(1-q)^2} q(s-q)  \right)  \\
            & = 1-2q^2+q^2s^2 + 2e^{t\left(1-q\right)^2} q^2 \left(1-s^2\right)+ e^{2t\left(1-q\right)^2}q^2s^2  \, ,\\
            & \le 1-2q^2+q^2s^2 +2e^{t} q^2 \left(1-s^2\right)+  e^{2t}q^2s^2  \, . 
        \end{align*}
        Note that the above bound on $\gamma_0$ determines the main terms in our final bound on $M_2$. 
        For the remaining terms, we will show
        $\sum_{i=1}^4\gamma_i t^i = O(q^3 t (1+t) )$ so that they can be made negligible later when we apply the bound on $M_2$ in the proof of \prettyref{lmm:ER_fixing_k} (see \prettyref{eq:extra_term}). 
        
        By \prettyref{eq:alpha_10}, \prettyref{eq:alpha_11},  \prettyref{eq:beta_11} and \prettyref{eq:gamma_11}, we get
         \begin{align*}
            \gamma_1
            & = 2 \alpha_0\alpha_1 -2 \beta_1 \\
            & =-2q^2+4q^3+4q^3s-4q^3s^2-4q^4+2q^4s^2 +  \left(-4q^3s+4q^3s^2+2q^4-2q^4s^2\right) e^{t\left(1-q\right)^2} 
           \\
            & \overset{(
            a)}{\le}   4q^3 +2  e^{t(1-q)^2} q^4  \overset{(
            b)}{\le}   6q^3 \, ,
        \end{align*}
        where $(a)$ holds because $-2q^2 + 4q^3 \le 0 $ given $q \le \frac{1}{2}$,  $ -4q^3s^2 - 4q^4 + 2q^4s^2 \le 0 $ and $ -4q^3s + 4q^3s^2 + 2q^4 - 2q^4s^2 \le 2q^4 $; $(b)$ holds because 
        $ 2 e^{ t \left(1-q\right)^2} q^4 \le 2 e^t q^4 \le 2 q^3 $,
        given $t \le \log \frac{1}{p}$. By \prettyref{eq:alpha_10}, \prettyref{eq:alpha_11},  \prettyref{eq:alpha_12}, \prettyref{eq:beta_12} and \prettyref{eq:gamma_12}, we get 
        \begin{align*}
            \gamma_2 
            & = \alpha_1^2 + 2\alpha_0\alpha_2 -2\beta_2 \\
            & = 2q^3 -2q^3s-q^4-4q^4s+6q^4s^2+10q^5s-8q^5s^2-6q^6-4q^6s+q^6s^2 +6q^7 +2q^7s-2q^8\\
            &~~~~ +e^{t\left(1-q\right)^2}\left(2q^4s-2 q^4s^2-2q^5s+ 4q^5s^2 +2q^6s \right)  \\
            & \overset{(a)}{\le } 2q^3 -2q^3s+q^4s^2  +10q^5s  +2 e^{t}q^4s  \overset{(b)}{\le }  6q^3 \, ,
        \end{align*}
        where $(a)$ holds because $-q^4-4q^4s+6q^4s^2 \le q^4 s^2$, $-8q^5s^2-6q^6-4q^6s+q^6s^2 +6q^7 +2q^7s-2q^8 \le 0$, and $-2 q^4s^2-2q^5s+ 4q^5s^2 +2q^6s \le 0$ given $q\le \frac{1}{2}$; $(b)$ holds because $ -2q^3s+q^4s^2  +10q^5s \le 2 q^3$ given $q\le \frac{1}{2}$, $2 e^{t}q^4s  \le 2q^3 s \le 2q^3$ given $t \le \log \frac{1}{p}$. By \prettyref{eq:alpha_11}, \prettyref{eq:alpha_12}, \prettyref{eq:beta_13} and \prettyref{eq:gamma_13}, we get
         \begin{align*}
            \gamma_3 
            & =  2 \alpha_1 \alpha_2 - 2\beta_3 \\
            & = -2q^4s+12q^5s-4q^5s^2-4q^6-16q^6s+10q^6s^2+8q^7-4q^7s^2-2q^8+ 2q^8s \\
            & \overset{(a)}{\le} -2q^4s+12q^5s \overset{(b)}{\le}  4 q^4 \, ,
        \end{align*}
        where $(a)$ holds because $-4q^5s^2-4q^6-16q^6s+10q^6s^2+8q^7  \le 0$ and $-4q^7s^2-2q^8+ 2q^8s  \le 0$; 
        $(b)$ holds by $q\le \frac{1}{2}$.  
        By \prettyref{eq:alpha_12}, \prettyref{eq:beta_14} and \prettyref{eq:gamma_14}, we get 
       \begin{align*}
            \gamma_4 
            & = \alpha_2^2 - 2\beta_4 \\
            & = \frac{1}{2}(q^5s+q^6-8q^6s+2q^6s^2+18q^7s-8q^7s^2-8q^8-8q^8s+8q^8s^2+8q^9-7q^9s+q^{10}) \\
            &  \overset{(a)}{\le}  \frac{1}{2}(q^5s+q^6-8q^6s+2q^6s^2+18q^7s)   \overset{(b)}{\le}  2q^5  \, ,
        \end{align*}
        where $(a)$ holds because $-8q^7s^2-8q^8-8q^8s+8q^8s^2+8q^9 \le 0$, and $-7q^9s+q^{10} \le 0$; $(b)$ holds because $-8q^6s + 2q^6s^2+18q^7s \le  4 q^6s$ and $q^5s+q^6+4 q^6s \leq 4 q^5$, given $q\le \frac{1}{2}$. 
       
        Combining the above bounds, we have
        \begin{align*}
            M_2 
            & \le   1- 2q^2  +q^2s^2 +  6q^3 t+  6q^3 t^2 + 4 q^4 t^3 +   2 q^5  t^4  + 2 e^{t}q^2 (1-s^2)+ e^{2t}  q^2s^2  \\
            & =  1- 2q^2  +q^2s^2 + q^3 t \left(6 +  6t  + 4 q t^2 +   2 q^2  t^3\right) +   2 e^{t}q^2 (1-s^2)+ e^{2t}  q^2s^2 \\  
            & \le  1- 2q^2  +q^2s^2 + 10 q^3 t(1+t)+ 2 e^{t}q^2 (1-s^2)    + e^{2t}  q^2s^2  \, ,
        \end{align*}
        where the last inequality holds because 
        $qt^2 \le 1$ and $q^2 t^3 \le 1$  given $t \le \log \frac{1}{p}$. 
        

        Moreover,  given $M_1 = T$ and  $M_2 = T^2 -2D$, 
        \begin{align*}
            \frac{M_1^2}{M_2} =  1+ \frac{2D}{M_2} \le 1+ 2D \le e^2 \, , 
        \end{align*}
        where the first inequality holds because $M_2 \ge 1$ by definition, 
        and the last inequality holds because 
        \begin{align*}
            D = (ac-b^2)q(s-q)   \overset{(a)}{\le} ac q(s-q)
            & \le e^{t}\left(1+ tq^2 + t^2 q^4\right) q(s-q)  \overset{(b)}{\le}  1+ tq^2 + t^2 q^4  \overset{(c)}{\le}  3 \, ,
        \end{align*}
        where $(a)$ holds because $b^2 q(s-q) \ge 0$; $(b)$ holds because  $e^t q \le 1$ given $t \le \log \frac{1}{p}$; $(c)$ holds because $ qt \le 1$, given $t \le \log \frac{1}{p} $.

        \item  Suppose $t \le \frac{1}{2^{10}}$.  Since $1+ x+x^2/2 \le  e^{x} \le 1+ x + x^2/2 + x^3$ for $ 0 \le x \le 1$ and $ 1+ x+x^2/2 +x^3 \le  e^{x} \le 1+ x + x^2/2$ for $-1 \le x< 0$, we get
        \begin{align*}
                1 +t (1-q)^2 + \frac{1}{2} t^2 (1-q)^4 
                & \le a \le 1 +t (1-q)^2 + \frac{1}{2} t^2 (1-q)^4 + t^3 (1-q)^6 \, , \\
                1- tq(1-q) + \frac{1}{2} t^2 q^2(1-q)^2 -  t^3 q^3 & \le b \le 1- tq(1-q) + \frac{1}{2} t^2 q^2(1-q)^2\, , \\
                1+ tq^2 + \frac{1}{2}t^2 q^4 
                & \le c \le 1+ tq^2 + \frac{1}{2}t^2 q^4 + t^3 q^6 \, .
        \end{align*} 
        It follows that
        \begin{align*}
            T 
            & \le \left(1 +t \left(1-q\right)^2 + \frac{1}{2} t^2 \left(1-q\right)^4 + t^3 \left(1-q\right)^6\right)qs  \\
            &~~~~ + 2\left(1- tq\left(1-q\right) + \frac{1}{2} t^2 q^2\left(1-q\right)^2\right)q\left(1-s\right) \\
            &~~~~ + \left(1+ tq^2 + \frac{1}{2}t^2 q^4 + t^3 q^6\right)\left(1-2q+qs\right) \\
            &  = 1+ \alpha_1 t+ \alpha_2 t^2 + \alpha_3 t^3 \, ,
        \end{align*}
        where
        \begin{align}
             \alpha_1 & = \left(1-q\right)^2qs- 2 q\left(1-q\right)q\left(1-s\right) +q^2 \left(1-2q+qs\right)= q(s-q) \label{eq:alpha_21} \, , \\
            \alpha_2 & = \frac{1}{2} \left(1-q\right)^4qs + q^2\left(1-q\right)^2 q\left(1-s\right) +\frac{1}{2} q^4 \left(1-2q+qs\right)  \label{eq:alpha_22} \, , \\
            \alpha_3 
            & = \left(1-q\right)^6 qs+q^6 \left(1-2q+qs\right) \label{eq:alpha_23} \, . 
        \end{align}
        And
        \begin{align*}
            D 
            & \ge \left(1 +t \left(1-q\right)^2 + \frac{1}{2} t^2 \left(1-q\right)^4 \right)\left(1+ tq^2 + \frac{1}{2}t^2 q^4 \right)q(s-q)\\
            &~~~~  -\left(1- tq\left(1-q\right) + \frac{1}{2} t^2 q^2\left(1-q\right)^2  \right)^2 q(s-q) \\
            & = \beta_1 t +\beta_2 t^2 + \beta_3 t^3 + \beta_4 t^4 
            \, ,
            \end{align*}
            where 
            \begin{align}
                \beta_1 & 
                =  \left(1-q\right)^2q(s-q)+ q^3 (s-q) + 2q\left(1-q\right)q(s-q) =q(s-q)  \label{eq:beta_21} \, , \\
                \beta_2 & =  \frac{1}{2} \left(1-q\right)^4q(s-q) +  \frac{1}{2}  q^5 (s-q)+ \left(1-q\right)^2 q^3 (s-q) \nonumber \\ 
                &~~~~ - q^2 \left(1-q\right)^2 q(s-q)-q^2\left(1-q\right)^2  q(s-q) \nonumber \\ 
                & = \frac{1}{2} \left(1-4q+4q^2\right)q(s-q) \label{eq:beta_22}  \, , \\
                \beta_3
                & = \frac{1}{2}(1-q)^2q^5(s-q) +\frac{1}{2} (1-q)^4q^3(s-q) + q(1-q)q^2(1-q)^2 q(s-q) \nonumber \\
                & = \frac{1}{2}\left(q^2-2q^3+q^4\right)q(s-q)\label{eq:beta_23}  \, , \\
                \beta_4
                  &  = \frac{1}{4}(1-q)^4q^5(s-q)- \frac{1}{4}(1-q)^4q^5(s-q)
                  = 0  \, . \label{eq:beta_24}
            \end{align}
        By \prettyref{eq:M_2}, we get that 
        \begin{align*}
            M_2 
             = T^2 -2D 
            & \le  \left(1+ \alpha_1 t+ \alpha_2 t^2 + \alpha_3 t^3 \right)^2 - 2 \left( \beta_1 t +\beta_2 t^2 + \beta_3 t^3 + \beta_4 t^4 \right) \\
            & \le 1+ \gamma_1 t + \gamma_2 t^2 + \gamma_3 t^3+ \gamma_4 t^4 + \gamma_5 t^5+ \gamma_6 t^6 
            \, ,
       \end{align*}
        where
        \begin{align}
              \gamma_1 
            & = \alpha_1-\beta_1  \label{eq:gamma_21}\, , \\
            \gamma_2 
            & =\alpha_1^2 +2 \alpha_2 -2\beta_2 \label{eq:gamma_22}\, , \\
              \gamma_3
            & = 2\alpha_1\alpha_2 +2 \alpha_3-2 \beta_3\label{eq:gamma_23} \, ,\\
               \gamma_4  
            & =\alpha_2^2 +2 \alpha_1 \alpha_3 -2 \beta_4  \label{eq:gamma_24} \, ,\\
            \gamma_5
             & = 2 \alpha_2 \alpha_3  \label{eq:gamma_25} \, , \\
             \gamma_6 
             & = \alpha_3^2 \label{eq:gamma_26}  \, . 
        \end{align}
        Next, we show $\gamma_1=0$ and get a tight bound on $\gamma_2$, which governs our  final bound to $M_2$. For the remaining terms, 
        we will show $\sum_{i=3}^6\gamma_i t^i = O(t^3 qs)$  so that they can be made negligible later when we apply the bound on $M_2$ in the proof of \prettyref{lmm:ER_fixing_k} (See the proof of Case 2 of \ER graphs). 
        
        By \prettyref{eq:alpha_21}, \prettyref{eq:beta_21} and \prettyref{eq:gamma_21}, we get 
        \begin{align*}
            \gamma_1 
            & = \alpha_1-\beta_1 =0    \, .
        \end{align*}
        Recall that $\rho = \frac{s(1-p)}{1-ps}$ and $\sigma^2 = ps(1-ps)$.  By \prettyref{eq:alpha_21}, \prettyref{eq:alpha_22}, \prettyref{eq:beta_22} and \prettyref{eq:gamma_22}, we get 
        \begin{align*}
            \gamma_2 
            & =\alpha_1^2 +2 \alpha_2 -2\beta_2 
             = q^2+q^2s^2-2q^3 -2q^3s +2q^4=  \sigma^4 \left(1+ \rho^2 \right)  \, . 
          \end{align*}
        By \prettyref{eq:alpha_21}, \prettyref{eq:alpha_22}, \prettyref{eq:alpha_23},\prettyref{eq:beta_23} and \prettyref{eq:gamma_23}, we get
        \begin{align*}
            \gamma_3
             & = 2\alpha_1\alpha_2 +2 \alpha_3-2 \beta_3 \\
            & = 2qs-12q^2s+q^2s^2+28q^3s-4q^3s^2+q^4-32q^4s +4q^4s^2-4q^5+22q^5s+6q^6-12q^6s-4q^7+4q^7s\\
            & \overset{(a)}{\le} 2qs-12q^2s+q^2s^2+28q^3s \overset{(b)}{\le}  4qs \, ,
          \end{align*}
         where $(a)$ holds because $-4q^3s^2+q^4 \le 0$, $-32q^4s  +4q^4s^2-4q^5+22q^5s+6q^6 \le 0$ and $-12q^6s-4q^7+4q^7s \le 0$; $(b)$ holds because 
         $-12q^2s+q^2s^2+28q^3s \le  -12q^2s+q^2s^2+14 q^2s  \le 4 q^2s \le 2 qs$ 
         given $q\le \frac{1}{2}$. 
         By \prettyref{eq:alpha_21}, \prettyref{eq:alpha_22}, \prettyref{eq:alpha_23}, \prettyref{eq:beta_24} and \prettyref{eq:gamma_24},  we get
          \begin{align*}
            \gamma_4  
            & =\alpha_2^2 +2 \alpha_1 \alpha_3 -2 \beta_4  \\
            & = \frac{1}{4} ( 9q^2s^2-12q^3s-56q^3s^2+4q^4+60q^4s+144q^4s^2-8q^5-146q^5s-192q^5s^2+8q^6\\
            &~~~~  +200q^6s+136q^6s^2-12q^7-136q^7s-48q^7s^2+q^8+32q^8s+16q^8s^2+16q^9-16q^9s) \\
            & \overset{(a)}{\le} \frac{1}{4} ( 9q^2s^2-12q^3s-56q^3s^2+4q^4+60q^4s+144q^4s^2 )  \overset{(b)}{\le} 13 q^2s^2  \, , 
          \end{align*}
        where $(a)$ holds because 
        \begin{align*}
            & -8q^5-146q^5s-192q^5s^2+8q^6  +200q^6s+136q^6s^2 \\
            & =\left( -8q^5+8q^6 \right)+ \left( -146q^5s-192q^5s^2 \right)+ \left(200q^6s+136q^6s^2 \right)\\
            & \le -338q^5s^2 +336q^6s \le 0 \, ,
        \end{align*} and
        \begin{align*}
            & -12q^7-136q^7s-48q^7s^2 + q^8 + 32q^8s + 16 q^8s^2 + 16q^9 \\
            & = \left(-12q^7+ q^8\right) + \left(-136q^7s-48q^7s^2 \right)+ \left(32q^8s + 16 q^8s^2 + 16q^9 \right) \\
            & \le -184q^7s^2 + 64q^8s \le 0 \, ;
        \end{align*}
        $(b)$ holds by holds because 
        $-12q^3s+4q^4 \le -8 q^3 s$, $60q^4s+144q^4s^2 \le 30q^3s+72q^3s^2 $, and then $-8 q^3 s -56q^3s^2 +30q^3s+72q^3s^2  \le 24q^3s+16 q^3s^2  \le  32 q^2s^2 $, given  $q\le \frac{1}{2}$. By \prettyref{eq:alpha_22}, \prettyref{eq:alpha_23} and \prettyref{eq:gamma_25}, we get
         \begin{align*}
            \gamma_5
             & = 2 \alpha_2 \alpha_3 \\
             & = q^2 \left(s-4qs+2q^2+4q^2s-3q^3\right)\\
             &~~~~ \left(s-6qs+15q^2s-20q^3s+15q^4s-6q^5s+q^5-2q^6+2q^6s\right) \\
             & \le q^2 \left(s+ 4qs+2q^2+4q^2s+ 3q^3\right)\\
             &~~~~ \left(s+6qs+15q^2s+20q^3s+15q^4s+6q^5s+q^5+2q^6+2q^6s\right)  \\
             & = q^2 s^2 \left(1+ 4q+2q p+4q^2+ 3q^2 p\right)\left(1+6q+15q^2+20q^3+15q^4+6q^5+q^4p+2q^5p+2q^6\right) \\
             & \le \left( 14 \cdot 68 \right) q^2 s^2  \le
             2^{10} q^2 s^2 \, .
          \end{align*}
        By \prettyref{eq:alpha_23} and \prettyref{eq:gamma_26}, we get
         \begin{align*}
             \gamma_6
            & = \alpha_3^2 = q^2 \left(s-6qs+15q^2s-20q^3s+15q^4s-6q^5s+q^5-2q^6+2q^6s\right)^2 \\
            & \le q^2 \left(s+6qs+15q^2s+20q^3s+15q^4s+6q^5s+q^5+2q^6+2q^6s\right)^2 \\
            & \le q^2s^2 \left(1+6q+15q^2+20q^3+15q^4+6q^5+q^4p+2q^5p+2q^6\right)^2 \le (68)^2 q^2s^2\le 2^{13 }q^2 s^2\, . 
        \end{align*}
        Then, we have 
        \begin{align*}
            M_2 
            & \le 1 + t^2 \sigma^4 \left(1+ \rho^2 \right) +  4 t^3qs + 13 t^4 q^2s^2 +  2^{10} t^5 q^2 s^2 + 2^{13 } t^6 q^2 s^2  \\
            & \le  1 + t^2 \sigma^4 \left(1+ \rho^2 \right) +  8 t^3qs  \, ,
        \end{align*}
        where the last inequality holds because 
        $13 t^4 +  2^{10} t^5 + 2^{12 } t^6 \le 4 t^3$ 
        given $t \le \frac{1}{2^{10}}$.

        Moreover,  given $M_1 = T$ and  $M_2 = T^2 -2D$, 
        \begin{align*}
            \frac{M_1^2}{M_2} =  1+ \frac{2D}{T^2} \le 1+ 2D \le e^2 \, , 
        \end{align*}
        where the first inequality holds because $M_2 \ge 1$ by definition, 
        and the second inequality holds because 
        \begin{align*}
            D
             & = (ac-b^2)q(s-q) \\
             & \stepa{\le} ac q(s-q) \\
             & \le  \left(1  +t  \left(1-q\right)^2  +  \frac{1}{2}  t^2  \left(1-q\right)^4+  t^3  \left(1-q\right)^6  \right)\left(1+  tq^2  +  \frac{1}{2}t^2  q^4  +  t^3  q^6  \right)q(s-q)  \stepb{\le} 3 \, ,
        \end{align*}
        where $(a)$ holds because $b^2 q(s-q) \ge 0$; $(b)$ holds because 
        $1+  t  +  \frac{1}{2}t^2  +  t^3 \le \frac{3}{2}$ 
        given $t \le \frac{1}{2^{10}}$. 
        \end{itemize}

    \section{Auxiliary results}\label{app:phix}

    
    The following lemma stated in \cite[Lemma 10]{wu2020testing} follows from the Hanson-Wright inequality \cite{hanson1971bound,rudelson2013hanson}. 
    \begin{lemma}[Hanson-Wright inequality] \label{lmm:hw}
        Let $U, V \in \reals^n$ are standard Gaussian vectors  such that the pairs 
        $(U_k, V_k)\sim \calN  \Big( \left(\begin{smallmatrix} 0\\ 0\end{smallmatrix}\right) , \left(\begin{smallmatrix} 1 & \rho \\ \rho & 1 \end{smallmatrix}\right)  \Big )
        $ are independent for $k=1,\ldots,n$.
        Let $M \in \reals^{n \times n}$ be any deterministic matrix. There exists some universal constant $c>0$
        such that 
        with probability at least $1-\delta$,
        \begin{align}
            \left| U^\top M V -  \rho \Tr(M) \right| \le c \left( \|M\|_F \sqrt{ \log (1/\delta) } +  \| M \|_2 \log (1/\delta) \right) \label{eq:hw_bilinear}. 
        \end{align}
    \end{lemma}

        \begin{lemma} \label{lmm:phix}
        The function $\phi$ in \prettyref{eq:phix}, namely
          \begin{align*}
        \phi\left(x\right) \triangleq 
          \frac{ \log \frac{1}{x} -1+x }{x \left(\log \frac{1}{x}\right) \left(1+ \log \frac{1}{x} \right) } \,  
        \end{align*} 
       is monotonically decreasing on $x \in (0,1)$. 
        \end{lemma}
    \begin{proof}
      We have
    \begin{align}
    \phi'(x) = \dfrac{\psi(x)}{x^2\left(\log\left(x\right)-1\right)^2\log^2\left(x\right)} \, , \label{eq:phi'x}
    \end{align}
    where 
    \[
    \psi(x) \triangleq \log^3\left(x\right)+\log^2\left(x\right)+\left(1-2x\right)\log\left(x\right)+x-1 \, . 
    \]
    Next, write
    \begin{align}
    \psi'(x) = \dfrac{\kappa(x)}{x} \, , \label{eq:psi'x}
    \end{align}
    where 
    \[
    \kappa(x) \triangleq 3\log^2\left(x\right)+\left(2-2x\right)\log\left(x\right)-x+1 \, .
    \]
    Note that for $x \in (0,1)$, 
    \[
     \kappa'(x) = -\dfrac{2x\log\left(x\right) - 6\log\left(x\right) +3x-2 }{x} \le 0 \,,
    \]
    where the last inequality holds because $ 2 x \log\left(x\right)  \ge - 1$ and $ - 6  \log\left(x\right) +3x \ge 3$  for $x \in (0,1)$. 
    Thus, $\kappa (x)$ is monotone decreasing on $x \in (0,1)$ and then $ \kappa(x) > \kappa (1) = 0$ for $x \in (0,1) $. 
    By \prettyref{eq:psi'x}, $\psi'(x) \ge 0$ for $x\in (0,1)$, and then $\psi(x)$ is monotone increasing  on $x\in(0,1)$ with $\psi (x) < \psi(1)= 0$. Then, by \prettyref{eq:phi'x}, 
    $\phi'(x) < 0$ on $x\in(0,1)$, as desired.
    \end{proof}
  \section*{Acknowledgement}
  The authors thank Lele Wang for pointing out an error in the earlier proof of \prettyref{eq:Y_tail_bobound}. 

\bibliography{detection}
\bibliographystyle{alpha}

\end{document}